\documentclass[11pt,a4paper]{article}
\usepackage{empheq}
\usepackage{amssymb,amsbsy,amsmath,amsfonts,amssymb,amscd,amsthm, mathrsfs, bbm}
\usepackage{amssymb}
\usepackage{color}

\usepackage{hyperref}

\newcommand{\1}{\mathbbm{1}}
\newcommand{\CC}{\mathbb{C}}
\newcommand{\NN}{\mathbb{N}}
\newcommand{\RR}{\mathbb{R}}

\newtheorem{theo}{Theorem}
\newtheorem{prop}[theo]{Proposition}
\newtheorem{lem}[theo]{Lemma}
\newtheorem{cor}[theo]{Corollary}
\newtheorem{rem}[theo]{Remark}

\newcommand{\beqn}{\begin{equation}}
\newcommand{\eeqn}{\end{equation}}
\newcommand{\bear}{\begin{eqnarray}}
\newcommand{\eear}{\end{eqnarray}}
\newcommand{\bean}{\begin{eqnarray*}}
\newcommand{\eean}{\end{eqnarray*}}

\begin{document}

\title{On  the linearized  system of equations for the condensate-normal fluid interaction near the critical temperature. }
\maketitle
\begin{center}
{\large M. Escobedo}\\
{\small Departamento de Matem\'aticas,} \\
{\small Universidad del
Pa{\'\i}s Vasco,} \\
{\small Apartado 644, E--48080 Bilbao, Spain.}\\
{\small E-mail~: {\tt miguel.escobedo@ehu.es}}
\end{center}
\noindent
{\bf Abstract}: The  Cauchy problem for the  linearization of a  system of equations arising in the kinetic theory of a  condensed  gas of bosons near the critical temperature around one of its  equilibria is solved for radially symmetric initial data.  It is proved that the linearized  system has global classical solutions that satisfy the natural conservation laws for a large set of initial data. Some regularity properties of the solutions and their long time asymptotic behavior are described.




\section{Introduction}
\setcounter{equation}{0}
\setcounter{theo}{0}

Correlations between the superfluid component and the normal fluid part in a uniform condensed Bose gas, at temperature below but close to the condensation temperature, and for a small number density of condensed atoms, may be described by the equation
\begin{align}
&\frac {\partial n} {\partial t}(t, p)=n_c(t)I_3(n(t))(p)\qquad t>0,\; p\in \RR^3, \label{PA}\\
&I_3(n)(p)=\!\!\iint _{(\RR^3)^2}\!\!\big[R(p, p_1, p_2)\!-\!R(p_1, p, p_2)\!-\!R(p_2, p_1, p) \big]dp_1dp_2. \label{E1BCD}\\
&R(p, p_1, p_2)\,= \left[\delta ( |p|^2-|p_1|^2- |p_2|^2)  \delta (p-p_1-p_2)\right]\times \nonumber \\
&\hskip 7cm  \times \left[ n_1n_2(1+n)-(1+n_1)(1+n_2)n \right],  \label{S1EA4BEJ}
\end{align}
where $n(t, p)$ represents the density of particles in the normal gas that at time $t>0$ have momentum $p$ and $n_c(t)$ is the density of the condensate at time $t$, that satisfies
\begin{equation}
\frac {dn_c} {dt} (t)=-n_c(t)\int _{ \RR^3 } I_3(n(t))(p))dp \qquad t>0.\label{PB}
\end{equation}
Equation (\ref{PA}) was first derived in \cite{Eckern} and \cite{ Kirkpatrick} and their treatment was afterwards extended to a trapped Bose gas. By including Hartree–Fock corrections to the energy of the excitations  the so called ZNG system was obtained (cf. \cite{Za}). On the interest of system (\ref{PA}),(\ref{PB}) for the description of condensed Bose gases  see also \cite{Sv, ST1, J}.
Other theoretical models  do exist to describe Bose gases in presence of a condensate (cf.\cite{PR}) but ZNG system, and  (\ref{PA}),(\ref{PB}) in particular, are very appealing by their simplicity and  are well suited for analytical PDE methods.

The two following functions of time,
\begin{equation}
\label{S1Eme}
\mathcal N(t)=\int  _{ \RR^3 }n(t, p)dp\,,\,\,\,\,\,\,\mathcal E(t)=\int  _{ \RR^3 }n(t, p)|p|^2dp
\end{equation}
give respectively the total number of particles  and the total energy  of the normal fluid part in the gas,  with density function $n(t, p)$. The total number of particles at time $t$  in the system condensate-normal fluid  is $n_c(t)+\mathcal N(t)$ and its total energy is $\mathcal E(t)$.  It formally follows from (\ref{PA}), (\ref{PB}) that these two quantities are constant in time: $\mathcal E(t)=\mathcal E(0)$  and $n_c(t)+\mathcal N(t)=n_c(0)+\mathcal N(0)$ for all $t>0$. This corresponds to the conservation of the total mass and energy property that is satisfied by the particle system in the physical description (cf. for example \cite{Eckern} ). It is also well known that  equation (\ref{PA})  has  a  family of non trivial equilibria,
\begin{align}
&n_0(p)=\left(e^{\beta |p|^2}-1\right)^{-1} \label{S1EEq21}
\end{align}
where  the mass of the particles is taken to be $m=1/2$ and $\beta $ is a positive constant related to the temperature of the gas whose particle's density is at the equilibrium $n_0$. It is easily checked that  $R(p, p_k, p_\ell)\equiv 0$ in  (\ref{S1EA4BEJ})  for $n=n_0$.

 Our purpose  is to prove the existence of classical solutions to  the Cauchy problem  for  the ``radially symmetric linearization'' of (\ref{PA})--(\ref{PB}) around an equilibrium $n_0$, and describe some of their properties. Such linearization is deduced  through the change of variables (cf. \cite{m, EPV})
\begin{align}
n(t, p)&=n_0(p)+n_0(p)(1+n_0(p))\Omega  (t, |p|)=n_0(p)+\frac {\Omega (t, |p|)} {4 \, \sinh^2 \left(\frac{\beta |p|^2}{2}\right)}.\label{S2Elinzn2}\\
&x=\frac {\sqrt {\beta }\,|p|} {\sqrt 2}=\frac {\sqrt \beta\, k } {\sqrt 2}, \quad  u(t, x)=\frac{\Omega (t, |p|)}{k^{2}},  \label{S2CcV}
\end{align}
and keeping only linear terms with respect to $u$ in (\ref{PA}). Since equation (\ref{PB}) is linear with respect to $n_c$ its linearization (\ref{S1ERRR10xp}) follows  by just keeping the  terms of 
$I_3(n_0(p)+n_0(p)(1+n_0(p))|p|^2u  (t, |p|))$ that are  linear with respect to $u$. It finally reads, for dimensionless variables in units which minimize the number of prefactors
\begin{align}
&\frac {\partial u} {\partial t }(t, x)=p_c(t)\mathcal L(u(t))\label{S3E23590}\\
&\mathcal L(u(t))=\int _0^\infty (u(t , y)-u(t , x)) M(x, y) dy\label{S3E2359Mf}\\
&\frac {d p_c} {dt}(t)=-p_c(t) \int _0^\infty \int _0^\infty  W (x, y)(u(t, y)-u(t, x))y^4x^2dydx\label{S1ERRR10xp}
\end{align}
for where, for all $x>0$, $y>0$, $x\not =y$,
\begin{align}
M(x, y)&= \left(\frac {1} {\sinh|x^2-y^2|}-\frac {1} {\sinh (x^2+y^2) } \right) \frac {y^3\sinh x^2} {x^3\sinh y^2},\label{S3E2359M}\\
W(x, y)&= \frac {M(x, y)} {(\sinh x^2)^2}\frac {x^2} {y^4} \label{S3E2359UU}\\
&= \left(\frac {1} {\sinh|x^2-y^2|}-\frac {1} {\sinh (x^2+y^2) } \right) \frac {1} {xy\sinh x^2\sinh y^2}.
\end{align} 
A different limit of the ZNG system for uniform condensed Bose gases was obtained in references \cite{Eckern} and \cite{ Kirkpatrick}, corresponding to very low temperatures and large number density of condensed atoms. Related works in the mathematical literature for the isotropic case may be found in  \cite{A, Ar, TS}. The  non isotropic  linearized system around the same equilibria $n_0$ is treated in \cite{m2}.

\subsection{The isotropic linearization of (\ref{PA}), (\ref{PB}).} The   linearization of (\ref{PA}), detailed in \cite{EPV} and recalled in \cite{m}, is briefly presented here
for the sake of completeness.
When $R(p, p_1, p_2)\!-\!R(p_1, p, p_2)\!-\!R(p_2, p_1, p)$ is written in terms of the  function $\Omega $ defined in (\ref{S2Elinzn2}) and only 
linear terms in  $\Omega $ are kept, the result is
\begin{align}
n_0(1+n_0)\frac {\partial \Omega (t)} {\partial t}=&\, n_c(t)L _{ I_3 }(\Omega (t)) \label{S1ERRR0}\\
L _{ I_3 }(\Omega (t))=&\int _0^\infty \left( \mathscr U(k, k')\Omega (t, k')-\mathscr V(k, k')\Omega (t, k)\right)k'^2dk', \label{S1ERRR}\\
\mathscr{U}(k, k') &=\frac {16 n_\mathrm{c} a^2 } {kk'}
\Big[ \theta(k-k')\times   \nonumber \\
& \times 
n_0(\omega(k))[1+n_0(\omega(k'))][1+n_0(\omega(k)-\omega(k'))] + (k \leftrightarrow k')  \Big] \nonumber \\ 
&- \, n_0(\omega(k)+\omega(k'))[1+n_0(\omega(k))][1+n_0(\omega(k'))] , \label{S1ERRRU}\\
\mathscr{V}(k, k') &= \frac{16 n_\mathrm{c} a^2}{kk'}\Big[
 \theta(k-k')\times \nonumber \\
&\times n_0(\omega(k))[1+n_0(\omega(k'))][1+n_0(\omega(k)-\omega(k'))]+  (k \leftrightarrow k')\Big]
 \label{S1ERRRV}
\end{align}
where $a$ is the s-wave scattering length,  $k=|p|$ and $k'=|p'|$. The functions  $\mathscr U(k, k')$ and $\mathscr V(k, k')$ have a non integrable singularity along the diagonal $k=k'$.  However, these singularities cancel each other when the two terms are combined as in (\ref{S1ERRR}) as far as it is assumed that, for all $t>0$, $\Omega (t)\in C^\alpha (0, \infty)$ for some $\alpha >0$.  But the integrand $\left( \mathscr U(k, k')\Omega (t, k')-\mathscr V(k, k')\Omega (t, k)\right)$ can not be split as for example in the linearization of   Boltzmann equations for classical particles. However an explicit calculation shows that, for all $k>0$,
\begin{equation}
\label{S1ERRE}
L _{ I_3 }(\omega )(k)=\int _0^\infty \left( \mathscr U(k, k') k'^2-\mathscr V(k, k')k^2\right)k'^2dk'=0
\end{equation}
from where we deduce, for all $k>0$,
\begin{align*}
\int _0^\infty &\left( \mathscr U(k, k') \frac {k'^2} {k^2}\Omega (t, k)-\mathscr V(k, k')\Omega (t, k)\right)k'^2dk'
=\frac {\Omega (t, k) } {k}L _{ I_3 }(\omega )(k)=0.
\end{align*}
We may then write,
\begin{align*}
L _{ I_3 }(\Omega (t))&=\int _0^\infty \left( \mathscr U(k, k')\Omega (t, k')-\mathscr V(k, k')\Omega (t, k)\right)k'^2dk'\\
&=\int _0^\infty  \mathscr U(k, k')\left( \frac {\Omega (t, k')} {k'^2}-  \frac {\Omega (t, k)} {k^2}\right)k'^4dk'
\end{align*}
Since equation (\ref{PB}) is linear with respect to $n_c$ its linearization (\ref{S1ERRR10xp}) follows  by just keeping the  terms of 
$I_3(n_0(p)+n_0(p)(1+n_0(p))|p|^2u  (t, |p|))$ that are  linear with respect to $u$. The linearized system reads then,
\begin{align}
&n_0(1+n_0)\frac {\partial \Omega (t, k)} {\partial t}=p_c(t)\int _0^\infty  \mathscr U(k, k')\left( \frac {\Omega (t, k')} {k'^2}-  \frac {\Omega (t, k)} {k^2}\right)k'^4dk' \label{S1ERRR01}\\
&p_c'(t)=-p_c(t)\int _0^\infty \int _0^\infty \mathscr U(k, k')\left( \frac {\Omega (t, k')} {k'^2}- \frac {\Omega (t, k)} {k^2}\right)k'^4k^2dk'dk,
 \label{S1ERRR01B}
\end{align}
or, in terms of $\tilde \Omega (t, k)=\Omega (t, k)/k^2$,
\begin{align*}
\frac {\partial \tilde \Omega(t, k)} {\partial t}=p_c(t)
\int _0^\infty \frac { \mathscr U(k, k')} {n_0(k)(1+n_0(k))k^2}\left(\tilde \Omega  (t, k')- \tilde \Omega (t, k)\right)k'^4dk' \\
p_c'(t)=-p_c(t)\int _0^\infty \int _0^\infty  \mathscr U(k, k')\left(\tilde \Omega  (t, k')- \tilde \Omega (t, k)\right)k'^4k^2dk'dk.
\end{align*}
Since $\left(k^2n_0(k)(1+n_0(k))\right)^{-1}=4k^{-2}\sinh^2\left(\frac {\beta k^2} {2}\right)$, 
\begin{align*}
\frac {\partial \tilde \Omega(t, k)} {\partial t}=4p_c(t)
\int _0^\infty  \left[\mathscr U(k, k')\sinh^2\left(\frac {\beta k^2} {2}\right) \frac {k'^4} {k^2}\right]\left(\tilde \Omega  (t, k')- \tilde \Omega (t, k)\right)dk'
\end{align*}
Use of the change of variables (\ref{S2Elinzn2}-\ref{S2CcV}) in (\ref{S1ERRR01}) yields system (\ref{S3E23590}), (\ref{S1ERRR10xp})
for  $(u, p_c)$, after scaling the time variable to get rid of some positive numerical constants. 

\subsection{A nonlinear approximation.}
Another  approximation of the system (\ref{PA}), (\ref{PB}) is possible where, in  the
equation (\ref{PB}), the function $n$ is replaced by $n_0+n_0(1+n_0)x^2 u$ in the nonlinear collision term $I_3$ given by (\ref{E1BCD}) to obtain the system
\begin{align}
&\frac {\partial v} {\partial t }(t, x)=\tilde p_c(t)\int _0^\infty (v(t , y)-v(t , x)) M(x, y) dy\label{S3E23590B}\\
&\frac {\partial \tilde p_c} {\partial t}(t)=-\tilde p_c(t)\int  _{0}^\infty I_3\left(n_0+n_0(1+n_0)x^2 v(t, x)\right)x^2dx \label{PBB}
\end{align}
instead of (\ref{S3E23590}), (\ref{S1ERRR10xp}). In that way the non linearity of $I_3$ in the equation for $\tilde p_c$ is kept. But the conservation in time of $\tilde p_c(t)+N(t)$ does not hold, and so an important global property of the original system (\ref{PA})-(\ref{PB}) is lost. As a consequence the time existence of the solutions to system (\ref{S3E23590B}), (\ref{PBB}) can not be proved to be $(0, \infty)$.
Then, system (\ref{S3E23590B}), (\ref{PBB}) is not too satisfactory to describe global properties of  the particle's system. But it may be a better approximation of the local properties of the solutions to the nonlinear system of equations (\ref{PA}), (\ref{PB}). In order to avoid any confusion, system (\ref{S3E23590B}), (\ref{PBB}) is considered in the Appendix.

\subsection{Further Motivation}
\label{motivation}
It is known  that  for all non negative measure $n _{ in }$ with a finite first  moment, and for every constant  $\rho >0$, system  (\ref{PA})-(\ref{PB})   has a weak solution  $(n(t), n_c(t))$ with initial data  $(n _{ in }, \rho )$ that satisfies the conservation of mass and energy (cf. \cite{CE}). For all $t>0$,  $n(t)$ is  a non negative measure that does not charge the origin, with finite first moment,  and $n_c(t)>0$. However, one basic aspect of the non equilibrium behavior of the system condensate--normal fluid  is the growth of the condensate after its formation (cf. \cite{Za, Bi, PR} and references therein). In the kinetic formulation  (\ref{PA})-(\ref{PB}), this behavior is driven  by the integral  of $I_3(n)$ in the right hand side of equation (\ref{PB}).  As shown in  \cite{S}, the behaviour of that term  crucially depends  on the behavior of $n(t, p)$ as $|p|\to 0$ (this was discussed also  in \cite{CE, L, Sv}). If for example the measure $n(t)$ is a radially symmetric, bounded function near the origin then,  from a simple use of Fubini's Theorem, 
\begin{equation*}
\int  _{ \RR^3 }I_3(n(t))(p)dp=C\int _0^\infty  x^3 n(t, x ) dx 
\end{equation*}
for some constant  $C>0$ independent of $n$, and this would give a monotone decreasing behavior of $n_c(t)$. On the contrary, as  it is shown in \cite{S}, if the measure $n(t)$ is a function such that 
\begin{equation}
\label{S1EASN}
n(t, p) \underset{p\to 0 }{\sim} a(t)|p|^{-2}
\end{equation} 
for some  $a(t)>0$,
and satisfies some H\"older regularity property with respect to $p$ in a neighborhood of the origin,  then for some other constant $C_1>0$ independent of $n$,
\begin{equation}
\label{S1EASN2}
\int  _{ \RR^3 }I_3(n(t))(p)dp=-C_1a^2(t)+C\int_0^\infty x^3 n(t, p )dx
\end{equation}
On the other hand, it was proved in  \cite{CE} that  if the measure  $|p|^2 n(t, p)$  has no atomic part and has an algebraic behavior as $|p|\to 0$ then it satisfies  (\ref{S1EASN}). Both results in \cite{S} and \cite{CE} assume some regularity of the solution $n$ with respect to $p$, although no regular solutions to (\ref{PA}) are known yet. The existence of regular classical solutions to (\ref{PA})-(\ref{PB}) satisfying (\ref{S1EASN}) is one of the motivations of our present work.  

Since (\ref{S1EASN}) is the behavior of the equilibrium  $n_0$ (with $a(t)\equiv \beta$), it is natural to first consider the existence of such regular solutions for the linearization of (\ref{PA}) around $n_0$. Because of the singular behavior (\ref{S1EASN}) of $n_0$ near the origin, the linear operator $\mathcal L$ in (\ref{S3E2359Mf}) has regularizing effects.  Similar regularizing effects may be expected also in the  non linear equation (\ref{PA}).

\subsection{Basic arguments and Main results.}
\label{MR}
The function $p_c(t)$ in the right hand side of (\ref{S3E23590}) may be absorbed by the  change of  variables,
\begin{align}
\label{S2NewTime}
\tau =\int _0^tp_c(s)ds,\,\,f(\tau , x)=u(t, x)
\end{align}
to obtain
\begin{align}
\label{S1ERRR10fx}
&\frac {\partial f(\tau , x)} {\partial \tau }=\mathcal L(f(\tau ))(x).
\end{align} 
This equation  may  be written,
\begin{align}
&\frac {\partial f} {\partial \tau }(\tau , x)=L(f(\tau ))(x)+F(f(\tau ))(x)\label{S3E1.1}\\
&L(f)(x)=\int _0^\infty (f (y)-f (x) )\left(\frac {1} {|x^2-y^2|}-\frac {1} {x^2+y^2} \right)\frac {y} {x}dy \label{S3E2359B}\\
&F(f(\tau ))(x)=(\mathcal L-L)(f(\tau ))(x)\label{S3E1.2}
\end{align}
where,  from (\ref{S3E2359Mf}), (\ref{S3E2359B}) and (\ref{S3E1.2}), the operator $F$ may be written,
\begin{align}
&F(f)(x)=-f(x)  \int _0^\infty T(x, y) dy+ \int _0^\infty T(x, y) f (y)dy\label{S3E1.2}\\
&T(x, y)=\frac {y^3\sinh x^2} {x^3\sinh y^2} \left(\frac {1} {\sinh|x^2-y^2|}-\frac {1} {\sinh (x^2+y^2) } \right) - \nonumber\\ 
&\hskip 5.5cm -\frac {y} {x}\left(\frac {1} {|x^2-y^2|}-\ \frac {1} {x^2+y^2} \right).\label{S3EInt2}
\end{align}
The equation (\ref{S3E1.1}) is  solved as a perturbation of
\begin{equation}
\frac {\partial f} {\partial \tau }(\tau , x)=L(f(\tau ))(x)\label{S3E1.L1}
\end{equation}
with  forcing term in  (\ref{S3E1.2}). To this end several results about equation (\ref{S3E1.L1}) obtained  in \cite{m}  are used, in particular the regularizing effects of the operator $L$ . If $\Lambda$ denotes the fundamental solution of (\ref{S3E1.L1}), for  all initial data $f_0\in L^1$  there exists a weak solution of (\ref{S3E1.L1}), 
\begin{equation}
\label{S2ESG}
S(t)f_0(x)=\int _0^\infty \Lambda\left(\frac {t} {y}, \frac {x} {y} \right)f_0(y)\frac {dy} {y},\,\,\forall t>0,\,\,\forall x>0,
\end{equation}
such that $S(\cdot )f_0\in C([0, \infty); L^1(0, \infty))$, $S(t)f_0\in C([0, \infty))$ for  all $t>0$ and,  for all $t>0$ and any compact interval $[x_0, x_1]\subset  (0, \infty)$ there exists $\alpha =\alpha (t, x_0)>0$ such that $S(t)f_0\in C^{\alpha }([x_0, x_1])$, and  (\ref{S3E1.L1}) is satisfied pointwise for almost every $t>0, x>0$. It was also shown that if $f_0\in L^1(0, \infty)\cap L^\infty _{ loc }(0, \infty)$ then  $L(u)\in L^\infty ((0, \infty)\times (0, \infty))$, $u_t\in L^\infty ((0, \infty)\times (0, \infty))$.

Once the Cauchy problem for equation (\ref{S3E1.1}) is solved using the  semigroup $S(t)$,  the change of time variable in (\ref{S2NewTime})  is inverted to obtain the function $u(t)$, and deduce $p_c(t)$ using the conservation of mass of system and equation (\ref{S1ERRR10xp}). Our first  result is then as follows. 
\begin{theo}
\label{MainThm}
Suppose that $u_0\in L^1(0, \infty)$ satisfies
\begin{equation}
|||u_0|||_\theta\equiv \sup _{ 0<x<1 }x^\theta |u_0(x)|+\sup _{ x>1 }|u_0(x)|<\infty \label{S4Etheta}
\end{equation}
for  some $\theta\in (0, 1)$. Then, there exists a pair $(u, p_c)$,
\begin{align}
&u\in C([0, \infty); L^1 (0, \infty))\cap L^\infty ((\delta , \infty); L^\infty(0, \infty)),  \forall \delta >0, \label{S4E124}\\
&p_c\in C([0, \infty)) \label{S4E124p}
\end{align}
such that for each $t>0$,  $u(t)$ is locally Lipschitz on $(0, \infty)$ and, for all almost every $t>0$ and $x>0$,
\begin{align}
&\frac {\partial u(t, x)} {\partial t}=p_c(t)\mathcal L(u(t))(x),\label{S4E125u}\\
&\frac {d p_c(t)} {dt}=-p_c(t)\int_0^\infty \mathcal L(u(t))(x) n_0(x^2)(1+n_0(x^2))x^4dx. \label{S4E125B}
\end{align}
Moreover:
\begin{align}
\frac {\partial u} {\partial t},\, \mathcal L(u) \in  L^\infty  _{ loc }((0, \infty)\times (0, \infty))\cap L^1((0, T);L^1(0, \infty)),\,\,\forall T>0,\label{S4E140}
\end{align}
there exists a function $H\in L^\infty ((\delta , 0)\times (0, \delta ))$ for all $\delta >0$, defined in (\ref{S5EH}), such that
\begin{align}
&\left|\frac {\partial u(t, x)} {\partial t} \right|+ |\mathcal L(u)(t, x)|\le C\left(\sup _{ 0<y<1 }y^{\theta}|f_0(y)| +||u_0|| _{ L^\infty(1, \infty) }+ ||u_0||_1\right)H(t, x)\label{S4E500}
\end{align}
and for all $\beta \in (0, 1)$ and $\delta (0, 1)$ there exists $ \lambda _{ \beta , \delta  } (t, x)$, defined in (\ref{S5Ell}), such that,
\begin{align}
\label{S1EPLXl}
&\left|\frac {\partial u} {\partial x}(t, x) \right|\le  \lambda _{ \beta , \delta  } (t, x)|||f_0||| _{ \theta, 1 }+\sup _{ 0\le s\le t }(|||f(s)|||_\theta+|||f(s)|||_1)
\int _0^t  \lambda _{\beta, \delta   } (s, x)ds. 
\end{align}
\begin{align}
&\forall T>0, \exists C_T>0;\,\, ||u(t)||_1\le C_T||u_0||_1,\,\,\forall t\in (0, T). \label{S4E164-1X}\\
&\forall  \delta >0,\,\,||u(t)||_\infty \le 
C(\theta) \left(||u_0|| _{ L^\infty(1, \infty) }+ \delta ^{-\theta }\sup _{ 0<y<1 }y^\theta|u_0(y)|+ ||u_0||_1\right),
\,\forall t\ge\delta. \label{S4E168X}
\end{align}
For all $\varphi \in C^1_0(0, \infty)$, the map
$
t\mapsto \displaystyle{\int _0^\infty \varphi (x)u(t, x)dx}
$
belongs to $W _{ loc }^{1,1}(0, \infty)$ and for almost every $t>0$,
\begin{equation}
\frac {d} {dt}\int _0^\infty \varphi (x)u(t, x)dx=\int_0^\infty L(u(t))(x)\varphi (x)dx+\int _0^\infty F(u(t, x))\varphi (x)dx. \label{S4cor1E2}
\end{equation}
\end{theo}
\medskip

In view of  (\ref{S2Elinzn2}), (\ref{S2CcV}) and (\ref{S2NewTime}), if $u$ is a solution of (\ref{S3E1.1}) given by Theorem \ref{MainThm},   the pair of functions
\begin{align}
\label{S2EM1}
&n(t, p)=n_0(p)+n_0(p)(1+n_0(p))|p|^2 u (t , |p|)\\
&p_c(t)=\exp\left( \int _0^t\int _0^\infty \int _0^\infty  \frac {M(x, y)} {(\sinh x^2)^2}(u(t, y)-u(t, x))\frac {x^2} {y^4}dydxds\right)\label{S2EM1B}\\
&\qquad \equiv\exp\left(\int _0^t\int _0^\infty  \mathcal L(u(t))(x) \frac {x^2dx ds} {(\sinh x^2)^2}\right)
\end{align}
may be seen as an approximated solution of (\ref{PA}), (\ref{PB}), as far as $n_0(p)(1+n_0(p))|p|^2 u (t , |p|)$ remains small compared to $n_0$. In view of (\ref{S1Eme}) it is natural to look at the quantities
\begin{align}
N(t )=\int _0^\infty n_0(x)(1+n_0(x)u(t , x)x^4dx\label{S2ENP}\\
E(t )=\int _0^\infty n_0(x)(1+n_0(x))u(t , x)x^6dx.\label{S2EE}
\end{align}
They represent respectively the variation of the total number of particles and of energy caused by the  initial perturbation $n_0(p)(1+n_0(p))|p|^2 u(0)$ of the equlibrium $n_0$. Let us also define,
\begin{align}
&N_0=\int _0^\infty n_0(x)(1+n_0(x)x^4dx\label{S2ENP}\\
&E_0=\int _0^\infty n_0(x)(1+n_0(x))x^6dx.\label{S2EE}
\end{align}
The two following properties, hold then true,
\begin{cor}
\label{SeEC}
Let $u_0$ and $u$ be as in Theorem \ref{MainThm} and  $n$ defined by (\ref{S2EM1}). Then, 
\begin{align}
&E(t )=E(0),\,\,\,\forall t >0, \label{SeEC1}\\
&p_c(t)+N(t)=p_c(0)+N(0),\,\,\forall t>0.\label{SeEC2}
\end{align} 
\end{cor}

\begin{cor}
\label{SeEC2} 
Let $u_0$ and $u$ be as in Theorem \ref{MainThm}. Then,
\begin{align}
&\lim _{ t \to  \infty } \int _0^\infty |u(t, x)-C_*|^2n_0(x)(1+n_0(x))x^6dx=0, \label{SeEC2E1}\\
&\lim _{ t \to  \infty } \int _0^\infty |u(t, x)-C_*|\,n_0(x)(1+n_0(x))x^4dx=0,\label{SeEC2E2}\\
&\hbox{where}\,\,\,\,C_*=\frac {E(0)} {E_0}.
\end{align}
\end{cor}

It follows from Corollary \ref{SeEC2} that the mass and the energy variations  due to the perturbation $n_0(1+n_0)|p|^2u (t)$ tend to the mass and energy of $
C_*n_0(1+n_0)|p|^2$, and this however small  the perturbation is at infinity, even if,  for example, $u_0$ is compactly supported.  This kind of flux of energy to infinity could be expected, since it is well known to happen in the nonlinear homogeneous version of wave turbulence type of the system (\ref{PA}), (\ref{PB}) and is called direct energy cascade (\cite{D, Zk} and \cite{EV3, TS}).

\begin{cor}
\label{SeEpc1} 
Let $u_0$ and $u$ be as in Theorem \ref{MainThm}. Then, the function $p_c\in C[0, \infty)$ is bounded on $[0, \infty)$ and
\begin{align}
&\lim _{ t\to \infty }p_c(t)=p_c(0)\exp\left(-\mathcal M_ \infty \right)\label{SeEpc11} \\
&\mathcal M _ \infty=C_*\int _0^\infty n_0(1+n_0)x^4dx - \int _0^\infty n_0(1+n_0)u_0(x)x^4dx\label{SeEpc12} 
\end{align}
\end{cor}

\subsection{Some Remarks.}

Several remarks follow from the previous results.
\subsubsection{On the formal approximation}

The approximation of (\ref{PA}), (\ref{PB}) by (\ref{S3E23590}), (\ref{S1ERRR10xp}) may be expected to be  reasonable only as long as  the perturbation  remains small with respect to $n_0$,
\begin{equation}
\label{S2Epert}
n_0(1+n_0)|\Omega (t)|<<n_0,
\end{equation}
and this requires $x^2|f(t, x)|$ small for $x\to \infty$.
However, although it could be proved that (\ref{S2Epert}) holds for small values of time  if it holds at $t=0$, it follows from (\ref{SeEC2E1})  that it can not be true for all $t>0$. Notice indeed that, for all $R\ge R_0>0$,
\begin{align}
&\left |\int _R^\infty x^6n_0(1+n_0)|u(t, x)|dx-\int _R^\infty x^6n_0(1+n_0)|C_*| \right| \le \nonumber\\
& \le
\int _R^\infty x^6n_0(1+n_0)|u(t, x)-C_*|dx\nonumber\\
&\le \left(\int _R^\infty x^6n_0(1+n_0)|u(t, x)-C_*|^2dx\right)^{1/2}
\left(\int _R^\infty x^6n_0(1+n_0)dx\right)^{1/2} \label{S2EPMN}
\end{align}
and the right hand side tends to zero as $t\to \infty$. If, on the other hand,  we had $x^2|u(t, x)|\le C$, for some $C>0$, $R_0>0$ and $t_0>0$ for all $x>R_0$ and all $t>t_0$,
\begin{align*}
\int _R^\infty x^6n_0(1+n_0)|u(t, x)|dx\le \frac {C} {R^2}\int _R^\infty x^6n_0(1+n_0)dx,\,\,\forall R>R_0, t>t_0
\end{align*}
and then, for $R>C/|C_*|$ and all $t>t_0$
\begin{align*}
&\left |\int _R^\infty x^6n_0(1+n_0)|C_*|dx-\int _R^\infty x^6n_0(1+n_0)|u(t, x)| \right| \ge\\
&\hskip 6cm \ge \left(|C_*|-\frac {C} {R^2}\right)\int _R^\infty x^6n_0(1+n_0)dx>0,\,\,\forall t>t_0.
\end{align*}
and this would contradict (\ref{S2EPMN}). System (\ref{S3E23590}), (\ref{S1ERRR10xp}) may then be considered ``close to'' (\ref{PA}), (\ref{PB}) only for small values of $t$.
Of course, $u$ could be such that,  for some $C(t)$ that tends to $\infty$ with $t$,  $x^2|u(t, x)|\le C(t)$ for all $x>0$. 

\subsubsection{The behavior of the perturbation as $|p|\to 0$.}
For all $t>0$,   the perturbation $n_0(1+n_0)\Omega (t, p)$ of $n_0$ satisfies (\ref{S1EASN}), for any $f_0$ as in the hypothesis of Theorem \ref{MainThm}, where $a(t)$ is given in Proposition \ref{S6Limzero}. The behavior $|p|^{-2}$ at the origin (that of the equilibria of (\ref{PA}), (\ref{PB})) is then instantaneously fixed, whatever the behavior at the origin of $f_0$ may be, as far as the  hypothesis of Theorem \ref{MainThm} are satisfied.
\subsubsection{The function $p_c(t)$.} In view of Corollary \ref{SeEpc1}, if the initial data $u_0$ is such that,
\begin{equation}
 \frac {E(0)}{E_0}N_0<N(0). \label{S1.44EM}
\end{equation}
or equivalently,
\begin{equation*}
\mathcal M_\infty=C_*\int _0^\infty n_0(1+n_0)x^4dx - \int _0^\infty n_0(1+n_0)u_0(x)x^4dx<0,
\end{equation*}
then $\lim _{ t\to \infty }p_c(t)>p_c(0)$, and conversely.

Condition (\ref{S1.44EM}) and its converse are both compatibles with $n_0(1+n_0)x^2u_0$ being a small perturbation of $n_0$. For example
\begin{align}
\label{S1.44EM1}
&n_0(1+n_0)x^2u_0(x)=\left(\frac {1.02} {1+x^2} -1\right)n_0,\\
&N(0)\approx-0.344949,\,\,\,E(0)\approx -0.523546 \nonumber\\
&\frac {N(0)} {E(0)}\approx 0.658872 < \frac {N_0} {E_0}\approx 0.778949
 \Longleftrightarrow  \frac{E(0)}{E_0}N_0<N(0)\nonumber
 \end{align}
 and
 \begin{align}
 \label{S1.44EM1}
 &n_0(1+n_0)x^2u_0(x)=0.2 \hbox{Arctg}\left(\frac{x}{10}\right)n_0,\\
 &N(0)\approx 0.0163705,\,\,\,E(0)\approx  0.0238295 \nonumber\\
 &\frac {E(0)} {N(0)}\approx 1.45564 > \frac {E_0} {N_0}\approx 1.28378
 \Longleftrightarrow  \frac{E(0)}{E_0}N_0>N(0).\nonumber
 \end{align}
 
 \subsection{Very low temperature and large $n_c$.}
Linearization of system (\ref{PA}), (\ref{PB}) for large number density of condensed atoms and very low temperature may be performed following similar arguments as above (cf. \cite{Bu} and \cite{m2}). No regularizing effects have been observed and 
the existence of a first positive eigenvalue and spectral gap for a suitable integrable operator (\cite{Bu} and \cite {ET}) provide a convergence rate to the equilibrium for a large set of initial data (cf.\cite{m2},  Theorem 2.2)
A necessary and sufficient condition on $p_c(0)$ to have a global solution.

\section{The operator $F$.}
\setcounter{equation}{0}
\setcounter{theo}{0}
Equation (\ref{S3E1.1}) may be treated as a perturbation of (\ref{S3E1.L1}) only  whenever the term  $F(f)$ in (\ref{S3E1.2}) is  bounded in   spaces where the properties of the solutions of (\ref{S3E1.L1})  may be used. The  purpose of  this Section is to establish that this is the case.
\begin{prop}
\label{S2AFF1}
(i) For all $g\in L^\infty(0, \infty)$, $F(g)\in L^\infty(0, \infty)$ and
$$
||F(g)|| _{ \infty }\le 2M||g|| _{ \infty }.
$$
(ii) For all $g\in L^1(0, \infty)$,
$$
||F(g)||_1\le C_F ||g||_1\,\,\, C_F=(M+\widetilde M),
$$
where
$$
M=\sup _{ x>0 }\int _0^\infty |T(x, x')|dx',\,\,\,\widetilde M=\sup _{ x'>0 }\int _0^\infty |T(x, x')|dx.
$$
(iii) For all $\theta\in [0, 1)$ there exists a positive constant $C(\theta)$ depending on $\theta$,  such that, if $g\in L^\infty _{ loc }(0, \infty)$ satisfies $ |||g|||_\theta<\infty$ then $|||F(g)|||_\theta\le C(\theta)|||g|||_\theta.$
\vskip 0.2cm
\noindent
(iv) For all $g\in L^1(0, \infty)\cap L^\infty _{ loc }(0, \infty)$, $F(g)\in L^1(0, \infty)\cap L^\infty _{ loc }(0, \infty)$.
\vskip 0.2cm
\noindent
(v) For all  $R>0$ there exists a constant $C=C(R)>0$ such that:
\begin{align*}
||F(g)|| _{ \infty }\le C\left(||g|| _{ L^1(0, R) }+||g|| _{ L^\infty(R, \infty) }\right),\,\,\forall g\in &L^1(0, \infty)\cap L^\infty (R, \infty).
\end{align*}
\end{prop} 
Proposition \ref{S2AFF1} follows from estimates of the kernel $T$ defined in (\ref{S3EInt2}), that we split as follows,
\begin{align*}
&T(x, x')=T_1(x, x')+T_2(x, x')\\
&T_1(x, x')=\frac {x'} {x}\left(\frac {1} {\sinh|x^2-x'^2|}-\frac {1} {|x^2-x'^2|}-\frac {1} {\sinh (x^2+x')^2} +\frac {1} {|x^2-x'^2|}\right)\\
&T_2(x, x')=\frac {x'^3} {x^3}\left(\frac {\sinh x^2} {\sinh x'^2} -\frac {x^2} {x'^2}\right)\left(\frac {1} {\sinh|x^2-x'^2|}-\frac {1} {\sinh (x^2+x')^2 } \right).
\end{align*}
The kernels $T_1$ and $T_2$ are estimated in the two next Propositions.
\begin{prop}
\label{S2RK1}
\begin{align}
&\forall R>0, \exists C_R>0;\,\,| T_1(x, x')|\le C_Rxx',\,\,\,\forall x'\in (0, R),\,\forall x\in (0, R)\label{S2RK1E1Sm}\\
&| T_1(x, x')|\le \frac {Cx'} {x }\left( \min(x^2 , x'^2 )+\mathcal O(x^2 +x '^2)^3\right),0\le x\le 1/2,\,\,0\le x'\le 1/2.\label{S2RK1E1}\\
& |T_1(x, x')|\le C, \hbox{if}\,\,\, x+x'>1, |x'-x|\le 1/8,\label{S2RK1E2}\\
&\forall x\in (0, 1),\,x'>\min (2, 3x/2),\,\,\,| T_1(x, x')|\le \frac {Cx} {x'^3}\label{S2RK1E3}
\end{align}
\begin{align}
&\hbox{If}\,\, x+x'>1, \,\,1/8<|x-x'|<x/2\nonumber\\
&| T_1(x, x')|\le \frac {Cx'} {x}\left(\frac{2\min(x,x')^2} {(x^2-x'^2)(x^2+x'^2)}+\left|\frac {1} {\sinh |x^2-x'^2|}-\frac {1} {\sinh (x^2 +x'^2)}\right|\right),\label{S2RK1E4}
\end{align}
\begin{align}
&\forall |x-x'|>x/2,\nonumber \\
&|T_1(x, x')|\le \frac {Cx'} {x }\left(\frac {e^{-C_1\max(x,x')^2}} {\left( 1-e^{-\beta (x^2+x'^2)}\right)\left(1-e^{-C_1\max(x', x)^2} \right)}+\frac {\min(x',x)^2} {\max(x', x)^4}\right) \label{S2RK1E5}\\
&|T_1(x, x')|\le \frac{Cx'} {x}\left(\frac {1} {\sinh \frac {3x^2} {4}}+\frac {\min(x',x)^2} {\max(x', x)^4}\right) 
\,\,\,\hbox{if}\,\,\,0<x'<\frac {x} {2}\label{S2RK1E5ter}
\end{align}
\end{prop}
\begin{proof} \hfill \break
0.- Proof of (\ref{S2RK1E1Sm}).
When $x'\in (0, R)$ and $x\in (0, R)$. We may use the series expansion of the function $1/\sinh x$ to obtain,
\begin{align*}
&h(z)=\frac {1} {z}-\frac {1} {\sinh z}=\sum _{ n=1 }^\infty \frac {2(2^{2n-1}-1)B_n} {(2n)!}z^{2n-1}\\
&B_0=1,\, B_n=\sum _{ \ell =1 }^{n-1}\binom{n}{\ell}\frac {B_\ell} {n+1-\ell},\,\,n\ge 1,\,\,(\hbox{Bernoulli's numbers}),\\
&h'(z)=\sum _{ n=1 }^\infty (2n-1)\frac {2(2^{2n-1}-1)B_n} {(2n)!}z^{2(n-1)}.
\end{align*}
For $x'\in (0, R)$ and $x\in (0, R)$:
\begin{align*}
&h(x^2+x'^2)-h(|x^2-x'^2|)=2x^2h'(\xi ),\,\,\,\xi \in (x'^2-x^2, x'^2+x^2)\\
&|h(x^2+x'^2)-h(|x^2-x'^2|)|\le 2x^2\sup _{ \xi \in (0, 2R^2) } |h'(\xi )|
\end{align*}
and
$$
|T_1(x, x')|\le 2xx'\sup _{ \xi \in (0, 2R^2) } |h'(\xi )|,\,\,\forall x'\in (0, R), x\in (0, R).
$$

1.- Proof of (\ref{S2RK1E1}).
Consider first the set where $x' <1/2$ and $x <1/2$, and use the Taylor's expansion of $1/\sinh z $ around $z =0$,
\begin{align*}
 \frac {1} {\sinh  |x'^2-x^2 | }=\frac {2} {\beta |x^2-x'^2|}-\frac {\beta |x^2-x'^2|} {12}+\mathcal O(|x^2-x'^2|)^3,\,\,|x-x'|\to 0\\
  \frac {1} { \sinh  (x'^2+x^2 )}=\frac {2} {\beta (x^2+x^2)}-\frac {\beta (x^2+x'^2)} {12}+\mathcal O((x^2+x'^2))^3,\,\,|x+x'|\to 0.
\end{align*}
Then 
\begin{align*}
\left( \frac {1} {\sinh  |x'^2-x^2 | }
-\frac {1} {\beta |x^2-x'^2|}-\frac {1} {\sinh  (x'^2+x^2 ) }
+\frac {1} {\beta (x^2+x'^2)}\right)=\\
=\frac {\min(x^2 ,x'^2)} {3}+\mathcal O(x^2 +x'^2)^3,
\end{align*}
\begin{align*}
|T_1(x, x')|\le \frac {Cx'} {x }\left( \min(x^2 , x '^2)+\mathcal O(x^2 +x'^2)^3\right)
\end{align*}
and that proves (\ref{S2RK1E1}).

2.- Proof of (\ref{S2RK1E3}). 
When $x\in (0, 1)$ and $x'>x$, in the identity,
\begin{align*}
\frac {1} {\sinh |x^2-x'^2|} - \frac {1} {\sinh (x^2 +x'^2)}=\frac {\sin (x'^2+x^2 )-\sinh (x'^2-x^2 )} 
{\sin (x'^2+x^2 )\sinh (x'^2-x^2 )}.
\end{align*}
we use the mean value Theorem to obtain
\begin{align*}
&\sin (x'^2+x^2 )-\sinh (x'^2-x^2 )=-2x^2 \cosh \xi \\
&\xi \equiv \xi (x'^2-x^2 , x'^2+x^2 )\in (x'^2-x^2, x'^2+x^2)\subset (x'^2-1, x'^2+1)
\end{align*}
and then,
$$
\cosh \xi \le C\cosh (x'^2).
$$
We deduce,
\begin{align*}
\left|\frac {1} {\sinh |x^2-x'^2|} - \frac {1} {\sinh (x^2+x'^2)}\right|\le C\frac {x^2\cosh x'^2} 
{\sin (x'^2+x^2)\sinh (x'^2-x^2)}
\le C\frac {x^2} {\sinh x'^2}.
\end{align*}
On the other hand, since $x'>\min(2, 3x/2)$, $x'^2-x^2>Cx'^2$ and,
\begin{align*}
\frac {1} {|x^2-x'^2|}-\frac {1} {|x^2+x'^2|}=\frac{2x^2} {(x'^2-x^2)(x^2+x'^2)}\le C\frac {x^2} {x'^4}
\end{align*}
it follows that,
$
|T_1(x, x')|\le C\frac {x^2} {x'^4}\frac {x'} {x}\le C\frac {x} {x'^3}
$
and that proves (\ref{S2RK1E3}).

3.- Proof of (\ref{S2RK1E2})
Suppose now that $x+x'>1$, and $|x-x'|\le 1/8$. We may still use the Taylor's expansion of $1/\sinh (|x^2-x'^2|) $ around $|x^2-x'^2|=0$,
\begin{align*}
\frac {1} {\sinh |x^2-x'^2|}-\frac {1} {|x^2-x'^2|}=-\frac { |x^2-x'^2|} {6}+\mathcal O(|x^2-x'^2|)^3,\,\,|x-x'|\to 0.
\end{align*}
Then, if $x+x'>1$,  $|x'-x|\le 1/8$,
\begin{align*}
|T_1(x, x')|&\le \frac {Cx'} {x}\left(|x^2-x'^2|-\frac {1} {\sinh |x'^2+x^2| }
+\frac {1} { |x^2+x'^2|}\right)\\
&\le \frac {Cx'} {x}\left(1+\frac {1} { (x^2+x'^2)}\right)\le \frac {Cx'} {x}
\end{align*}
and that proves (\ref{S2RK1E2}).

4.- Proof of (\ref{S2RK1E4}). Consider  now, the cases where $x+x'>1$, $|x'-x|>1/8$ and $|x-x'|<x/2$. Since,
\begin{align*}
\frac {1} {|x^2-x'^2|}-\frac {1} {|x^2+x'^2|}=\frac{2\min(x, x')^2} {(x'^2-x^2)(x^2+x'^2)}\le \frac{16\min(x, x')^2} {(x'+x)(x^2+x'^2)}
\end{align*}
If  $x+x'>1$, $|x'-x|>1/8$ and $|x-x'|<x/2$,
\begin{align}
|T_1(x, x')|&\le  \frac {Cx'} {k}\left(\frac{2\min(x, x')^2} {(x^2-x'^2)(x^2+x'^2)}+\left|\frac {1} {\sinh |x^2-x'^2|}-\frac {1} {\sinh (x^2+x'^2)}\right|\right) \label{S2EXpr1}\\
&\le \frac {Cx'} {x}\left(\frac{2\min(x, x')^2} {(x'+x)(x^2+x'^2)}+\frac {1} {\sinh\frac { |x+x '|} {8}}+\frac {1} {\sinh (x^2+x'^2)}\right)\nonumber \\
&\le \frac {Cx'} {x}\left(\frac{2\min(x, x')^2} {(x'+x)(x^2+x'^2)}+\frac {1} {\sinh\frac { |x+x '|} {8}}\right) \label{S2EXpr}
\end{align}
The estimate (\ref{S2EXpr1}) is nothing but  (\ref{S2RK1E4}). 

5.- Proof of (\ref{S2RK1E5})
On the other hand, if $x>1$ and  $|x'-x|>x/2$ we may still use hat,
\begin{align*}
\frac {1} {|x^2-x'^2|}-\frac {1} {|x^2+x'^2|}=\frac{2\min(x, x')^2} {(x'^2-x^2)(x^2+x'^2)} &\le 
\begin{cases}
\frac {Cx^2} {x'^4},\,\,\forall x'>3x/2>0\\
\frac {Cx'^2} {x^4},\,\,\,\forall x'<x/2.
\end{cases}
\end{align*}
If $y>x$,
\begin{align*}
\left( \frac {1} {\sinh  |y^2-x^2| }- \frac {1} {\sinh  |y^2+x^2| }\right)=
\frac {e^{x^2+y^2}\left(1-e^{-2(x^2+y^2)}-e^{-2x^2}+e^{-2y^2}\right)} {e^{x^{2}+y^2}\left( 1-e^{-2(x^2+y^2)}\right)\left( e^{y^2-x^2}-e^{-y^2+x^2}\right)}\\
=\frac {1-e^{-2(x^2+y^2)}-e^{-2x^2}+e^{-2y^2}} {\left( 1-e^{-2(x^2+y^2)}\right)e^{y^2-x^2}\left(1-e^{-2(y^2-x^2)} \right)}
\le\frac {e^{-(y^2-x^2)}} {\left( 1-e^{-2(x^2+y^2)}\right)\left(1-e^{-2(y^2-x^2)} \right)}
\end{align*}
Then, if $x'>3x/2$,
\begin{align*}
\left( \frac {1} {\sinh  |x'^2-x^2| }- \frac {1} {\sinh (x'^2+x^2)}\right)\le 
\frac {e^{-(x'^2-x^2)}} {\left( 1-e^{-2(x^2+x'^2)}\right)\left(1-e^{-2 (x'^2-x^2)} \right)}\\
\le \frac {e^{-\frac {5 } {9}x'^2}} {\left( 1-e^{-2 (x^2+x'^2)}\right)\left(1-e^{-\frac {5 } {9}x'^2} \right)}
\end{align*}
If $x>2x'$,
\begin{align*}
\left( \frac {1} {\sinh|x'^2-x^2|}- \frac {1} {\sinh  |x'^2+x^2| }\right)\le
\frac {e^{-(x^2-x'^2)}} {\left( 1-e^{-2 (x^2+x'^2)}\right)\left(1-e^{-2 (x^2-x'^2)} \right)}\\
\le \frac {e^{-\frac {3} {4}x^2}} {\left( 1-e^{-2 (x^2+x'^2)}\right)\left(1-e^{-\frac {3 } {4}x^2}\right)}
\end{align*}
and that proves (\ref{S2RK1E5}). 
But also, in the case $x>2x'$, $x^2-x'^2>3x^2/4$ and then,
\begin{align*}
\frac {1} {\sinh |x'^2-x^2| }\le  \frac {1} {\sinh \frac {3  x^2 } {4}};\,\,\,\,
 \frac {1} {\sinh  (x'^2+x^2)}\le  \frac {1} {\sinh x^2}
\end{align*}
and that proves (\ref{S2RK1E5ter}).
\end{proof}

\begin{prop} 
\label{S2RK2}
\begin{align}
&\forall R>0,\exists C_R>0;\,\,\, |T_2(x, x')|\le C_R,\,\,\forall x\in (0, R),\,\,\forall x'\in (0,R)\label{S2RK2E1}\\
&\forall \delta >0,\,\,\exists C_\delta >0;\,\,\,\forall x\in (0, 1),\,\forall x'>x+\delta:\,\,\,\,\,\, |T_2(x, x')|\le C_\delta \frac {x x'} {\sinh x'^2} \label{S2RK2E2}\\
&\forall \,\,\,x+x'>1, |x-x'|<1/8:\,\,|T_2(x, x')|\le C \label{S2RK2E3Bis}\\
&\forall \,\,\,x+x'>1,\frac {1} {8}< |x-x'|<x/2: \nonumber \\
&|T_2x, x')|\le C\left| \frac {1} {\sinh|x^2 -x'^2|}-\frac {1} {\sinh (x^2 +x'^2)}\right|,
\label{S2RK2E3}
\end{align}
If $x+x'>1$ and $|x-x'|>x/2$,
\begin{align}
&|T_2(x, x')|\le  \frac {C} {\sinh \frac { 3 x^2 } {4} },\,\,x'<x/2 \label{S2RK2E4}\\
&|T_2(x, x')|\le  Ce^{-(x'^2 -x^2 )}\left( e^{-(x'^2 -x^2)}+\frac {x^2 } {x'^2} \right)\frac {x'^3} {x^3},\,\,\,x'\ge 3x/2
\label{S2RK2E5}
\end{align}

\end{prop}
\begin{proof} \hfill \break

1.- Proof of (\ref{S2RK2E1}). Let us write first,
\begin{align*}
&\left|\frac {\sinh x^2 } {\sinh  x'^2}-\frac {x^2} {x'^2}\right|=
\left|\frac {x'^2\sinh x- x^2 \sinh x'^2} {x'^2\sinh x'^2}\right|\\
&=\left| \frac {x'^2\sinh x^2-x^2\sinh x^2+x^2\sinh  x^2- x^2 \sinh  x'^2} {x'^2\sinh  x'^2}\right|
\end{align*}
and
\begin{align*}
x'^2\sinh x^2 -x^2 \sinh x'^2=x'^2(\sinh x^2 -\sinh x'^2)+(x'^2-x^2 )\sinh x'^2.
\end{align*}
Using Taylor's expansion we have, for some $\xi (x^2 , x'^2)$ between  $x^2 $ and $x'^2$,
\begin{align*}
x'^2&(\sinh x^2 -\sinh x'^2)=\\
&=x'^2\left((x^2 -x'^2)\cosh x'^2+\frac {1} {2}(x^2 -x'^2)^2+\frac {1} {6}(x^2 -x'^2)^3\cosh \xi (x^2, x'^2)\right)
\end{align*}
and,
\begin{align*}
&x'^2(\sinh x^2 -\sinh x'^2)+(x'^2-x^2 )\sinh x'^2=\\
&=x'^2\left((x^2 -x'^2)\cosh x'^2+\frac {1} {2}(x^2 -x'^2)^2+\frac {1} {6}(x^2 -x'^2)^3\cosh\xi (x^2 , x)^2\right)+\\
&\hskip 7cm +(x'^2-x^2 )\sinh x'^2\\
&=(x^2 -x'^2)\left(x'^2 \cosh x'^2 -\sinh x'^2\right)+\frac {x'^2} {2}(x^2 -x'^2)^2+\\
&\hskip 7cm +\frac {x'^2} {6}(x^2 -x'^2)^3\cosh \xi (x^2 , x'^2).
\end{align*}
We deduce that for all $x\in (0, R)$ and $x'\in (0, R)$,
\begin{align*}
&|x'^2(\sinh x^2 -\sinh x'^2)+(x'^2-x^2 )\sinh x'^2|
\le C_R|x^2 -x'^2\times \\
&\times |( x'^6+x'^2|x^2 -x'^2|+|x^2 -x'^2|^2x'^2) \le C'_R|x^2 -x'^2|\left( x'^6+x'^2x^2\right).
\end{align*}
Similarly, using the change $x\leftrightarrow x'$,
\begin{align*}
|x'^2(\sinh x^2 -\sinh x'^2)+(x'^2-x^2 )\sinh x'^2|\le C_R|x^2 -x'^2|\left( x^6+x'^2x^2\right).
\end{align*}
It follows that, for all $x\in (0, R)$, $x'\in (0, R)$,
\begin{align*}
|x'^2(\sinh x^2 -\sinh x'^2)+(x'^2-x^2 )\sinh x'^2|\le C_R|x^2 -x'^2|\times \\
\times \min\{x'^2(x^2+x'^4), x^2(x^4+x'^2)\}\le  C_R|x^2 -x'^2| x^2x'^2
\end{align*}
and
\begin{align*}
&\left|\frac {\sinh x^2 } {\sinh  x'^2}-\frac {x^2} {x'^2}\right|
\le C_R\frac {|x^2 -x'^2| x^2\,  x'^2 } {x'^4}.
\end{align*}
On the other hand, since $\sinh $ is locally Lipschitz,
\begin{align*}
\left|\frac {1} {\sinh |x'^2-x^2|}- \frac {1} {\sinh (x^2+x'^2)}\right|\le
\frac {C_R\min(x^2, x'^2)} {|x'^2-x^2 |(x'^2+x^2 )}
\end{align*}
for some constant $C_R>0$. We deduce, 
\begin{align*}
\left|\frac {\sinh x^2 } {\sinh  x'^2}-\frac {x^2} {x'^2}\right|
\left|\frac {1} {\sinh |x'^2-x^2|}- \frac {1} {\sinh (x^2+x'^2)}\right|
&\le C_R\frac {x^2 x'^2\min(x^2, x'^2) } {x'^4(x'^2+x^2 )}\\
&\le C_R\frac {x^2\min(x^2, x'^2) } {x'^2(x^2 +x'^2)}
\end{align*}
and 
\begin{align*}
|T_2(x, x')|\le C_R\frac {x' \min(x^2, x'^2)} {x(x^2+x'^2)}\le C_R,\,\,\forall x\in (0, R),\, x'\in (0, R).
\end{align*}
This proves (\ref{S2RK2E1}).

2.- Proof of (\ref{S2RK2E2}). 
Suppose now that $x\in (0, 1)$ and $x'>x+\delta $ for any $\delta >0$. Then, as we have ween in the proof of Proposition \ref{S2RK1},
\begin{align*}
\left|\frac {1} {\sinh |x^2-x'^2|} - \frac {1} {\sinh (x^2 +x'^2)}\right|\le C\frac {x^2 \cosh x'^2} 
{\sin (x'^2+x^2 )\sinh (x'^2-x^2 )}
\le C_\delta \frac {x^2 } {\sinh x'^2}.
\end{align*}
We have also, since $\sinh x^2 \le \sinh 1$,
\begin{align*}
\left|\frac {\sinh x^2 } {\sinh x'^2}-\frac {x^2} {x'^2}\right|
&=\left|\frac {x'^2\sinh x^2 -x^2 \sinh x'^2} {x'^2\sinh x'^2}\right|\\
&\le C\frac {x^2 (x'^2+\sinh x'^2)} {x'^2\sinh x'^2}=Cx^2 \left(\frac {1} {\sinh x'^2}+\frac {1} {x'^2}\right)\le \frac {Cx^2 } {x'^2}.
\end{align*}
Then,
\begin{align*}
|T_2(x, x')|\le C_\delta \frac {x^2 } {\sinh x'^2}\frac {x^2 } {x'^2}\frac {x'^3} {x^3}=C_\delta \frac {x\ x'} {\sinh x'^2}
\end{align*}
and this proves (\ref{S2RK2E2}).

3.- Proof of  (\ref {S2RK2E3Bis})
We consider now the case where $x+x'>1$ and $|x-x'|<1/8$. Suppose again that $0<x'^2<x $. We may still use the Taylor's expansion around $x^2 =x'^2$, and write,
\begin{align*}
|x'^2\sinh x^2 -x^2 \sinh x'^2|&\le x'^2|\sinh x^2 -\sinh x'^2|+|x^2 -x'^2 |\sinh x'^2\\
&\le x'^2 |x'^2-x^2 |\cosh \xi (x^2 , x'^2)+|x^2 -x'^2 |\sinh x'^2
\end{align*}
and
\begin{align*}
&\left|\frac {\sinh x^2} {\sinh x'^2}-\frac {x^2} {x'^2}\right|\le
\frac {x'^2 |x'^2-x^2 |\cosh \xi (x^2 , x'^2)+|x^2 -x'^2 |\sinh x'^2} {x'^2\sinh x'^2}\\
&=|x^2 -x'^2|\left( \frac {\cosh \xi  (x^2 , x'^2)} {\sinh x'^2}+\frac {1} {x'^2}\right)
\end{align*}
If $|x-x'|<1/8$ and $x+x'>1$ then, for some $x_0>0$,  $x'\ge x_0$ and $x\ge x_0$. Then, there exists $C>0$ such that
$$\frac {\cosh \xi  (x^2 , x'^2)} {\sinh x'^2}\le C,\,\,\forall x,\,\,\forall x; x+x'>1, |x-x'|<1/8$$
and,
\begin{align*}
&\left|\frac {\sinh x^2} {\sinh x'^2}-\frac {x^2} {x'^2}\right|\le
C|x^2 -x'^2|\,\,\forall x,\,\,\forall x'; x+x'>1, |x-x'|<1/8.
\end{align*}
On the other hand, if $|x+x'|>1$ and $|x-x'|<1/8$,
\begin{align*}
\left|\frac {1} {\sinh |x^2-x'^2|}- \frac {1} {\sinh (x^2+x'^2)}\right|\le
\frac {C} {|x'^2-x^2 |}+\frac {1} {\sinh  (x^2+x'^2)}
\end{align*}
for some positive $C$. It follows,
\begin{align*}
&\left|\frac {\sinh x^2} {\sinh x'^2}-\frac {x^2} {x'^2}\right|
\left|\frac {1} {\sinh |x^2-x'^2|}- \frac {1} {\sinh (x^2+x'^2)}\right|\le
C; \,\, k+k'>1, |k-k'|<1/8.
\end{align*}
and,
\begin{align*}
|T_2(x, x')|\le C \frac {x'^3} {x^3}\le C\frac {(x+1/8)^3} {x^3}\le C',\,\,x+x'>1, |x-x'|<1/8.
\end{align*}
and this shows (\ref {S2RK2E3Bis}).

4.- Proof of (\ref {S2RK2E3})
Suppose now that $x+x'>1$, $|x-x'|>1/8$ and $|x-x'|\le x/2$. Then
\begin{align*}
&\left|\frac {\sinh x^2 } {\sinh x'^2}-\frac {x^2} {x'^2}\right|\le C
\end{align*}
and,
\begin{align*}
\left|\frac {\sinh x^2 } {\sinh x'^2}-\frac {x^2} {x'^2}\right| 
\left|\frac {1} {\sinh |x^2-x'^2|}- \frac {1} {\sinh (x^2+x'^2)}\right|\le
C\left|\frac {1} {\sinh |x^2-x'^2|}- \frac {1} {\sinh (x^2+x'^2)}\right|
\end{align*}
and this shows (\ref {S2RK2E3}).

5.- Proof of (\ref {S2RK2E4}) and  (\ref {S2RK2E5}) .
Suppose now $x+x'>1$, $|x-x'|>x/2$. As we have seen in the proof of  Proposition \ref{S2RK1},
If $x'>x$,
\begin{align*}
\left|\frac {1} {\sinh |x^2-x'^2|}- \frac {1} {\sinh (x^2+x'^2)}\right|\le\frac {e^{-(x'^2-x^2)}} {\left( 1-e^{-2(x^2+x'^2)}\right)\left(1-e^{-2(x'^2-x^2)} \right)}\le Ce^{-(x'^2-x^2 )}
\end{align*}
and
\begin{align*}
\left|\frac {\sinh x^2 } {\sinh x'^2}-\frac {x^2} {x'^2}\right| \le
Ce^{-(x'^2 -x^2)}+\frac {x^2 } {x'^2}.
\end{align*}
Then,
\begin{align*}
\left|\frac {\sinh x^2 } {\sinh x'^2}-\frac {x^2} {x'^2}\right| 
\left|\frac {1} {\sinh |x^2-x'^2|}- \frac {1} {\sinh (x^2+x'^2)}\right|\le Ce^{-(x'^2-x^2 )}\left( e^{-(x'^2 -x^2)}+\frac {x^2 } {x'^2} \right).
\end{align*}
If on the other hand, $x'<x/2$, $x^2-x'^2>3x^2/4$, and then
\begin{align*}
\left|\frac {1} {\sinh |x^2-x'^2|}- \frac {1} {\sinh (x^2+x'^2)}\right|\le
\frac {1} {\sinh \frac {3 x^2 } {4}}+\frac {1} {\sinh x^2}\le \frac {2} {\sinh \frac {3x^2  } {4} }
\end{align*}
We deduce,
\begin{align*}
&|T_2(x, x')|\le \frac {C} {\sinh \frac {3 x^2 } {4} }\frac {x'^3} {x^3}\le \frac {C} {\sinh \frac { 3x^2} {4} },\,\,x'<x/2\\
&|T_2(x, x')|\le  Ce^{-(x'^2-x^2 )}\left( e^{-(x'^2 -x^2)}+\frac {x^2 } {x'^2} \right)\frac {x'^3} {x^3},\,\,\,x'\ge 3x/2
\end{align*}
and this proves (\ref {S2RK2E4}) and  (\ref {S2RK2E5}).
\end{proof}

Two Corollaries  follow from Proposition (\ref{S2RK1}) and Proposition (\ref{S2RK2}). First,
\begin{cor}
\label{S4PM}
$$
M=\sup  _{ x>0 }m(x)=\sup _{ x>0 }\int _0^\infty |T(x, x')|dx'<\infty
$$
\end{cor}
\begin{proof}
By definition, for all $x>0$,
\begin{align*}
m(x)=m_1(x)+m_2(x);\,\,\,
|m_1(x)|\le \int _0^\infty  |T_1(x, x')|dx';\,\,\,
|m_2(x)|\le \int _0^\infty  |T_2(x, x')|dx'.
\end{align*}
Suppose first that $x\in (0, 1)$. Then, by  (\ref{S2RK1E1}) and (\ref{S2RK1E3})
\begin{align}
m_1(x)&\le \int _0^2 |T_1(x, x')|dx'+\int _2^\infty  |T_1(x, x')|dx'\nonumber \\
&\le Cx\int _0^1x'dx'+x \int_1^\infty \frac {dx'} {x'^3}=Cx,\,\forall x\in (0, 1).
\label{S2.4ME0}
\end{align}
With a similar argument, using instead (\ref{S2RK2E1}) and (\ref{S2RK2E2}), 
\begin{align}
m_2(x)&\le \int _0^2 |T_2(x, x')|dx'+\int _2^\infty |T_2(x, x')|dx'\nonumber \\
&\le C+Cx \int_2^\infty \frac {x'dx'} {\sinh x'^2}\le C,\,\forall x\in (0, 1).
\label{S2.4ME0Bis}
\end{align}
Suppose now that $x>1$ and write, for $\ell=1, 2$,
\begin{align}
m_\ell(x)=\int _0^{x/2} dx'+\int  _{ x/2 }^{x-\frac {1} {8}}dx'+\int  _{ x-\frac {1} {8} }^{x+\frac {1} {8}}dx'+\int_{x+\frac {1} {8}}^{3x/2}dx'+\int  _{ 3x/2 }^\infty dx'\label{S2.4ME5}
\end{align}
When $x'\in (0, x/2)$ we may use (\ref{S2RK1E5ter}) to obtain,
\begin{align}
\int _0^{\frac {x} {2}}|T_1(x, x')| &\le \frac {C} {x} \int _0^{\frac {x} {2}}x'\left(\frac {1} {\sinh \frac {3x^2} {4}}+\frac {x'^2} { x^4}\right)\le C\left(\frac {x} {\sinh \frac {3x^2} {4}}+\frac {1} {x}\right)\le \frac {C} {x},\,\,\,\,\forall x>1.\label{S2.4ME7}
\end{align}
If we use (\ref{S2RK2E4}) in the same region we obtain, 
\begin{align}
&|T_2(x, x')|\le \frac {C} {\sinh \frac { 3 x^2 } {4} }\nonumber\\
\hskip -4.5cm \hbox{and then,}\hskip 3cm &
\int _0^{\frac {x} {2}}|T_2(x, x')| \le \frac {Cx} {\sinh \frac { 3x^2  } {4} },\,\,\,\,\forall x>1.\label{S2.4ME7Bis}
\end{align}

When $x'\in (x/2, x-1/8)$, by (\ref{S2EXpr}),
\begin{align}
|T_1(x, x')|\le &
\frac {Cx'} {x}\left(\frac{2\min(x, x')^2} {(x'+x)(x^2+x'^2)}+\frac {1} {\sinh\frac { |x+x '|} {8}}\right)\le  C\frac {x'} {x^2}\label{S2Eqpo1} \\
\hskip -4cm \hbox{and then,}\hskip 3cm &
\int _{x/2}^{x-\frac {1} {8}}|T_1(x, x')|dx'\le C,\,\,\forall x>1.\label{S2.4ME6}
\end{align}

We use now (\ref{S2RK2E3}) in the same region,
\begin{align}
|T_2(x, x')|\le &C\left| \frac {1} {\sinh|x^2 -x'^2|}-\frac {1} {\sinh (x^2 +x'^2)}\right|\le C
\left (\frac {1} {\sinh \frac {x} {8}}+\frac {1} {\sinh x^2} \right),\nonumber\\
\label{S2.4ME6Bis}
\hskip-1.8cm \hbox{from where:} \hskip 1.5cm &\int _{x/2}^{x-\frac {1} {8}} |T_2(x, x')|dx'\le C\frac {x} {\sinh \frac {x} {8}},\,\,\forall x>1.
\end{align}
Suppose now that $x'\in (x-1/8, x+1/8)$, by (\ref{S2RK1E2}) and (\ref{S2RK2E3Bis}) $T_1(x, x')$ and $T_2(x, x')$ are both bounded on 
$[x-1/8, x+1/8]$ and then
\begin{align}
\label{S2.4ME7}
\int  _{ x-\frac {1} {8} }^{x+\frac {1} {8}} (|T_1(x, x')|+|T_2(x, x')|)dx'\le C,\,\,\forall x>1.
\end{align}
In the region $x'\in (x+1/8, 3x/2)$, we may use again  (\ref{S2EXpr}) to obtain, as in (\ref{S2Eqpo1}),
\begin{align}
 |T_1(x, x')|\le&
 \frac {Cx'} {x}\left(\frac{2\min(s, x')^2} {(x'+x)(x^2+x'^2)}+\frac {1} {\sinh\frac { |x+x '|} {8}}\right)\le \frac {Cx'} {x^2}\nonumber \\
\hskip -2.4cm\hbox{from where}\hskip 3cm&
\int _{x-\frac {1} {8}}^{\frac {3x} {2}}|T_1(x, x')|dx'\le C,\,\,\forall x>1.\label{S2.4ME8}
\end{align}
By  (\ref{S2RK2E3}) we have  in the same region,
\begin{align}
|T_2(x, x')|\le &C\left| \frac {1} {\sinh|x^2 -x'^2|}-\frac {1} {\sinh (x^2 +x'^2)}\right|\le C
\left (\frac {1} {\sinh \frac {x} {8}}+\frac {1} {\sinh x^2} \right),\\
\label{S2.4ME8Bis}
\hskip -0.33cm \hbox{and then}\hskip 2cm &\int _{x-\frac {1} {8}}^{\frac {3x} {2}} |T_2(x, x')|dx'\le C\frac {x} {\sinh \frac {x} {8}},\,\,\forall x>1.
\end{align}
If $x'>3x/2$, by (\ref{S2RK1E5}),
\begin{align*}
|T_1(x, x')|\le \frac {Cx'} {x }\left(e^{-C_1x'^2}+\frac {x^2} {x'^4}\right)
\end{align*}
and we may write,
\begin{align*}
&e^{-C_1x'^2}\le e^{-\frac {C_1x^2} {2}}e^{-\frac {C_1x'^2} {2}}\,\,\hbox{and}\,\,\,\,\,\frac {x} {x'^3}\le \frac {1} {x'^2 }
\end{align*}
from where,
\begin{align}
\label{S2.4ME9}
\int _{ \frac {3x} {2} }^\infty |T_1(x, x')|dx'\le Ce^{-\frac {C_1x^2} {2}} +\int _{ \frac {3x} {2} }^\infty \frac {dx'} {x'^2}
\le C\left(e^{-\frac {C_1x^2} {2}}+\frac {1} {x}\right),\,\,\forall x>1.
\end{align}
By (\ref{S2RK2E5}), 
\begin{align*}
|T_2(x, x')|&\le  
\begin{cases}
Ce^{-(x'^2-x^2 )}\left( e^{-(x'^2 -x^2)}+\frac {x^2 } {x'^2} \right)\frac {x'^3} {x^3},\,\,\,\,\forall x>\frac {3x} {2}\\
Ce^{-\frac {5x'^2} {9}}\left( e^{-\frac {5x'^2} {9}}+1\right)\frac {x'^3} {x^3},\,\,\,\,\forall x>\frac {3x}{2},
\end{cases}\nonumber\\
\int _{ \frac {3x} {2} }^\infty & |T_2(x, x')|dx'\le Ce^{-\frac {5x^2 } {9}},\,\,\forall x>1.
\end{align*}
\vskip -1cm 
\end{proof}
Similar arguments show the second Corollary,

\begin{cor}
\label{S4PMu}
$$
\widetilde M=\sup  _{ x'>0 }\mu (x')=\sup _{ x'>0 }\int _0^\infty |T(x, x')|dx<\infty
$$
\end{cor}
\begin{proof}
As before,
\begin{align*}
&\mu (x')\le \mu _1(x¡)+\mu _2(x')\\
&\mu _\ell (x')=\int _0^\infty |T_\ell (x, x')|dx,\,\ell =1, 2.
\end{align*}
Suppose first that $x'\in (0, 2)$. We have then,
\begin{align}
\label{S4PMuE1}
\mu _1(x')\le \int _0^5|T_1(x, x')|dx+\int _5^\infty |T_1(x, x')|dx
\end{align}
We use (\ref{S2RK1E1Sm}) in the first integral of the right hand side of (\ref{S4PMuE1}).  Since $x>5>2x'$ in the second integral of the right hand side of (\ref{S4PMuE1}), we may use (\ref{S2RK1E5ter}) and deduce that $\sup _{ x'\in (0, 2) }\mu_1 (x')<\infty$. Similarly,
\begin{align*}
\mu _2(x')\le \int _0^5|T_2(x, x')|dx+\int _5^\infty |T_2(x, x')|dx
\end{align*}
where (\ref{S2RK2E1}) and (\ref{S2RK2E4}) yield $\sup _{ x'\in (0, 2) }\mu _2(x')<\infty$. 

Suppose now that $x'\in (2, 15/2)$. We write then,
\begin{align}
\label{S4PMuE2}
\int _5^\infty |T_1(x, x')|dx=\int  _{ 5 }^{2x'}|T_1(x, x')|dx+\int_{2x'}^\infty|T_1(x, x')|dx
\end{align}
The first integral in the right hand side of (\ref{S4PMuE2}) is estimated using (\ref{S2RK1E1Sm}),
$$
\int  _{ 5 }^{2x'}|T_1(x, x')|dx\le \int  _{ 5 }^{15}|T_1(x, x')|dx\le C.
$$
The second integral in the right hand side of (\ref{S4PMuE2}) may be estimated using (\ref{S2RK1E5ter}) to obtain,
\begin{align*}
\int_{2x'}^\infty|T_1(x, x')|dx\le Cx' \int  _{ 2x' }^\infty \frac {1} {x}\left(\frac {1} {\sinh \frac {3x^2} {4}}+\frac {x'^2} {x^4}\right)dx
\le C\int _0^\infty e^{-\frac {3x^2} {8}}dx+Cx'^{-1}.
\end{align*}
and it follows that $\sup _{ x'\in (2, 15/2 )}\mu _1(x')<\infty$.  A similar argument using  (\ref{S2RK2E1}) and  (\ref{S2RK2E4}) gives $\sup _{ x'\in (2, 15/2 )}\mu _2(x')<\infty$.

Suppose now that $x'>15/2$. Then, 
\begin{align}
\label{S4PMuE3}
\int _5^\infty |T_1(x, x')|dx=\int  _{ 5 }^{\frac {2x'} {3}}|T_1(x, x')|dx+\int  _{ \frac {2x'} {3} }^{2x'}|T_1(x, x')|dx+\int_{2x'}^\infty|T_1(x, x')|dx
\end{align}
In the first and third integrals  of the righthand side of (\ref{S4PMuE3}) we use (\ref{S2RK1E5}),
\begin{align*}
\int  _{ 5 }^{\frac {2x'} {3}}|T_1(x, x')|dx\le Cx'\int _5^{\frac {2x'} {3}} \frac {1} {x }\left(e^{-C_1x'^2} +\frac {x^2} {x'^4}\right)dx<C\\
\int_{2x'}^\infty|T_1(x, x')|dx\le Cx'\int _{2x'} ^\infty \frac {1} {x }\left(e^{-C_1x^2} +\frac {x'^2} {x^4}\right)dx<C
\end{align*}
Since $|x-x'|<x/2$ in the second integral at the right hand side of (\ref{S4PMuE3}) we write,
\begin{align}
\label{S4PMuE4}
\int  _{ \frac {2x'} {3} }^{2x'}|T_1(x, x')|dx=\int  _{ |x'-x|<\frac {1} {8} } |T_1(x, x')|dx+\int  _{ \frac {1} {8}<|x'-x|<x/2 } |T_1(x, x')|dx
\end{align}
Using (\ref{S2RK1E2})  
$$
\int  _{ |x'-x|<\frac {1} {8} } |T_1(x, x')|dx\le C
$$
We use now  (\ref{S2RK1E4}) in the second integral at the righthand side of (\ref{S4PMuE4}),
\begin{align*}
&\int  _{ \frac {1} {8}<|x'-x|<x/2 } |T_1(x, x')|dx\le \\
&\le Cx' \int  _{\frac {1} {8}<|x-x'|<x/2 }\frac {1} {x}\left(\frac{2\min(x',x)^2} {(x^2-x'^2)(x^2+x'^2)}+\left|\frac {1} {\sinh |x^2-x'^2|}-\frac {1} {\sinh (x^2 +x'^2)}\right|\right)dx\\
&\le Cx' \int  _{\frac {1} {8}<|x-x'|<x/2 }\frac {1} {x}\left(\frac{\min(x',x)^2} {(x+x')(x^2+x'^2)}+\frac {1} {\sinh |(x+x')/8|}+\frac {1} {\sinh (x^2 +x'^2)}\right)dx<C.
\end{align*}
from where $\sup _{ x'>15/2 }\mu _1(x')<\infty$. A similar argument shows $\sup _{ x'>15/2 }\mu _2(x')<\infty$ , using  (\ref{S2RK2E3}), (\ref{S2RK2E3Bis}) instead of  (\ref{S2RK1E2})  and 
(\ref{S2RK2E5}), (\ref{S2RK2E4}) instead of (\ref{S2RK1E4}).
\end{proof}
 
\begin{proof}
[\upshape\bfseries{Proof of Proposition  \ref{S2AFF1}}]

The proof of (i) and (ii)  are now straightforward:
\begin{align}
|F(g(x))|&\le |g(x)|\int _0^\infty |T(x, x')|dx'+\int_0^\infty |g(x')|T(x, x')|dx' \label{S2AFF1E1}\\
&\le 2M||g||_\infty \nonumber \\
||F(g)||_1&\le \int _0^\infty |g(x)|\int _0^\infty |T(x, x')|dx'dx+\int _0^\infty\int _0^\infty|g(x')| |T(x, x')|dx'dx \nonumber \\
&\le \widetilde M||g||_1+M||g||_1. \nonumber
\end{align}
{\it Proof of (iii).} Consider again the right hand side of (\ref{S2AFF1E1}) and notice that,
by (\ref{S2RK1E1Sm}) and (\ref{S2RK2E1}), for $x\in (0, 1)$,
\begin{align*}
&\int_0^1 |g(x')||T(x, x')|dx'\le \int_0^1 |g(x')|\Big(|T_1(x, x')|+|T_2(x, x')|\Big)dx'\\
&\le C\int_0^1 |g(x')|\Big(1+xx'\Big)dx'\le C\sup _{ 0\le x'\le 1 } x'^\theta |g(x')| \int _0^1x'^{-\theta}\Big(1+xx' \Big)dx'\\
&\le  C|||f(s)|||_\theta.
\end{align*}
by Corollary \ref{S4PM},
\begin{align*}
&\int_1^\infty  |g(x')||T(x, x')|dx'\le ||g|| _{ L^\infty(1, \infty) }\int _0^\infty |T(x, x')|dx'\le M ||g|| _{ L^\infty(1, \infty) }
\end{align*}
and then, 
\begin{align*}
\sup _{ 0\le x'\le 1 } x'^\theta |F(g)(x')|\le M \sup _{ 0\le x'\le 1 } x'^\theta |g(x')|+C|||g||| _{ \theta }.
\end{align*}
By (\ref{S2RK1E5ter}) and (\ref{S2RK2E4}), for $x>2$,
\begin{align*}
\int _0^1 |g(x')| \Big(|T_1(x, x')|+|T_2(x, x')|\Big)dx'\le \sup _{ 0\le x'\le 1 } x'^\theta |g(x')|\times \\
\times \int _0^1 x'^{-\theta} \Big(|T_1(x, x')|+|T_2(x, x')|\Big)dx'\le C \sup _{ 0\le x\le 1 } x'^\theta |g(x')|\times \\
 \int _0^1 x'^{1-\theta} \Big(1+x'^2\Big)dx'= C \sup _{ 0\le x'\le 1 } x'^\theta |g(x')|
\end{align*}
and by, Corollary \ref{S4PM},
\begin{align*}
&\int_1^\infty  |g(x')||T(x, x')|dx'\le ||g|| _{ L^\infty(1, \infty) }\int _0^\infty |T(x, x')|dx'\le M ||g|| _{ L^\infty(1, \infty) }.
\end{align*}
and $ |||F(g)||| _{ \theta }\le C|||g|||_\theta$.
{\it Proof of (iv).} By (ii)  only  $F(g)\in L^\infty _{ loc }(0, \infty)$ remains to be proved. For  $K=[a, b] \subset (0, \infty)$, $[a, b] \in (\alpha , \beta )$ and all $x\in [a, b]$,
\begin{equation*}
 \int _0^\infty T(x, x') g (x')dx'= \int  _{ y\in (\alpha , \beta ) } T(x, x') g (x')dx'+ \int  _{ y\in (\alpha , \beta )^c }T(x, x') g (x')dx'=I_1+I_2.
\end{equation*}
where,
\begin{align*}
|I_1(x)|=\left| \int _\alpha ^\beta  T(x, x') g (x')dx\right|&\le ||g|| _{ L^\infty(\alpha , \beta ) } \int _\alpha ^\beta  |T(x, x')|dx'\\
&\le m(x)||g|| _{ L^\infty(\alpha , \beta ) } \le M||g|| _{ L^\infty(\alpha , \beta ) } .
\end{align*}
and
\begin{equation*}
|I_2(x)|\le \int_0^\alpha |T(x, x')| |g (x')|dx' +\int_\beta ^\infty |T(x, x')| |g (x')|dx'.
\end{equation*}
By (\ref{S2RK1E1Sm}) and (\ref{S2RK2E1}),
\begin{align*}
\int_0^\alpha |T(x, x')| |g (x')|dx' &\le \sup \{|T(x, y)|;x\in [a, b],\,y<\alpha  \}\int _0^\infty |g(x')|dx' \\
&\le C_\beta (1+b\alpha )\int _0^\infty |g(x')|dx' 
\end{align*}
A simple inspection of the expression of $T(x, x')$ given by (\ref{S3EInt2}) shows that $C(K)=\sup \{|T(x, y)|;x\in [a, b],\,y>\beta \}<\infty$. Then,
\begin{equation*}
\int_\beta ^\infty |T(x, x')| |g (x')|dx'\le C(K)\int _0^\infty |g(x')|dx'.
\end{equation*}
and,
\begin{align*}
\sup _{ x\in K }|F(g)(x)|\le ||g|| _{ L^\infty((\alpha , \beta )) } M+( C_\beta +C(K)) (1+b\alpha )||g||_1.
\end{align*}
{\it Proof of (v).} For all $x>0$,
\begin{align*}
\left|\int _0^\infty T(x, x')g(x')dx'\right|\le \int _0^R T(x, x')|g(x')|dx'+\int _R^\infty T(x, x')|g(x')|dx'\\
\le  \int _0^R T(x, x')|g(x')|dx'+M||g|| _{ L^\infty(R, \infty) }.
\end{align*}
If $x\le 2R$, by (\ref{S2RK1E1Sm}) and (\ref{S2RK2E1})
\begin{equation*}
\int _0^R T(x, x')|g(x')|dx'\le C _{ 2R }x\int _0^Rx' |g(x')dx'\le 2C_{ 2R }R^2||g||_1.
\end{equation*}
For $x>2R$, we use the original expression of $T(x, x')$ in (\ref{S3EInt2}):
\begin{align*}
&0\le \frac {x'^3\sinh x^2} {x^3\sinh x'^2} \left(\frac {1} {\sinh|x^2-x'^2|}-\frac {1} {\sinh (x^2+x')^2 } \right) \le
C\frac {e^{x^2}} {x^3} \left( e^{-x^2+x'^2} \right)\le \frac {C_R} {x^3}\\
&0\le \frac {x'} {x}\left(\frac {1} {|x^2-x'^2|}-\ \frac {1} {x^2+x'^2} \right)\le \frac {R} {x^3}
\end{align*}
then,
\begin{equation*}
|T(x, x')|\le \frac {C_R} {x^3},\,\,\forall x>2R,\,x'\in (0, R).
\end{equation*}
and (iv) follow since,
\begin{equation*}
\int _0^R T(x, x')|g(x')|dx'\le \frac {C _{ R }} {x^3}||g||_1,\,\,\forall x>2R.
\end{equation*}
\vskip-1cm
\end{proof}
\section{Existence of global solution $f$.}
\setcounter{equation}{0}
\setcounter{theo}{0}
Using the properties of the operator $L$, Proposition  \ref{S2AFF1} and a  fixed point argument, classical solutions $f\in C([0, \infty); L^1(0, \infty))$ of the Cauchy problem for   (\ref{S3E1.1}) with  initial data $f_0\in L^1(0, \infty)$ are obtained.
If moreover $f_0\in L^1(0, \infty)\cap L^\infty (0, \infty)$ then 
$f\in C([0, \infty); L^1(0, \infty))\cap L^\infty((0, \infty; L^\infty)$. However it is interesting to consider initial data slightly more general  than in $f_0\in L^1(0, \infty)\cap L^\infty (0, \infty)$ but whose solutions are more regular than just integrable with respect to $x$ in $(0, \infty)$.

\begin{theo}
\label{S6WeakSolL2}
Suppose that $f_0\in L^1(0, \infty)$ satisfies
\begin{equation}
\label{S4E1244}
|||f_0|||_\theta\equiv \sup _{ 0<x<1 }x^\theta |f_0(x)|+\sup _{ x>1 }|f_0(x)|<\infty
\end{equation}
for  some $\theta\in (0, 1)$. Then,  there exists a function 
\begin{equation}
\label{S4E124}
f\in C([0, \infty); L^1 (0, \infty))\cap L^\infty _{ loc }((0, \infty); L^\infty(0, \infty))
\end{equation}
satisfying,  
\begin{align}
\label{S3EFPf}
f(t)=S(t)f_0+\int _0^tS(t-s)F(f(s))ds.
\end{align}
For all $t>0$,  $f(t)\in C(0, \infty)$ and for all $T>0$ and $t\in (0, T)$,
\begin{align}
&||f(t)||_1\le C_T||f_0||_1 \label{S4E164-1}\\
&||f(t)||_1+|||f(t)||| _{ \theta }\le C(T)(||f_0||_1+|||f_0||| _{ \theta }),\,\forall t\in (0, T)\label{S4E164-2}\\
&||f(t)||_\infty \le 
C(T, \theta) \left(||f_0|| _{ L^\infty(1, \infty) }+ t^{-\theta }\sup _{ 0<y<1 }y^\theta|f_0(y)|+ ||f_0||_1\right). \label{S4E168}
\end{align}
\end{theo}
\begin{proof}  Given $f_0$ fixed and  satisfying the hypothesis,  consider the operator
\begin{equation*}
\mathscr L (f)(t, x)=S(t)f_0(x)+\int _0^t S(t-s)(F(f(s)))ds
\end{equation*}
on the space
\begin{align*}
&Z_T=C\left([0, T); L^1 (0, \infty)\right) \cap C\left((0, T);L^\infty (0, \infty)\right), \\
&||f|| _{ Z_T }=\sup _{ s\in (0, T) } \left(||f(s)|| _{ 1 }+||f(s)|| _{ L^\infty(1, \infty) }+s^{\theta}\sup _{ 0<x<1 }|f(s, x)|\right)
\end{align*}
By (ii) in Proposition  (\ref{S2AFF1}), 
\begin{align}
||\mathscr L (f(t))|| _{ 1 }\le C _{ s }||f_0|| _{ 1 }+C_sC_F\int _0^t||f(s)|| _{ 1 }ds \label{S4EV0}
\end{align}
On the other hand,
\begin{align*}
|\mathscr L (f(t))(x)|\le |S(t)f_0(x)|+\int _0^t |S(t-s)F(f(s))(x)|ds.
\end{align*}
By (iv) in Proposition  (\ref{S2AFF1}) and Proposition \ref{S5P1} in the Appendix, for all $t>0$, $s\in (0, t)$ and $x\in (0, 2)$,
\begin{align*}
&|S(t-s)F(f(s))(x)|\le C||F(s)|| _{ L^\infty (1, \infty) }+C(t-s)^{-\theta}\sup _{ 0<x<1 }x^{\theta}|F(f(s))(x)|\\
&\le C(1+(t-s)^{-\theta})||F(f)||_\infty\le C(1+(t-s)^{-\theta})(||f(t) || _{ L^\infty (1, \infty) }+||f(t)||_1)
\end{align*}
Since we also have, for $x\in (0, 2)$,
\begin{align}
\label{S4EPMX1}
|S(t)f_0(x)|\le C\left(||f_0|| _{ L^\infty (1, \infty) }+t^{-\theta}\sup _{ 0<x<1 }x^\theta |f_0(x)|\right)
\end{align}
we deduce, for $t>0$, $s\in (0, t)$ and $x\in (0, 2)$,
\begin{align*}
|\mathscr L (f(t))(x)|\le C\left(||f_0|| _{ L^\infty (1, \infty) }+t^{-\theta}\sup _{ 0<x<1 }x^\theta |f_0(x)|\right)+\\
+C\int _0^t (1+(t-s)^{-\theta})\left(||f(s)||_1+||f(s)|| _{L^\infty(1, \infty)}\right)ds
\end{align*}
and then,
\begin{align}
\sup _{ 0<x<2 }t^\theta &|\mathscr L (f(t))(x)|
\le  C\left(t^{\theta }||f_0|| _{ L^\infty (1, \infty) }+\sup _{ 0<x<1 }x^\theta |f_0(x)|\right)+\nonumber \\
&+Ct^\theta \int _0^t (1+(t-s)^{-\theta})\left(||f(s)||_1+||f(s)|| _{L^\infty(1, \infty)}\right)ds.\label{S4EV1}
\end{align}
If $x>2$,
\begin{align*}
|S(t-s)F(f(s))(x)|\le &C||F(s)|| _{ L^\infty (1, \infty) }+C||F(f(s))||_1\\
\le &C\left(||f(s)||_1+||f(s)|| _{ L^\infty(1, \infty) }\right)\\
|S(t)f_0(x)|\le &C\left(||f_0|| _{ L^\infty (1, \infty) }+||f_0||_1\right),
\end{align*}
from where,
\begin{align}
\sup _{ x>2 }|\mathscr L (f(t))(x)|&\le C\left(||f_0|| _{ L^\infty (1, \infty) }+||f_0||_1\right)+\nonumber \\
&+C\int _0^t\left(||f(s)||_1+||f(s)|| _{ L^\infty(1, \infty) }\right)ds. \label{S4EV2}
\end{align}
Adding (\ref{S4EV0}), (\ref{S4EV1}) and (\ref{S4EV2}),
\begin{align}
||\mathscr L (f(t))|| _{ 1 }&+t^\theta \left|\sup _{ 0<x<1 }x^\theta \mathscr L (f(t))(x)\right|+\left|\sup _{ x>1 } \mathscr L (f(t))(x)\right|\le\nonumber\\
&\le ||\mathscr L (f(t))|| _{ 1 }+ t^\theta  \left|\sup _{ 0<x<1 } \mathscr L (f(t))(x)\right|
+\left|\sup _{ x>1 } \mathscr L (f(t))(x)\right| \nonumber\\
&\le C\left(||f_0|| _{ L^\infty (R, \infty) }+||f_0||_1\right)+C\left(t^{\theta }||f_0|| _{ L^\infty (1, \infty) }+ \sup _{ 0<x<1 }x^\theta |f_0(x)|\right)+ \nonumber\\
&+C t^\theta \int _0^t (1+(t-s)^{-\theta})\left(||f(s)||_1+||f(s)|| _{L^\infty(1, \infty)}\right)ds+ \nonumber\\
&+C\int _0^t\left(||f(s)||_1+||f(s)|| _{ L^\infty(1, \infty) }\right)ds \nonumber\\
&\le C(1+ t^\theta) ||f_0|| _{ L^\infty (1, \infty) }+||f_0||_1+C  \sup _{ 0<x<1 }x^\theta |f_0(x)|+ \nonumber\\
&+C t^\theta \sup _{ 0<s<t } \left(||f(s)||_1+||f(s)|| _{L^\infty(R, \infty)}\right) (t+t^{1-\theta}) +T \!\!\! \sup_ 
 {\substack{0<x<1\\0<s<T }} s^\theta |f(t, x)| \nonumber\\
\le  &C_1\left((1+T^\theta) ||f_0|| _{ L^\infty (1, \infty) }+||f_0||_1+  \sup _{ 0<x<1 }x^\theta |f_0(x)|\right)+C_2 T ||f|| _{ Z_T }.
\label{S4EV5-1}
\end{align}
If we denote
\begin{equation}
\label{S4EV5}
\gamma _0=C_1(2 ||f_0|| _{ L^\infty (1, \infty) }+||f_0||_1+ \sup _{ 0<x<1 }x^\theta |f_0(x)|)
\end{equation}
we have then proved,
\begin{align}
\label{S4EV8}
||\mathscr L(f)|| _{ Z_T }\le \gamma _0+C _2T ||f|| _{ Z_T },\,\,\forall T\in (0, 1).
\end{align}
Let   $\rho >0$ and $T>0$  be such that:
\begin{align}
\label{S4EV6}
&T<\min \left(1, \frac {1} {2C_2} \right)\\
&\rho >2\gamma _0.\label{S4EV6B}
\end{align}
Then,  for all $f\in Z_T$ such that $||f|| _{ Z_T }\le \rho $,
\begin{equation}
\label{S4EV10}
||\mathscr L(f)|| _{ Z_T }\le \gamma _0+C _2T\rho\le \rho 
\end{equation}
and then,
\begin{align}
&\mathscr L: B _{ Z_T }(0, \rho )\mapsto B _{ Z_T }(0, \rho ),\label{S4EV12}\\
&B _{ Z_T }(0, \rho )=\left\{f\in Z_T; ||f|| _{ Z_T }\le \rho  \right\}.\label{S4EV14}
\end{align}
On the other hand,
\begin{align*}
\mathscr L (f(t))-\mathscr L (g(t))=\int _0^t S(t-s)F(f(s)-g(s))ds
\end{align*}
and arguing as before,
\begin{align*}
||\mathscr L (f)-\mathscr L (g)|| _{ Z_T }\le C T||f-g|| _{ Z_T }
\end{align*}
The map $\mathscr L$ is then a contraction form $B _{ Z_T }(0, \rho )$ into itself if $T$ is small enough, and has  a fixed point $u\in B _{ Z_T }(0, \rho )$ that satisfies
\begin{equation}
\label{S4EV18}
f(t)=S(t)f_0+\int _0^tS(t-s)F(f(s))ds
\end{equation}
in $Z_T$. Property (\ref{S4E164-1})  follows from  and Gronwall's Lemma on $(0, T)$  and by (\ref{S4EV1}) and (\ref{S4EV2}), 
\begin{align}
\label{S4EV18mm}
&||f(t)||_1+|||f(t)||| _{ \theta }\le C(||f_0||_1+|||f_0||| _{ \theta })+\\
&+C\left(1+t^\theta\right) \int _0^t\Big(1+(t-s)^{-\theta}\Big)\Big(||f(s) || _{ L^\infty (1, \infty) }+||f(s)||_1\Big),\,\,\forall t\in (0, T).
\end{align}
Then, there exists a constant $C=C(T)>0$ such that, (\ref{S4E164-2}) holds true.

On the other hand, since $f_0\in L^1(0, \infty)$ and $|||f_0|||_\theta <\infty$, by Proposition \ref{S5P1} in the Appendix,  $S(t)f_0\in L^\infty(0, \infty)$. Moreover, by  Proposition (\ref{S5P1}) and Proposition (\ref{S2AFF1}), for $t\in (0, T)$ and $x\in (0, 2)$,
\begin{align}
|S(t-s)F(f(s))(x)|\le &C(t-s)^{-\theta} \sup _{0<x<1 }x^\theta|F(f(s)) (x)|+||F(f(s)|| _{ L^\infty(1, \infty) } \nonumber\\
& \le   C (1+ (t-s)^{-\theta} ||F(f(s))||_\infty)\nonumber \\
&\le C(1+  (t-s)^{-\theta})(||f(s)||_1+||f(s)|| _{ L^\infty(1, \infty) })\label{S4E82}\\
&\le C(1+ (t-s)^{-\theta})\rho \nonumber
\end{align}
It immediately follows that for all $t\in (0, T)$ and $x\in (0, 2)$,
\begin{align*}
|f(t, x)|\le C||f_0|| _{ L^\infty(1, \infty) }+Ct^{-\theta }\sup _{ 0<y<1 }y^\theta|f_0(y)|+C\rho \int _0^t (1+  (t-s)^{-\theta})ds.
\end{align*}
and then, $f(t)\in L^\infty(0, \infty)$ for all $t\in (0, T)$. 
We wish to extend now this function $f$ for all $t>0$. We notice to this end that, for all $x>1$,

\begin{align*}
|f(t, x)|\le C||f_0|| _{ L^\infty(1, \infty) }&+Ct^{-\theta }\sup _{ 0<y<1 }y^\theta|f_0(y)|+ \int _0^t |S(t-s)F(f(s))(x)|ds\\
\le C||f_0|| _{ L^\infty(1, \infty) }+&Ct^{-\theta }\sup _{ 0<y<1 }y^\theta|f_0(y)|+\\
&+C\int _0^t(1+  (t-s)^{-\theta})(||f(s)||_1+||f(s)|| _{ L^\infty(1, \infty) })
\end{align*}
Since by Proposition \ref{S2AFF1}
\begin{equation}
\label{S4R164}
||f(t)||_1\le C||f_0||_1+C\int _0^t||f(s)||_1ds,
\end{equation}
we obtain, 
\begin{align*}
\sup _{ x>1 } |f(t, x)|+||f(t)||_1\le 
C\left(t^{-\theta }\sup  _{ x\in (0, 1)}x^\theta|f_0(x)|+||f_0||_1\right)+\\
+C\int _0^t(1+  (t-s)^{-\theta})(||f(s)||_1+||f(s)|| _{ L^\infty(1, \infty) }).
\end{align*}
It follows by Gronwall's Lemma,  that for some constant $C$ depending on $T$ and $\theta$,
\begin{align}
\label{S4E80}
\sup _{ x>1 } |f(t, x)|+||f(t)||_1\le 
C(T, \theta)\left(||f_0|| _{ L^\infty (1, \infty) }+t^{-\theta}\sup _{ 0<x<1 }x^\theta |f_0(x)| +||f_0||_1\right).
\end{align}
On the other hand, for $x\in (0, 2)$, using (\ref{S4E82}) 
\begin{align*}
|f(t, x)|&\le C||f_0|| _{ L^\infty(1, \infty) }+Ct^{-\theta } \sup _{ 0<y<1 }y^\theta|f_0(y)|+\nonumber\\
&+C \int _0^t (1+  (t-s)^{-\theta})(||f(s)||_1+||f(s)|| _{ L^\infty(1, \infty) })ds
\end{align*}
and by  (\ref{S4E80}), for all $x\in (0, 2)$,
\begin{align}
|f(t, x)|&\le C||f_0|| _{ L^\infty(1, \infty) }+Ct^{-\theta }\sup _{ 0<y<1 }y^\theta|f_0(y)|+\nonumber\\
+C(T&, \theta) \int _0^t (1+  (t-s)^{-\theta}) \left(||f_0|| _{ L^\infty(1, \infty) }+s^{-\theta}\sup _{ 0<x<1 }x^\theta |f_0(x)| +||f_0||_1\right)ds\nonumber\\
&\le C(T, \theta) \left(||f_0|| _{ L^\infty(1, \infty) }+ t^{-\theta }\sup _{ 0<y<1 }y^\theta|f_0(y)|+ ||f_0||_1\right).
\label{S4E84}
\end{align}
We obtain from (\ref{S4R164}), (\ref{S4E80}) and  (\ref{S4E84}) for all $t\in (0, T)$,
\begin{equation}
\label{S4EPMX5}
||f(t)||_1+||f(t)|| _\infty \le
C(T, \theta) \left(||f_0|| _{ L^\infty(1, \infty) }+ t^{-\theta }\sup _{ 0<y<1 }y^\theta|f_0(y)|+ ||f_0||_1\right)
\end{equation}

By a classical argument it follows that the function $f$ may be extended to a function, still denoted $f$, for all $t>0$ such that $f\in Z_t$ for all $t>0$ and satisfies  (\ref{S4EV18}) for all $t>0$. 

The same arguments used to prove  the estimates (\ref{S4E164-1}), (\ref{S4E164-2}) and (\ref{S4E168}) on the interval of time given by (\ref{S4EV6}) may now be applied to obtain (\ref{S4E164-1}), (\ref{S4E164-2}), (\ref{S4E168}) on all finite interval $(0, T)$ for all $T>0$. 

Since $f_0\in L^1(0, \infty)$ and $|||f_0|||_\theta<\infty$ it follows from Proposition \ref{S3P7} that $S(t)f_0\in C(0, \infty)$ for every $t>0$. Since  $f(s)\in L^\infty(0, \infty)$ for all $s>0$ it follows by  Proposition \ref{S2AFF1}  that $F(f(s))\in L^\infty(0, \infty)$ too and therefore,  again by \ref{S3P7}, $S(t-s)F(f(s))\in C(0, \infty)$ for all $t>0$ and $s\in (0, t)$. This shows that $f(t)\in C(0, \infty)$ for all $t>0$.
\end{proof}
\begin{theo}
\label{S6WeakSolL2B}
Suppose that $f_0\in L^1(0, \infty)$ satisfies (\ref{S4E1244}) and $f$ is the function given by Theorem \ref{S6WeakSolL2B}. Then, 
\begin{equation}
 \frac {\partial f} {\partial t},\, \mathcal L(f) \in L^\infty _{ loc }((0, \infty)\times (0, \infty))\cap L^1((0, T)\times (0, \infty)),\,\forall T >0,\label{S4E140-4X} 
\end{equation}
and, for all almost every $t>0,\, x>0$,
\begin{equation}
\label {S4E140-2X}
\frac {\partial f(t, x)} {\partial t}=\mathcal L(f(t))(x).
\end{equation}
There exists a function $\tilde H\in L^\infty ((\delta , \infty)\times (\delta , \infty))$ for all $\delta >0$ such that
\begin{align}
\left|\frac {\partial f(t, x)} {\partial t} \right|&+ |\mathcal L(f)(t, x)|\le C\left(\sup _{ 0<y<1 }y^{\theta}|f_0(y)|+||f_0|| _{ L^\infty(1, \infty) }+ ||f_0||_1\right)\tilde H(t, x)\label{S4E125}
\end{align}
($\tilde H$ is defined in (\ref{S3EHH1}) below), and
\begin{align}
&\left|\frac {\partial f} {\partial x}(t, x) \right|\le \tilde\lambda _{ \beta , \delta  } (t, x)|||f_0||| _{ \theta, 1 }+\sup _{ 0\le s\le t }(|||f(s)|||_\theta+|||f(s)|||_1)
\int _0^t \tilde\lambda _{\beta, \delta   } (s, x)ds, 
\label{S3fx1}\\
&\hbox{where:}\,\,\,\forall \delta \in (0, 1),\forall \beta \in (0, 1),\,\quad \tilde\lambda _{\beta, \delta   } (t, x)=
\begin{cases}
x^{-1+\delta }t^{-\theta-\delta },\,\,0<x<t<1\\
t^2,\,\,t>1, x\in (0, t)\\
x^{-1-\beta  }t^{\beta  -1},\,\,\,x>t.
\end{cases}
\end{align}

For all $\varphi \in C^1_0([0, \infty))$, the map
$
t\mapsto \displaystyle{\int _0^\infty \varphi (x)f(t, x)dx}
$
belongs to $W _{ loc }^{1,1}(0, \infty)$ and for almost every $t>0$,
\begin{equation}
\frac {d} {dt}\int _0^\infty \varphi (x)f(t, x)dx=\int_0^\infty \mathcal L(f(t))(x)\varphi (x)dx \label{S4cor1E2}
\end{equation}
\end{theo}
\begin{proof} We begin proving  (\ref{S4E140-4X}),  (\ref{S4E140-2X}), and  (\ref{S4E125}).
Since $f_0\in L^1(0, \infty)$ and $|||f_0|||_\theta <\infty$, by Proposition \ref{S5P1} in the Appendix $S(t)u_0\in L^\infty(0, \infty)$ and  Theorem 1.2  in \cite{m}, 
\begin{align*}
&\frac {\partial } {\partial t}((S(t)f_0)\in L^\infty(0, \infty);\,\,\,\, L(S(t)f_0)(x)\in L^\infty(0, \infty),\,\,\forall t>0\\
&\frac {\partial } {\partial t}((S(t)f_0)(x))=L(S(t)f_0)(x),\,\, \hbox{\it for a. e.}\,\, t>0, x>0.
\end{align*}

Since for all $s>0$, $F(f(s))\in L^\infty(0, \infty)\cap L^1(0, \infty)$, for almost every  $x>0$, $t\in (0, T)$ and $s\in (0, t)$, by the same argument, 
\begin{equation}
\frac {\partial}{\partial t}[(S(t-s)F(f(s)))(x)]=L(S(t-s)F(f(s)))(x),\label{S4Evyu1}
\end{equation}
where both terms belong to $L^\infty(0, \infty)$. Let us define
\begin{align}
&\xi (t, x)=\frac {t^3} {\max (t^4, x^4) } \label{S1Exi}\\
&\zeta _\theta (t, x)=\frac {\min (t, t^{2-\theta})} {x^3}\1 _{ t<2x/3 }+\frac {x\1 _{ 2x< t}} {\max(t^{2+\theta}, t^3)} +\frac {\1 _{ 2x/3<t<2x }} {\max(x^2, x^{1+\theta})}, \label{S1Zetai}.
\end{align}
By (\ref{S5R1}),  and  (iv) of  Proposition (\ref{S2AFF1}),
for all $t\in (0, T)$, $s\in (0, t)$ and $x>0$,

\begin{align}
\label{S4Evyu2B}
\left|\frac {\partial}{\partial t}(S(t-s)F(f(s)))(x)\right| & \le C||F(f(s)||_\infty \left(1+\xi (t-s, x)\right)+\zeta_0 (t-s, x)||F(f(s)|| _{ 1, 0 } \nonumber \\
& \le C\left(||f(s)|| _{ L^\infty(R, \infty) } +||f(s)||_1\right) \left(1+\xi (t-s, x)\right)+\nonumber \\
&\hskip 2.5cm +\zeta_0 (t-s, x)\left(||f(s)||_1+||f(s)|| _{L^ \infty (R, \infty) } \right) \nonumber \\
& \le C(1+\rho )\left(1+\xi (t-s, x)+\zeta_0 (t-s, x) \right).
\end{align}
The right hand side of (\ref{S4Evyu2B}) may now be estimated for all $s\in (0, t)$ using:
\begin{align*}
\zeta_0   (t-s, x)+\xi  (t-s, x)\le\begin{cases}
\displaystyle{\frac {C} {x}+ \frac {\1 _{ {0<s<t-x} }} {x}+\frac {\1 _{ t-x<s<t }t^3} {x^4}},\,\,\hbox{when}\,\,x\in (0, t)\\
\displaystyle{\,\,\,\,\,\frac {C} {x}+ \frac {1} {x^4}},\,\,\hbox{when}\,\,x>t
\end{cases}
\end{align*}
We deduce, for all $t>0$, $x>0$, 
\begin{align}
\label{S4E142}
\frac {\partial } {\partial t}\left(\int _0^tS(t-s)F(f(s))(x)ds \right)=F(f(t))(x)+\int _0^t\frac {\partial } {\partial t}\left(S(t-s)F(f(s))(x)\right)ds.
\end{align}
and,
\begin{align}
\label{S4Evyu2}
&\left|\frac {\partial}{\partial t}(S(t-s)F(f(s)))(x)\right|  \le C||F(f(s)||_\infty \left(1+\xi (t-s, x)\right)+\zeta (t-s, x)||F(f(s)||_1 \nonumber \\
&\le C\left(||f(s)|| _{ L^\infty(R, \infty) } +||f(s)||_1\right) \left(1+\xi (t-s, x)+\zeta_0 (t-s, x)\right).
\end{align}
Then,
\begin{align*}
\left|\frac {\partial f(t, x)} {\partial t} \right|\le \left|\frac {\partial } {\partial t}(S(t)f_0)(x)\right|+|F(f(t)(x)|+
\left|\int _0^t \frac {\partial } {\partial t}\left(S(t-s)F(f(s))\right)(x)ds \right| \nonumber \\
\le  C||S(t)u_0)||_\infty \left(1+\xi (t, x)\right)+\zeta_0 (t, x)||S(t)f_0|| _{ 1, 0 }+||F(f(t)|| _{ \infty }+\nonumber \\
+C\int _0^t
\left(||f(s)|| _{ L^\infty(1, \infty) } +||f(s)||_1\right) \left(1+\xi (t-s, x)+\zeta _0(t-s, x)\right)ds 
\end{align*}
Then, for all $T>0$ and $t\in (0, T)$, there exists $C=C(T, \theta)$ such that,
\begin{align*}
&\left|\frac {\partial f(t, x)} {\partial t} \right|\le C\left(||f_0|| _{ L^\infty(1, \infty) }+ t^{-\theta }\sup _{ 0<y<1 }y^\theta|f_0(y)|+ ||f_0||_1\right) (1+\xi (t, x)+\zeta _0(t, x))+\\
&+C\int _0^t\left(||f_0|| _{ L^\infty(1, \infty) }+ s^{-\theta }\sup _{ 0<y<1 }y^\theta|f_0(y)|+ ||f_0||_1\right)\left(1+\xi (t-s, x)+\zeta _0(t-s, x)\right)ds
\end{align*}
Using Proposition \ref{SAP5} in the Appendix it follows,
\begin{align*}
&\int _0^t\left(||f_0|| _{ L^\infty(1, \infty) }+ s^{-\theta }\sup _{ 0<y<1 }y^\theta|f_0(y)|+ ||f_0||_1\right)\left(1+\xi (t-s, x)+\zeta _0(t-s, x)\right)ds
\\ 
&\le C(||f_0|| _{ L^\infty(1, \infty) }+||f_0||_1)\left(\Xi_1(t, x)+ \omega  _{ 1, 0 }(t, x)\right)+\sup _{ 0<y<1 }y^\theta|f_0(y)|\left( \Xi_2(t, x)+\omega  _2(x, t)\right).
\end{align*}
We deduce, 

\begin{align}
&\left|\frac {\partial f(t, x)} {\partial t} \right| \le  C\left(||f_0|| _{ L^\infty(1, \infty) }+ t^{-\theta }\sup _{ 0<y<1 }y^\theta|f_0(y)|+ ||f_0||_1\right) (1+\xi (t, x)+\zeta _0(t, x))+\nonumber\\
&+ C(||f_0|| _{ L^\infty(1, \infty) }+||f_0||_1)\left(\Xi_1(t, x)+ \omega  _{ 1, 0 }(t, x)\right)+\sup _{ 0<y<1 }y^\theta|f_0(y)|\left( \Xi_2(t, x)+\omega  _{ 2 }(x, t)\right)\nonumber\\
&\le C\left(||f_0|| _{ L^\infty(1, \infty) }+ t^{-\theta }\sup _{ 0<y<1 }y^\theta|f_0(y)|+ ||f_0||_1\right)\tilde H(t, x)
\end{align}
with
\begin{align}
\tilde H(t, x)=(1+\xi (t, x)+\zeta _0(t, x))+\Xi_1(t, x)+ \omega  _{ 1, 0 }(t, x)+ \Xi_2(t, x)+\omega  _{ 2 }(x, t).\label{S3EHH1}
\end{align}
Estimates (\ref{S4E140-4X}) and (\ref{S4E140-2X}) for $\frac {\partial f} {\partial t}$  follow using again Proposition \ref{SAP5}. Moreover, by (\ref{S4E142}), (\ref{S4Evyu2})  and (\ref{S4Evyu1}),
\begin{align*}
\frac {\partial } {\partial t}\left(\int _0^tS(t-s)F(f(s))ds \right)=F(f(t))+\int _0^t\frac {\partial } {\partial t}\left(S(t-s)F(f(s))\right)ds.\\
=F(f(t))+\int _0^t L(S(t-s)F(f(s))) ds,
\end{align*}
for almost every $t>0$ and $x>0$,
\begin{align*}
\frac {\partial f(t, x)} {\partial t}&=\frac {\partial } {\partial t}((S(t)f_0)(x))+\frac {\partial } {\partial t}\left(\int _0^tS(t-s)F(f(s))(x)ds\right)\\
&=L(S(t)f_0)(x)+F(f(t))(x)+\int _0^t L(S(t-s)F(f(s)))(x) ds\\
&=L\left(S(t)f_0+ \int _0^t S(t-s)F(f(s)) ds\right)(x)+F(f(t))(x)\\
&=L(f(t))(x)+F(f(t))(x).
\end{align*}
By Proposition \ref{S2AFF1} this shows $\mathcal L(f)\in L^\infty _{ loc }((0 ,\infty); L^\infty(\delta , \infty)\cap L^1(0, \infty))$ for all $\delta >0$. This ends the proof of (\ref{S4E140-4X}) and  (\ref{S4E140-2X}), and proves (\ref{S4E125}). 

To prove (\ref{S3fx1}) notice  that by Corollary \ref{S4C26m}, $\partial x(S(t)f_0)$ exists for almost every $t>0$ and $x>0$ and satisfies   \ref{S4C26mE1} and by Proposition \ref{S2AFF1} and Corollary \ref{S4C26m}, $\partial x(S(t-s) F(f(s)))$ exists for almost every $t>0$, $s\in (0, t)$ and $x>0$ and satisfies:
\begin{align*}
\left|\frac {\partial S(t-s)f(s)(x)} {\partial x} \right|&\le h(t-s)g(x)|||F(f(s))||| _{ 1, \theta }\\
&\le C h(t-s)g(x)\left(||f(s)||_1+|||f(s)|||_\theta \right)\\
&\le C h(t-s)g(x)\sup _{ 0\le s\le t }\left(||f(s)||_1+|||f(s)|||_\theta \right).
\end{align*}
Since $h(t-s)\in L^1(0, t)$, (\ref{S3fx1}) easily follows.

On the other hand,  if we multiply both sides of (\ref{S4EV18}) by $\varphi \in C^1_0([0, \infty))$ and integrate,
\begin{equation*}
\int_0^\infty  f(t, x) \varphi (x)dx=\int_0^\infty \varphi (x)S(t)f_0(x)dx+\int _0^\infty \varphi (x)\int _0^tS(t-s)F(f(s))(x)dsdx.
\end{equation*}
In order to derivate this expression with respect to $t$, we use again (\ref{S4Evyu1}) and (\ref{S4Evyu2}) to obtain,
\begin{align}
\frac {d} {dt}\int_0^\infty  f(t, x) \varphi (x)dx=&\,\int_0^\infty \varphi (x)L(S(t)f_0)(x)dx+\int _0^\infty\varphi (x)F(f(t, x)dx+\nonumber\\
&+\int _0^\infty \varphi (x)\int _0^t L(S(t-s)F(f(s))(x))dsdx, \label{S4E38}
\end{align}
and for all $t\in (0, T)$,
\begin{align*}
&\left|\frac {d} {dt}\int_0^\infty  f(t, x) \varphi (x)dx\right|\le C||\varphi ||_1||f_0||_\infty+C(1+T)||\varphi ||_1||f|| _\infty.
\end{align*}
which  shows 
$$\frac {d} {dt}\int_0^\infty  f(t, x) \varphi (x)dx\in L^\infty _{ loc }([0, \infty).$$
Identity (\ref{S4cor1E2}) follows now, since
\begin{align*}
\frac {d} {dt}\int_0^\infty  f(t, x) \varphi (x)dx&=\int_0^\infty \varphi (x)L\left( (S(t)f_0)(x)+\int _0^t S(t-s)F(f(s))(x)ds\right)dx+\\
+\int _0^\infty&\varphi (x)F(f(t, x)dx=\int_0^\infty \varphi (x)L\left(f\right)dx+\int _0^\infty\varphi (x)F(f(t, x)dx.
\end{align*}
\vskip -0.5cm 
\end{proof}
\section{Further properties of the solution $f$.}
\setcounter{equation}{0}
\setcounter{theo}{0}
We describe in this Section  some further properties of the solutions $f$ given by Theorem \ref{S6WeakSolL2}. We first consider what are the variations of mass and  energy induced by  the initial perturbation $n_0(1+n_0)x^2f(\tau )$ of the equilibrium $n_0$ introduced in (\ref{S2Elinzn2}), (\ref{S2CcV}). Then we prove that for all $\delta >0$, $f\in L^\infty ((\delta , \infty)\times (0, \infty))$ and that for every $t>0$  the function $f(t)$ has a limit as $x\to 0$.

\subsection{Mass and Energy.}
It will be sometimes denoted in what follows 
\begin{align*}
&N_0(x)=\nu_0(x^2)(1+\nu_0(x^2))x^6;\,\,\,\,\,\,d\mu (x)=N_0(x)dx,\\
&\int _0^\infty \int _0^\infty  W(x, x')|f(\tau , x')-f(\tau , x)|^2x'^4x^4dx'dx=D(f(\tau )). 
\end{align*}
A first basic property  is the following,
\begin{prop}
\label{S4Ep}
Let $f_0$ and $f$ be as in Theorem \ref{S6WeakSolL2}. Then for all $p>1$,
\begin{align}
&\frac {d} {dt}\int _0^\infty|f(t, x)|^pd\mu (x) \le 0,\,\,\forall t>0.\label{S4EpE1}\\
&||f(t_2)|| _{ L^\infty(d\mu ) }\le ||f(t_1)|| _{ L^\infty(d\mu ) },\,\,\forall t_2>t_1\ge 0.\label{S4EpE2}
\end{align}
\end{prop}
\begin{proof}
Since $f$ satisfies (\ref{S4E140-4X} ), (\ref{S4E125}), and $x^6n_0(1+n_0)|f|^{p-2}f \in L^1(0, \infty)$, multiplication of both sides of (\ref{S4E125}) and integration over $(0, \infty)$ gives, using (\ref{S1ERRR10fx}) and the symmetry of $W(x, x')$ as follows,
\begin{align}
&\frac {d} {dt}\int _0^\infty \nu_0(x^2)(1+\nu_0(x^2))|f(t, x)|^px^6dx= \nonumber\\
&=\int _0^\infty \int _0^\infty  W(x, x')(f(\tau , x')-f(\tau , x)) |f|^{p-2}fx'^4x^4dx'dx\nonumber\\
&=-\frac {1} {2}\int _0^\infty \int _0^\infty  W(x, x')(f(\tau , x')-f(\tau , x)) \left(|f'|^{p-2}f'-|f|^{p-2}f\right) x'^4x^4dx'dx\le 0.
\label{S5EDEp} 
\end{align}
Since $d\mu $ is a non negative finite measure on $(0, \infty)$,
\begin{align*}
||f(t_2)|| _{ L^\infty(d\mu ) }=\lim _{ p\to \infty }||f(t_2)|| _{ L^p(d\mu ) }\le \lim _{ p\to \infty }||f(t_1)|| _{ L^p(d\mu ) }=||f(t_1)|| _{ L^\infty(d\mu ) }.
\end{align*}
\end{proof}
\begin{lem} 
\label{S5Etheta}
For all $\theta\in [0, 3)$,
\begin{align}
&(i)\qquad  ||f||_{L^\infty (d\mu )}\le \max \left(||f|| _{ L^\infty(1, \infty)}, \sup _{ x\in (0, 1) }x^\theta|f(x)|\right) \label{S5Etheta1}\\
&(ii)\qquad  ||x^6n_0(1+n_0)f|| _{ \infty }\le  ||f||_{L^\infty (d\mu )},\,\,\forall f\in L^\infty (d\mu )\cap L^\infty(0, \infty)
 \label{S5Etheta2}
\end{align}
\end{lem}
\begin{proof}
(i)\,\,For all  $C>0$ denote,
\begin{align*}
A_C=\{x\in (0, \infty); |f(x)|>C\}.
\end{align*}
Then, if $C>||f|| _{ L^\infty(1, \infty) }$, the Lebesgue measure of $A_c\cap (1, \infty)$ is zero and
\begin{align*}
d\mu (A_c)&=\int  _{ A_C\cap (0, 1) } n_0(1+n_0)x^6dx+\int  _{ A_C\cap (1, \infty) } n_0(1+n_0)x^6dx\\
&=\int  _{ A_C\cap (0, 1) } n_0(1+n_0)x^6dx
\end{align*}
On the other hand,  
\begin{align*}
\int  _{ A_C\cap (0, 1) } n_0(1+n_0)x^6dx&\le \sup _{ x>0 }\left(n_0(1+n_0)x^4 \right)\int  _{ A_C\cap (0, 1) }x^2dx\\
&\le \sup _{ x>0 }\left(n_0(1+n_0)x^4 \right)\frac {1} {C}\int  _{ A_C\cap (0, 1) }|f(x)|x^2dx\\
&\le\sup _{ x>0 }\left(n_0(1+n_0)x^4 \right)\frac {1} {C} \int _{ A_C \cap (0, 1) }x^\theta |f(x)|dx
\end{align*}
Then, since $x^\theta f\in L^1(0, 1)$, if $C>\sup _{ x\in (0, 1) }x^\theta |f(x)|$,
\begin{equation*}
 \int _{ A_C \cap (0, 1) }x^\theta |f(x)|dx=0.
\end{equation*}
It follows that for all $C>\max(||f|| _{ L\infty(1, \infty) }, \sup _{ x\in (0, 1) }x^\theta |f(x)|)$,
\begin{align*}
|f(x)|\le C,\,\,\,\hbox{except in a set of zero measure},
\end{align*}
and this proves (\ref{S5Etheta1}).

(ii) Let us denote now,
\begin{equation*}
B_C=\{x>0;\, x^6n_0(1+n_0)|f(x)|\ge C\}
\end{equation*}
and suppose that $C>||f|| _{ L^\infty(d\mu ) }$. Then
\begin{align*}
m(B_C)&=\int  _{ B_C }dx\le \frac {1} {C}\int  _{ B_C } x^6n_0(1+n_0)|f(x)|dx\\
&\le  \frac {||f||_\infty} {C}\int  _{ B_C } x^6n_0(1+n_0)dx=  \frac {||f||_\infty} {C}d\mu (B_C)=0.
\end{align*}
\vskip -0.6cm 
\end{proof}
The following easily follows now,
\begin{cor}
\label{SeECf}
Let $f_0$ and $f$ be as in Theorem \ref{S6WeakSolL2}. Then, 
\begin{align*}
\frac {d} {dt} \int _0^\infty n_0(x)(1+n_0(x)f(t , x)x^6dx=0,\,\,\forall t>0.
\end{align*}
\end{cor}

\begin{proof} 
Since  $n_0(1+n_0)x^6\in C_0^1([0, \infty))$, identity  (\ref{S4cor1E2}) with  $\varphi =n_0(1+n_0)x^6$ gives, by the definition of $\mathcal L$ (cf. (\ref{S1ERRR10fx})),
\begin{align*}
\frac {d} {dt}\int _0^\infty n_0(1+n_0)x^4f(\tau , x)x^2dx=\int _0^\infty n_0(1+n_0)x^4\mathcal L(f(\tau ))(x)dx\\
=\int _0^\infty \int _0^\infty  W(x, x')(f(\tau , x')-f(\tau , x))x'^4dx'x^4dx.
\end{align*}
The result follows from the symmetry of the kernel $W(x, y)$. 
\end{proof}

The following property will show the boundedness of the variation of the mass.
\begin{prop}
\label{S5EBN1}
Let $f_0$ and $f$ be  as in Theorem \ref{S6WeakSolL2}. Then  for all $t>0$ and all $p>3$,
\begin{align*}
&\int _0^\infty n_0(1+n_0)x^4|f(t, x)|dx\le C_p\max \left(||f_0|| _{ L^\infty(1, \infty)}, \sup _{ x\in (0, 1) }x^\theta|f_0(x)|\right)^{\frac {p-1} {p}}\!\!\!  ||f_0||^{1/p} _{ L^1(d\mu ) }\\
&C_p=\left(\int _0^\infty n_0(1+n_0) x^{\left(\frac {4p} {p-1}-\frac {6} {p-1} \right) }dx\right)^{\frac {p-1} {p}}
\end{align*}
\end{prop}
\begin{proof}
By Holder's inequality,
\begin{align*}
\int _0^\infty n_0(1+n_0)x^4|f(t, x)|dx\le \left( \int _0^\infty n_0(1+n_0)x^6|f(t, x)|^pdx\right)^{\frac {1} {p}}\times \\
\times \left(\int _0^\infty n_0(1+n_0) x^{\left(4-\frac {6} {p} \right) \frac {p} {p-1}}dx \right)^{\frac {p-1} {p}}\\
=||f(t)|| _{ L^p(d\mu ) }\left(\int _0^\infty n_0(1+n_0) x^{\left(\frac {4p} {p-1}-\frac {6} {p-1} \right) }dx \right)^{\frac {p-1} {p}}.
\end{align*}
If $p>3$, 
\begin{align*}
&\frac {4p} {p-1}-\frac {6} {p-1}>3,\,\,\,\hbox{and then},\\
&\int _0^\infty n_0(1+n_0) x^{\left(\frac {4p} {p-1}-\frac {6} {p-1} \right) }dx=C^{\frac {p} {p-1}}_p<\infty,\\
&\lim _{ p\to 3 }C_p=\infty.
\end{align*}
It follows by Proposition \ref{S4Ep},  for all $t> 0$,
\begin{equation*}
\int _0^\infty n_0(1+n_0)x^4|f(t, x)|dx\le C_p ||f(t)|| _{ L^p(d\mu ) }\le C_p ||f_0|| _{ L^p(d\mu ) }
\end{equation*}
Since,
\begin{align*}
||f_0|| _{ L^p(d\mu ) } \le ||f_0|| _{L^ \infty(d\mu ) }^{\frac {p-1} {p}}
||f_0||^{1/p} _{ L^1(d\mu ) }
\end{align*}
it follows from Lemma \ref{S5Etheta} that
\begin{align*}
||f_0|| _{ L^p(d\mu ) } &\le \max \left(||f_0|| _{ L^\infty(1, \infty)}, \sup _{ x\in (0, 1) }x^\theta|f_0(x)|\right)^{\frac {p-1} {p}}  ||f_0||^{1/p} _{ L^1(d\mu ) }
\end{align*}
and the Proposition follows.
\end{proof}
It immediately follows from Proposition \ref{S5EBN1},
\begin{cor}
\label{SeECM}
Let $f_0$ and $f$ be as in Theorem \ref{S6WeakSolL2}. Then, for $N(\tau )$ defined in (\ref{S2ENP}), and all $p>3$,
\begin{align*}
N(\tau )\le  C_p\max \left(||f_0|| _{ L^\infty(1, \infty)}, \sup _{ x\in (0, 1) }x^\theta|f_0(x)|\right)^{\frac {p-1} {p}}\!\!\!  ||f_0||^{1/p} _{ L^1(d\mu ) },\,\,\,\forall \tau >0,
\end{align*} 
where $C_p$ is given in Proposition \ref{S5EBN1}.
\end{cor}
The following property also follows from similar arguments.
\begin{prop}
\label{S5P56}
Suppose $f_0$ and $f$ are as in Theorem \ref{S6WeakSolL2}, $C>0$ is a constant and $f_0\le C$. Then $f(t)\le C$  for all $t>0$.
\end{prop}
\begin{proof}
Since $\mathcal L (C)\equiv 0$ it  follows 
\begin{equation*}
\frac {\partial (f(t, x)-C)} {\partial t}=\mathcal L\Big(f(t)-C\Big)(x).
\end{equation*}
If we multiply the equation by $n_0(1+n_0)x^6f_C(t)^+$, with $f_C(t)=(f(t)-C)$
\begin{align*}
\frac {d} {dt}\int _0^\infty\!\!\! n_0(1+n_0)|f_C^+(t, x)|^2x^6dx=\int _0^\infty\int _0^\infty W(x, y)(f_C(t, y)-f_C(t, x))f_C^+(t, x)x^4y^4dydx\\
=-\int _0^\infty\int _0^\infty W(x, y)(f_C(t, y)-f_C(t, x))f^+(t, y)x^4y^4dydx\\
=-\frac {1} {2}\int _0^\infty\int _0^\infty W(x, y)(f_C(t, y)-f_C(t, x))(f_C^+(t, y)-f_C^+(t, x))x^4y^4dydx
\end{align*}
If $f_C(t, y)>f_C(t, x)$ and $f_C^+(t, y)\le f_C^+(t, x)$ then $f_C(t, y)\le 0$ because if $f_C(t, y)>0$ we would have $f_C^+(t, y)=f_C(t, y)>f(t, x)=f_C^+(t, x)$, a contradiction. But this implies 
$f_C^+(t, x)=f_C^+(t, y)=0$. On the other hand, if  $f_C(t, y)<f_C(t, x)$ and $f_C^+(t, y)\ge f_C^+(t, x)$ then, $f_C(t, x)<0$ because if $f_C(t, x)>0$ then $f_C^+(t, x)=f_C(t, x)>0$ then, as before we would have,
$f_C^+(t, y)\ge f_C^+(t, x)>0$ and so $f_C(t, y)=f_C^+(t, y)$ but this would give $_Cf^+(t, y)=f_C(t, y)<f_C(t, x)=f_C^+(t, x)$, which is a contradiction. So in that case again $f_C^+(t, x)=f_C^+(t, y)=0$. We deduce,
\begin{align*}
&\frac {d} {dt}\int _0^\infty n_0(1+n_0)|f_C^+(t, x)|^2x^6dx\le 0\\
&\int _0^\infty n_0(1+n_0)|f_C^+(t, x)|^2x^6dx= 0,
\end{align*}
from where $f_C^+(t)(x)=(f(t, x)-C)^+=0\,\, a. e.$ and the Proposition follows.
\end{proof}
The following Corollary immediately follows,
\begin{cor}
\label{S4CBorn1}
Suppose that $f$ and $g$ are two solutions of (\ref{S4E125})  given by Theorem  \ref{S6WeakSolL2} with  initial data $f_0$ and $g_0$  satisfying the hypothesis of Theorem  \ref{S6WeakSolL2}  and such that $f_0\le g_0$. Then, $f(t)\le g(t)$ for all $t>0$.
\end{cor}
And let us also deduce,
\begin{cor}
\label{S4CBorn2}
Suppose that $f_0$ and $f$ are as in Theorem \ref{S6WeakSolL2}. Then, for all $\delta >0$ and all $t\ge \delta $,
\begin{align*}
||f(t)||_\infty\le C(\theta)\Big(||f_0|| _{ L^\infty (1, \infty) }+ \delta ^{-\theta} \sup _{ 0\le y\le 1} y^\theta |f_0(y)|+||f_0||_1\Big).
\end{align*}
\end{cor}
\begin{proof}
By  Theorem \ref{S6WeakSolL2}, for all$\delta >0$, $f(\delta )\in L^\infty(0, \infty)\cap L^1(0, \infty)$ and,
\begin{align*}
||f(\delta )||_\infty\le C(\theta)\Big(||f_0|| _{ L^\infty (1, \infty) }+ \delta ^{-\theta} \sup _{ 0\le y\le 1} y^\theta |f_0(y)|+||f_0||_1\Big).
\end{align*}
Since the constant $||f(\delta )||_\infty$ is a solution of (\ref{S4E125}), the result  follows by  Corollary \ref{S4CBorn1}.
\end{proof}

Our next result concerns the long time behavior of  the solution $f$.
The $L^2(d\mu )$ norm and $D(f(\tau ))$ play here the usual roles of entropy and entropy's dissipation through the identity (cf. (\ref{S5EDEp})),
\begin{equation}
\label{S5E1}
\frac {d} {dt}\int _0^\infty|f(t, x)|^2d\mu (x)=-D(f(\tau ))
\end{equation}
Consider the convex, proper lower semi continuous function
\begin{align*}
j(u, v)=
\begin{cases}
|u-v|^2,\,\,\hbox{if}\,\,u>0, v>0,\,\\
0,\,\,\hbox{if}\,\,u\le 0, v\le 0\\
\infty\,\,\,\hbox{elsewhere}
\end{cases}
\end{align*}
 and define the following regularized kernel for all $n\in \NN\setminus\{0\}$,
\begin{equation*}
W_n(x, x')=\left(\frac {\sinh (x^2+y^2)-\sinh|x^2-y^2 |} {\sinh (x^2+y^2)\sinh|x^2-y^2 | +\frac {1} {n}}\right) \frac {1} {xy (\sinh x^2)\,(\sinh y^2)}
\end{equation*}
so that,  $d\sigma _n(x, x')=W_n(x, x') x^4x'^4dx'dx$  is now a bounded measure on $\RR^2$.  Consider then for any pair of functions $f, g$ defined on $(0, \infty)$ the function $U$ defined on $(0, \infty)^2$ as,
\begin{equation*}
U(x, y)=(f(x), g(y))\in \RR^2
\end{equation*}
and denote,
\begin{align*}
&J_n(U)=
\begin{cases}
\displaystyle {\,\,\,\iint\limits  _{ (0, \infty)^2 } j(f(x), g(x')) d\sigma _n(x'x),\,\,\,\hbox{if}\,\, j(f, g')\in L^1(d\sigma _n)}\\
\infty,\,\,\hbox{elsewhere.}
\end{cases}
\end{align*}
It follows that $J_n$ is convex and l.s.c. on $L^2(d\mu )$.

\begin{prop}
\label{SeEC2f}
Let $f_0$ and $f$ be as in Theorem \ref{S6WeakSolL2}. Then, for all $\varphi \in L^2(d\mu)$,
\begin{equation*}
\lim _{ \tau \to  \infty } \int _0^\infty f(\tau, x )\varphi (x)d\mu (x)=\frac {E(0)} {\int _0^\infty d\mu (x)}\int _0^\infty \varphi (x)d\mu (x)
\end{equation*}
\end{prop}
\begin{rem}
Notice that, since $d\mu $ is a non negative bounded measure on $(0, \infty)$, $L^2(d\mu )\subset L^1(d\mu )$. Corollary \ref{SeEC2f} shows the weak convergence,
\begin{align*}
n_0(1+n_0)x^6\,f(\tau )\underset{\tau \to \infty}{\rightharpoonup}  \frac {E(0)} {\int _0^\infty d\mu (x)},\,\,\hbox{weakly in}\,L^2(0, \infty).
\end{align*}
\end{rem}

\begin{proof}
Consider  any sequence $\{\tau _k\} _{ k\in \NN }$ where $\tau _k\to \infty$ as $k\to \infty$ and  define
\begin{equation}
\label{S5ETk}
f_k(\tau )=f(\tau +\tau _k).
\end{equation}
By (\ref{S5E1}), for all $T>0$,
\begin{align*}
\frac {1} {2}\int _0^TD(f(s))dt=-\int _0^\infty |f(T, x)|^2d\mu (x)+\int _0^\infty |f(0, x)|d\mu (x)\le\int _0^\infty |f(0, x)|^2d\mu (x)
\end{align*}
from where we deduce  $D(f(s))\in L^1(0, \infty)$ and
\begin{align*}
\int _0^T D(f_k(t))dt= \int  _{ t_k }^{t_k+T} |Df(t)|dt\to 0,\,\,\hbox{as}\,\,k\to \infty.
\end{align*}
Then,  for all $T>0$,
\begin{equation}
\label{S5EDE1}
\lim _{ k\to \infty }\int _0^T D_n(f_k(t))dt=0,\,\,\,\forall n.
\end{equation}
and since $D_n(f_k(t))\ge 0$ for all $t>0$,
\begin{equation}
\label{S5EDE2}
\lim _{ k\to \infty } D_n(f_k(t))=0,\,\,\,\hbox{for a. e. }t\in (0, T).
\end{equation}
On the other hand, by (\ref{S5E1}) again,  the sequence $\{f_k\} _{ k\in \NN }$ is bounded in $L^\infty ((0, \infty); L^2(d\mu ))$ and there exists a subsequence still denoted  $\{t_k\} _{ k\in \NN }$, such that $t_k\to \infty$ as $k\to \infty$  and $g\in L^\infty ((0, \infty); L^2(d\mu ))$ satisfying
\begin{equation}
f_k \underset{k\to \infty }{\rightharpoonup}  g,\,\,\,\hbox{in the weak}^*\hbox{ topology of }  \,\,L^\infty ((0, \infty); L^2(d\mu )).\label{S5EDE3}
\end{equation}
It then follows by the lower semicontinuity of $D_n$,
\begin{equation*}
\int _0^T D_n(g(s))ds\le \liminf _{ k\to \infty }\int _0^T D_n(f_k(s))ds=0.
\end{equation*}
Then,
\begin{equation*}
D_n(g(s)=\int _0^\infty \int _0^\infty  W_n(x, x')|g(s , x')-g(s , x)|^2x'^4x^4dx'dx=0\\
\end{equation*}
and, for almost every  $s \in (0, T)$ the measure $|g(s , x')-g(s , x)|^2x'^4x^4dx'dx$ is concentrated on the diagonal $\{(x, x')\in (0, \infty)^2;\,x'=x\}$. Since $g(s )\in L^2(d\mu )$ for almost every $s\in (0, T)$, it follows that $g=C_*$ for some constant $C_*\in \RR$.  Since $\1 _{ (0, 1) }\in L^1(0, \infty; L^2(d\mu ))$, it follows that:
\begin{align*}
C\int _0^\infty d\mu (x)&=\lim _{ k\to \infty } \int _0^1\int _0^\infty f_k(t, x)d\mu (x)dt\\
&=\lim  _{ k\to \infty  }\int _0^1\int _0^\infty f_k(t, x)n_0(x)(1+n_0(x)x^6dxdt\\
&=\lim  _{ k\to \infty  }\int _0^1 E(t+t_k)dt=E(0)
\end{align*}
and then,
\begin{equation*}
C_*=\frac {E(0)} {\int _0^\infty d\mu (x)}.
\end{equation*}

For all $\varphi \in C_c^1(0, \infty)$, the function:
\begin{equation*}
t\to \int _0^\infty f(t, x)\varphi (x)n_0(x)(1+n_0(x))x^6dx\in C(0, \infty; \RR)\cap L^\infty(0, \infty).
\end{equation*}
and,
\begin{align*}
&\left|\frac {d} {dt}\int _0^\infty f(t, x)\varphi (x)n_0(x)(1+n_0(x))x^6dx\right| \le \\
&\le
\int_0^\infty \left|\int_0^\infty W(x, y)(f(t, y)-f(t, x))y^4dy\right|\varphi (x)dx\\
&\le ||\mathcal L(f)(t)|| _{ \infty }||\varphi ||_1\in L^\infty _{ loc }(0, \infty).
\end{align*}
There exists then a sequence of $\{t_j\} _{ k\in \NN }$,  and $C(\varphi )$ such that,
\begin{align*}
\lim _{ j\to \infty }\int _0^\infty f(t_j)\varphi (x)d\mu (x)=C(\varphi ).
\end{align*}
Since $\1 _{ (0, 1) }(t)\varphi (x) \in L^1(0, \infty; L^2(d\mu ))$, by  (\ref{S5EDE3})
\begin{align*}
\lim _{ j\to \infty }\int _0^1\int _0^\infty f(t_j)\varphi (x)d\mu (x)dt=C_*\int _0^1 \int _0^\infty \varphi (x)d\mu (x)dt=C_*\int _0^\infty \varphi (x)d\mu
\end{align*}
and 
\begin{equation*}
C(\varphi )=C_*\int _0^\infty \varphi (x)d\mu.
\end{equation*}
\end{proof}

The following auxiliary result is used in the proof of the next Proposition.
\begin{lem}
\label{SeElem1}
For all $\varepsilon >0$, there exists a constant $C_\varepsilon >0$ such that
\begin{align*}
&\forall (x, y)\in \left(\varepsilon , \frac {1} {\varepsilon }\right)\times \left(\varepsilon , \frac {1} {\varepsilon }\right),\\
&W(x, y)\ge  \left(\frac{C _{\varepsilon  }} {\sinh (x^2+y^2)\sinh|x^2-y^2 | }\right) \frac {1} {xy (\sinh x^2)\,(\sinh y^2)}
\end{align*}
\end{lem}
\begin{proof}
There certainly exists $C _{1, \varepsilon  } >0$ such that
\begin{align*}
x^2+y^2-|x^2-y^2|\ge C' _{ \varepsilon  },\,\,\,\forall (x, y)\in \left(\varepsilon , \frac {1} {\varepsilon }\right)\times \left(\varepsilon , \frac {1} {\varepsilon }\right)
\end{align*}
Then, by continuity,
\begin{align*}
\sinh (x^2+y^2)-\sinh|x^2-y^2 |\ge C _{ \varepsilon  },\,\, \forall (x, y)\in \left(\varepsilon , \frac {1} {\varepsilon }\right)\times \left(\varepsilon , \frac {1} {\varepsilon }\right)
\end{align*}
and the result follows.
\end{proof}

\begin{prop}
\label{SeEC2f2}{SeEC2f3}
Let $f_0$ and $f$ be as in Theorem \ref{S6WeakSolL2}. Then,
\begin{equation*}
\lim _{ \tau \to  \infty } \int _0^\infty |f(t, x)-C^*|^2d\mu (x)=0
\end{equation*}
\end{prop}
\begin{proof} 
It follows from Lemma \ref{SeElem1}, for all $\varepsilon >0$, for all  $x\in (\varepsilon /2, 2/\varepsilon )$ and $y\in (\varepsilon /2, 2/\varepsilon )$
\begin{align*}
x^4 y^4W_n(x, y)\ge  \left(\frac{C _{\varepsilon  }} {\sinh (x^2+y^2)\sinh|x^2-y^2 |+1/n }\right) \frac {x^3y^3} { (\sinh x^2)\,(\sinh y^2)}\\
=\frac {x^6} {(\sinh x^2)^2}\left\{
\left(\frac{C _{\varepsilon  }} {\sinh (x^2+y^2)\sinh|x^2-y^2 |+1/n }\right) \frac {y^3(\sinh x^2)} { x^3\,(\sinh y^2)}\right\}\\
\ge \frac {x^6} {(\sinh x^2)^2}\left\{
\left(\frac{C _{\varepsilon  }} {\sinh (x^2+y^2)^2+1/n }\right) \frac {y^3(\sinh x^2)} { x^3\,(\sinh y^2)}\right\}\\
\ge \frac {x^6} {(\sinh x^2)^2}\left\{
\left(\frac{C _{\varepsilon  }} {\sinh (2\varepsilon ^{-2})^2+1/n }\right) \frac {\varepsilon ^6(\sinh \varepsilon ^2)} {(\sinh \varepsilon ^{-2})}\right\}
=\gamma  _{ \varepsilon , n } \frac {x^6} {(\sinh x^2)^2}
\end{align*}
Let then be $\theta_\varepsilon \in C_0(0, \infty)$, $\theta_\varepsilon (x)=1$ for $x\in (\varepsilon , 1/\varepsilon )$ and $\theta_\varepsilon (x)=0$ if $x\in (0, \varepsilon /2)$ or $x>2/\varepsilon )$.   Consider   $\{f_k\} _{ k\in \NN }$ the sequence constructed in  (\ref{S5ETk}).

\begin{align}
\gamma  _{ \varepsilon , n }&\int _0^T\int _0^\infty \theta_\varepsilon (x) n_0(1+n_0)x^6 |f_k(t, x)-C_*|^2dx\le \nonumber\\
&\le \int _0^T\int _0^\infty \int _0^\infty \theta_\varepsilon (x) \theta_\varepsilon (y)W_n(x, y)\left( |f_k(t, x)-C_*|^2+ |f_k(t, y)-C_*|^2\right)x^4y^4dxdydt \nonumber\\
&\le \int _0^TD(f_k(t))dt+2\int _0^\infty \int _0^\infty \theta_\varepsilon (x) \theta_\varepsilon (y) W_n(x, y)\times \nonumber \\
&\times \left(f_k(t, x)f_k(t, y)+C_*^2-C_*f_k(t, x)-C_*f_k(t, y)\right)x^4y^4dxdydt \label{S5EC2f21}
\end{align}
In order to prove that the last term in the right hand side tends to zero  the following Lemma is needed,
\begin{lem}
\label{S5LP1}
For all $\varphi \in C_c(0, \infty)$ and all $T>0$,
\begin{align*}
\lim _{ k \to \infty } \int _0^T \left|\int _0^\infty f_k(t, x)\varphi (x) d\mu (x)-C_* \int _0^\infty\varphi (x) d\mu (x) \right|dt=0
\end{align*}
\end{lem}
\begin{proof}
For all $\varphi \in C_0(0, \infty)$,
\begin{align*}
\int_0^\infty \int_0^\infty \mathcal U(x, x')(f'-f)\varphi x^4x'^4dx'dx=-\int_0^\infty \int_0^\infty \mathcal U(x, x')(f'-f)\varphi 'x^4x'^4dx'dx\\
=\frac {1} {2}\int_0^\infty \int_0^\infty \mathcal U(x, x')(f'-f)(\varphi -\varphi') x^4x'^4dx'dx=\\
=\frac {1} {2}\int_0^\infty \int_0^\infty \mathcal U(x, x')f'(\varphi -\varphi') x^4x'^4dx'dx-\frac {1} {2}\int_0^\infty \int_0^\infty \mathcal U(x, x')f(\varphi -\varphi') x^4x'^4dx'dx\\
=\frac {1} {2}\int_0^\infty \int_0^\infty \mathcal U(x, x')f'(\varphi -\varphi') x^4x'^4dx'dx-
\frac {1} {2}\int_0^\infty \int_0^\infty \mathcal U(x, x')f'(\varphi' -\varphi) x^4x'^4dx'dx\\
=\int_0^\infty \int_0^\infty \mathcal U(x, x')f'(\varphi -\varphi') x^4x'^4dx'dx
\end{align*}
\begin{align*}
I_k(t )=&\left|\frac {d} {dt}\int _0^\infty n_0(1+n_0)x^4f_k(t, x)\varphi (x)dx\right|\le \\
& \hskip 3cm \le \int_0^\infty \int_0^\infty \mathcal U(x, x') \left|\varphi (x')-\varphi (x) \right| |f_k(t, x)|x^4x'^4dx'dx
\end{align*}
If $\rho >0$ and $R>0$  are such that ${\rm supp } \varphi  \subset (\rho , R)$, then $\left|\varphi (y)-\varphi (x) \right|=0$ for $x\in (0, \rho )$ and $y\in (0, \rho )$. If on the other hand, $x\ge \rho $ or $y\ge \rho $,
\begin{align*}
\sinh|x^2-y^2|=\sinh \left((x+y)|x-y| \right)\ge \sinh (2\rho |x-y|)\\
\frac{|\varphi (y)-\varphi (x)|}{\sinh|x^2-y^2|}\le C_\varphi \frac {|\varphi (y)-\varphi (x)|} {\sinh (2\rho  |x-y|)}
\end{align*}
If $x>R$ and $y>R$,   then $\left|\varphi (y)-\varphi (x) \right|=0$ again and therefore
\begin{align*}
&I_k(t )\le\\
&\le \int_0^R \int_0^\infty \left(\frac {\sinh (x^2+y^2)-\sinh|x^2-y^2 |} {\sinh (x^2+y^2)}\right) \frac{|\varphi (y)-\varphi (x)|}{\sinh|x^2-y^2 |}
\frac { |f_k(t, x)|x^3x'^3dx'dx} { (\sinh x^2)\,(\sinh y^2)}+\\
&+ \int_0^\infty \int_0^R \left(\frac {\sinh (x^2+y^2)-\sinh|x^2-y^2 |} {\sinh (x^2+y^2)}\right) \frac{|\varphi (y)-\varphi (x)|}{\sinh|x^2-y^2 |}
\frac { |f_k(t, x)|x^3x'^3dx'dx} { (\sinh x^2)\,(\sinh y^2)}\\
&=I _{ k, 1 }(t)+I _{ k,2 }(t)
\end{align*}
The term $I _{ k, 1 }$ is easily estimated as follows,
\begin{align*}
&I _{ k,1 }(t )\le\\
&=\int_0^R \int_0^\infty \left(\frac {\sinh (x^2+y^2)-\sinh|x^2-y^2 |} {\sinh (x^2+y^2)}\right) \frac{|\varphi (y)-\varphi (x)|}{\sinh|x^2-y^2 |}
\frac { |f_k(t, x)|x^3x'^3dx'dx} { (\sinh x^2)\,(\sinh y^2)}\\
&\le C_\varphi \int_0^R \int_0^\infty \frac { |f_k(t, x)|x^3dx} { \sinh x^2}  \frac {y^3dy} {\sinh y^2} 
\end{align*}
Using H\"older's inequality and   Proposition \ref{S4Ep}
\begin{align*}
I _{ k,1 }(t)&\le C_\varphi \sqrt R  \left(\int _0^R \frac {|f(t, x)|^2x^6dx} {(\sinh x^2)^2}\right)^{1/2}\int_0^\infty\frac {y^3dy}{\sinh y^2}\\
&\le C'_\varphi  \left(\int _0^\infty \frac {|f_0(x)|^2x^6dx} {(\sinh x^2)^2}\right)^{1/2}.
\end{align*}
The integral $I _{ k, 2 }$ may be split again as follows,
\begin{align*}
I _{ k, 2 }=I _{ k, 2 , 1}+I _{ k, 2 , 2},
\end{align*}
where,
\begin{align*}
I _{ k, 2 , 1}\le \int_0^{2R} \int_0^R  \frac{|\varphi (y)-\varphi (x)|}{\sinh|x^2-y^2 |}
\frac { |f_k(t, x)|x^3x'^3dx'dx} { (\sinh x^2)\,(\sinh y^2)}
\end{align*}
and $I _{ k, 2 , 1}$ is then estimated as $I _{ k,1 }$.   The estimate of the term $I _{ k, 2 , 2}$ uses that when $y\le R<x/2$ then $\sinh |x^2-y^2|\ge \sinh (3x^2/4)$ as follows
\begin{align*}
I _{ k, 2 , 1}\le \int_{2R} ^\infty \int_0^R\frac{|\varphi (y)-\varphi (x)|}{\sinh|x^2-y^2 |}
\frac { |f_k(t, x)|x^3x'^3dx'dx} { (\sinh x^2)\,(\sinh y^2)}\\
\le C_\varphi \int_{2R} ^\infty \int_0^R  \frac { |f_k(t, x)|x^3x'^3dx'dx} {\sinh(3x^2/4) (\sinh x^2)\,(\sinh y^2)}
\end{align*}
and using Holder's inequality again,
\begin{align*}
I _{ k, 2 , 1}
\le &C_\varphi   \left(\int _{2R} ^\infty \frac {|f_k(t, x)|^2x^6dx} {(\sinh x^2)^2}\right)^{1/2}\left(\int _{2R} ^\infty \frac {dx} {(\sinh (3x^2/4)^2}\right)^{1/2} \int_0^R  \frac {y^3dy} {\sinh y^2}
\end{align*}
We deduce, that 
\begin{align*}
\int _0^T\left|\frac {d} {dt}\int _0^\infty f_k(t, x)\varphi (x)d\mu (x) \right|dt\le C' _{ \varphi  } T \left(\int _0^\infty \frac {|f_0(x)|^2x^6dx} {(\sinh x^2)^2}\right)^{1/2}.
\end{align*}
We have then, for all $T>0$,
\begin{align*}
\int _0^\infty f_k(t, x)\varphi (x)d\mu (x) \in W^{1, 1}(0, T),\,\,\,
\end{align*}
and by compactness of the injection $W^{1, 1}(0, T)\subset L^1(0, T)$, there exists a sequence $t_j \underset{j\to \infty }{\to}~\infty$ and $h\in L^1(0, T)$ satisfying
\begin{align*}
\lim _{ k\to \infty }\int _0^T\left|\int _0^\infty f_k(t, x)\varphi (x)d\mu (x) -h(t)\right|dt=0.
\end{align*}
But, since by Corollary \ref{SeEC2f},
\begin{align*}
\lim _{ k\to \infty }\int _0^T \int _0^\infty f_k(t, x)\varphi (x)d\mu (x)dt=C_* T\int _0^\infty \varphi (x)d\mu (x) 
\end{align*}
we deduce that,
\begin{align*}
\int _0^Th(t)dt=C_* T\int _0^\infty \varphi (x)d\mu (x) ,\,\,\forall T>0
\end{align*}
and by the fundamental Theorem of Calculus,
\begin{align*}
C_* T\int _0^\infty \varphi (x)d\mu (x) =\frac {d} {dT}\int _0^Th(t)dt=h(T).
\end{align*}
\vskip -0.6cm
\end{proof}
The second term in the right hand side of (\ref{S5EC2f21}) may be split  as follows
\begin{align*}
2&\int _0^\infty  \int _0^\infty \theta_\varepsilon (x) \theta_\varepsilon (y)\mathcal U_n(x, y)\times \nonumber \\
&\times \left(f_k(t, x)f_k(t, y)+C_*^2-C_*f_k(t, x)-C_*f_k(t, y)\right)x^4y^4dxdydt =A_1+A_2\\
A_1&=\int _0^T  \int _0^\infty  \int _0^\infty \theta_\varepsilon (x) \theta_\varepsilon (y)\mathcal U_n(x, y)
\frac {f_k(t, x)(f_k(t, y)-C_*)x^4y^4dxdydt} {xy (\sinh x^2)\,(\sinh y^2)}\\
A_2&=\int _0^T  \int _0^\infty  \int _0^\infty \theta_\varepsilon (x) \theta_\varepsilon (y)\mathcal U_n(x, y)
\frac {C_*(C_*-f_k(t, y))x^4y^4dxdydt} {xy (\sinh x^2)\,(\sinh y^2)}
\end{align*}
The first term may be written,
\begin{align*}
A_1=&\int _0^T\int_0^\infty  \theta_\varepsilon (y)\frac {y^3(f_k(t, y)-C_*)dy} {\sinh y^2}\times \\
&\times \int _0^\infty \theta_\varepsilon (x) \frac {f_k(t, x) x^3} {(\sinh x^2)}\left(\frac {\sinh (x^2+y^2)-\sinh|x^2-y^2 |} {\sinh (x^2+y^2)(\sinh|x^2-y^2 | +1/n)}\right)dxdt\\
\end{align*}
Since,
\begin{align*}
\int _0^\infty \theta_\varepsilon (x)& \frac {|f_k(t, x)| x^3} {(\sinh x^2)}\left(\frac {\sinh (x^2+y^2)-\sinh|x^2-y^2 |} {\sinh (x^2+y^2)(\sinh|x^2-y^2 | +1/n)}\right)dx\le \\
&\le n\int _0^\infty \theta_\varepsilon (x) \frac {|f_k(t, x)| x^3} {\sinh x^2}\\
&\le \frac {n} {\varepsilon ^2}\left(\int _0^\infty |f_k(t, x)|^2d\mu   \right)^{1/2}\le \frac {n} {\varepsilon ^2}\left(\int _0^\infty |f_0(x)|^2d\mu   \right)^{1/2}
\end{align*}
we deduce,
\begin{align*}
|A_1|&\le \frac {n} {\varepsilon ^2}||f_0|| _{ L^2(d\mu ) }\int _0^T \left|\int_0^\infty \theta_\varepsilon (y) \frac {y^3(f_k(t, y)-C_*)dy} {\sinh y^2}\right|dt\\
&= \frac {n} {\varepsilon ^2}||f_0|| _{ L^2(d\mu ) }\int _0^T \left|\int_0^\infty  (f_k(t, y)-C_*) \theta_\varepsilon (y)\frac {\sinh y^2}{y^3}d\mu (y) \right|dt
\underset{k\to \infty }{\to} 0.
\end{align*}
The same argument shows that $A_2$ tends to zero too  as $k$ goes to $\infty$ and this shows that, for all $\varepsilon >0$ fixed,
\begin{align}
\lim _{ k\to \infty }\int _0^T\int _0^\infty \theta_\varepsilon (x) n_0(1+n_0)x^6 |f_k(t, x)-C_*|^2dx=0. \label{S5EC2f22}
\end{align}
On the other hand, for $\varepsilon >0$,
\begin{align}
\int _0^T\int _0^\varepsilon  n_0(1+n_0)x^6 |f_k(t, x)-C_*|^2dx=\int _0^T\int _0^\varepsilon   |f_k(t, x)-C_*|^2d\mu (x) \nonumber\\
\le  T  ||f_k(t)-C_*|| _{ L^2(d\mu ) }\int _0^\varepsilon d\mu (x)\le T \left(C_*+ ||f_0|| _{ L^2(d\mu ) }\right)\int _0^\varepsilon d\mu (x) \label{S5EC2f23}
\end{align}
and similarly,
\begin{align}
\int _0^T\int _{1/\varepsilon} ^\infty n_0(1+n_0)x^6 |f_k(t, x)-C_*|^2dx
\le  T \left(C_*+ ||f_0|| _{ L^2(d\mu ) }\right)\int _{1/\varepsilon}^\infty d\mu (x) \label{S5EC2f24}
\end{align}
We deduce  from (\ref{S5EC2f22})-(\ref{S5EC2f24}),
\begin{align*}
\lim _{ k\to \infty }\int _0^T \int _0^\infty |f(t, x)-C^*|^2d\mu (x)dt=0.
\end{align*}
Since by Proposition \ref{S4Ep}
\begin{equation*}
\int _0^\infty |f(t, x)-C^*|^2d\mu (x)\ge \int _0^\infty |f(T, x)-C^*|^2d\mu (x),\,\,\forall t\in (0, T),
\end{equation*}
and Corollary \ref{SeEC2f2} follows since the right hand sides of  (\ref{S5EC2f23}) and (\ref{S5EC2f24}) may be done arbitrarily small for $\varepsilon >0$ small enough.
\end{proof}

\begin{cor}
\label{SeEC2f3}
Let $f_0$ and $f$ be as in Theorem \ref{S6WeakSolL2}. Then,
\begin{equation*}
\lim _{ \tau \to  \infty } \int _0^\infty |f(\tau , x)-C^*|x^4n_0(x)(1+n_0(x))dx=0
\end{equation*}
\end{cor}
\begin{proof}
Arguing as in Proposition \ref {S5EBN1}, for all $\varepsilon >0$ and $t>1$,
\begin{align*}
\int _0^\varepsilon n_0(1+n_0)x^4|f(t, x)-C_*|dx\le 
\left( \int _0^\varepsilon  n_0(1+n_0)x^6(|f(t, x)|^p+C_*^p)dx\right)^{\frac {1} {p}}\times \\
\times \left(\int _0^\varepsilon  n_0(1+n_0) x^{\left(4-\frac {6} {p} \right) \frac {p} {p-1}}dx \right)^{\frac {p-1} {p}}\\
\le \left(C+ ||f(1)|| _{ L^p(d\mu ) }\right)\left(\int _0^\varepsilon  n_0(1+n_0) x^{\left(\frac {4p} {p-1}-\frac {6} {p-1} \right) }dx \right)^{\frac {p-1} {p}}.
\end{align*}
For $p>3$, $n_0(1+n_0) x^{\left(\frac {4p} {p-1}-\frac {6} {p-1} \right)}\in L^\infty(0, \infty)$ and then
\begin{align*}
\lim _{ \varepsilon \to 0 }\int _0^\varepsilon n_0(1+n_0)x^4|f(t, x)-C_*|dx=0.
\end{align*}
On the other hand,
\begin{align*}
\int _\varepsilon^\infty  n_0(1+n_0)x^4|f(t, x)-C_*|dx&\le \left(\int _\varepsilon ^\infty x^6 n_0(x)(1+n_0(x)) |f(t, x)-C_*|^2dx \right)^{1/2}\times \\
&\hskip 4cm \times \left( \int _\varepsilon ^\infty n_0(1+n_0)x^2 dx\right)\\
&\le C_\varepsilon  \left(\int _0 ^\infty  |f(t, x)-C_*|^2d\mu (x) \right)^{1/2}
\end{align*}
and the Corollary \ref{SeEC2f3}  follows from Corollary \ref{SeEC2f2}.
\end{proof}

\subsection{The limit of  $f(t,x)$ as $x\to0$ for $t>0$.}
We show now the existence of the limit, 
\begin{equation}
b(t)=\lim _{ x\to 0 }f(t, x)\,\,\,\forall t>0,
\end{equation}
for all $f_0$ and  $f$ as in Theorem \ref{MainThm}, and describe its time evolution.  The  following property, proved in Proposition 1.3 of  \cite{m}, is used
\begin{align}
&\forall f_0\in L^1(0, \infty),\,\forall t>0,\,\,\,\,\lim _{ x \to 0 }S(t)f_0(x)=\ell(f_0; t) \in (-\infty, \infty) \label{S4L25E1}\\
&\ell(f_0; t)=A_1t^{-3}\int _0^tf_0(y)y^2dy+A_2t^{-4}\int _0^t f_0(y)y^3dy +\int _0^t f_0(y)\,b_1\!\!\left(\frac {t} {y} \right)\frac {dy} {y}\label{S4L25E2}
\end{align}
for some constants $A_1$, $A_2$ and a function $b_1$ such that $b _1(t)=\mathcal O( t^{-8})$ for $t>1$.
A slightly more precise information may be obtained and is shown here,  since it is of further interest.

\begin{prop}
\label{S6Limzero}
Suppose that  $f_0$ and  $f$ are as in Theorem  \ref{S6WeakSolL2}. Then for all $t>0$, and $\delta >0$ as small as desired,
\begin{align}
\label{S6LimzeroE0}
f(t, x)=&b(t)+\left(||f_0|| _{ L^\infty (1, \infty) }+||f_0||_1\right)\mathcal O\left(t^\delta x^{1-\delta } \right)+\nonumber\\
&+\sup _{ 0<x<1 }x^\theta |f_0(x)|\mathcal O\left( t^{\delta -\theta} x^{1-\delta } \right)
,\,\,x\to 0,
\end{align}
where,
\begin{align}
&b (t)=\ell (f_0; t)+\int _0^t \ell (F(f(s)); t-s)ds \label{S6LimzeroE1}\\
&\ell(F(f(s)); t-s)=\frac{A_1}{(t-s)^{3}}\int _0^{t-s}\!\!\!\!\!\!F(f(s, y))y^2dy+\frac{A_2}{(t-s)^{4}}\int _0^{t-s} \!\!\!\!\!\!F(f(s, y))y^3dy+\nonumber \\
&\hskip 8cm  +\int _0^{t-s} \!\!\!\!\!\! F(f(s, y))\,b_1\!\!\left(\frac {t-s} {y} \right)\frac {dy} {y} \label{S6LimzeroE2}
\end{align}
and  there exists a constant $C>0$ such that
\begin{align}
\label{S6LimzeroE25}
|b(t)|\le C|||f_0|||_\theta \left(t^{-\theta}+t \right),\,\,\forall t>0.
\end{align}
\end{prop}

\begin{proof} By construction, for all $t>0$ and $x>0$,
\begin{equation*}
f(t, x)=S(t)f_0(x)+\int _0^t S(t-s)(F(f(s))(x)ds.
\end{equation*}
By Proposition 1.5  in \cite{m},
\begin{align*}
&(S(t)f_0)(x)=\ell (f_0; t)+\left( t^{-2+\delta }\int _0^t|f_0(y)|dy+t^{5+\delta }\int _t^\infty \frac {|f_0(y)|dy} {y^7}\right)\mathcal O_\delta (x^{1-\delta})\\
&S(t-s)F(f(s))(x)=\ell(F(f(s)); t-s)+\Bigg( (t-s)^{-2+\delta }\int _0^{t-s}|F(f(s, y))|dy+\\
&+(t-s)^{5+\delta }\int _{t-s}^\infty \frac {|F(f(s, y))|dy} {y^7}\Bigg)\mathcal O_\delta (x^{1-\delta}).
\end{align*}
By (iv) in Proposition (\ref{S2AFF1})
\begin{align*}
 (t-s)^{-2+\delta }\int _0^{t-s}|F(f(s, y))|dy\le  (t-s)^{-2+\delta }\int _0^{t-s}(||f(s)||_1+||f(s)|| _{ L^\infty(1, \infty) }))dy\\
 \le (t-s)^{-1+\delta }(||f(s)||_1+||f(s)|| _{ L^\infty(1, \infty) }))\\
 (t-s)^{-1+\delta }\int _{t-s}^\infty \frac {|F(f(s, y))|dy} {y^7}\le (t-s)^5(||f(s)||_1+||f(s)|| _{ L^\infty(1, \infty) })\int  _{ t }^\infty \frac {dy} {y^7}\\
 = \frac {(t-s)^{-1+\delta }} {6}(||f(s)||_1+||f(s)|| _{ L^\infty(1, \infty) })
\end{align*}
and, by (\ref{S4E80})  for $T>0$,
\begin{align*}
\sup _{ x>1 } |f(t, x)|+||f(t)||_1\le 
C(T, \theta)\left(||f_0|| _{ L^\infty (1, \infty) }+t^{-\theta}\sup _{ 0<x<1 }x^\theta |f_0(x)| +||f_0||_1\right).
\end{align*}
\begin{align*}
&S(t-s)F(f(s))(x)=\ell(F(f(s)); t-s)+(||f(s)||_1+||f(s)|| _{ L^\infty(1, \infty) })\mathcal O\left((t-s)^{-1+\delta }x^{1-\delta } \right)\\
&= \ell(F(f(s)); t-s)+\left(||f_0|| _{ L^\infty (1, \infty) }+s^{-\theta}\sup _{ 0<x<1 }x^\theta |f_0(x)| +||f_0||_1\right)\times \\
&\times \mathcal O_T\left((t-s)^{-1+\delta }x^{1-\delta } \right)
\end{align*}
and then,
\begin{align*}
&\int _0^t S(t-s)F(f(s))(x)= \int _0^t \ell(F(f(s)); t-s)ds+\\
&+(||f_0|| _{ L^\infty (1, \infty) }+||f_0||_1)\mathcal O\left(t^\delta x^{1-\delta } \right)+\sup _{ 0<x<1 }x^\theta |f_0(x)|\mathcal O\left( t^{\delta -\theta} x^{1-\delta } \right)
\end{align*}
This shows (\ref{S6LimzeroE0}), (\ref{S6LimzeroE1}), (\ref{S6LimzeroE2}).
On the other hand, by property (iv) in Proposition \ref{S2AFF1}, there exists a constant $C$ that depends on $\theta$ such that,
\begin{align*}
|||F(f(s))||| _{ \theta }\le C|||f(s)|||_\theta,\,\,\,\forall s>0.
\end{align*}
Then, for $t\in (0, 1)$ and $s\in (0, t)$,
\begin{align*}
\int _0^t (t-s)^{-n}\int _0^{t-s}F(f(s, y))y^{n-1}dyds&\le
C|||f_0|||_\theta \int _0^t (t-s)^{-n}\int _0^{t-s}y^{n-1-\theta}dyds\\
&=C|||f_0|||_\theta t^{-\theta}.
\end{align*}
and for $t>1$,
\begin{align*}
\int _0^t (t-s)^{-n}\int _0^{t-s}F(f(s, y))y^{n-1}dyds=J_1+J_2+J_3\\
J_1=\int _0^{t-1} (t-s)^{-n}\int _0^1F(f(s, y))y^{n-1}dyds\\
J_2=\int _0^{t-1} (t-s)^{-n}\int _1^{t-s}F(f(s, y))y^{n-1}dyds\\
J_3=\int _{t-1}^t (t-s)^{-n}\int _0^{t-s}F(f(s, y))y^{n-1}dyds.
\end{align*}
with,
\begin{align*}
|J_1|&\le C|||f_0|||_\theta \int _0^{t-1} (t-s)^{-n}\int _0^1y^{n-1-\theta}dyds\\
&=C(n, \theta)\int _0^{t-1}(t-s)^{-n}ds\le C|||f_0|||_\theta,\\
|J_2|&\le C|||f_0|||_\theta \int _0^{t-1} (t-s)^{-n}\int _1^{t-s}y^{n-1}dyds\\
&=C|||f_0|||_\theta \int _0^{t-1} \left(1-(t-s)^{-n} \right)ds=\le C|||f_0|||_\theta\, t
\end{align*}
and
\begin{align*}
|J_3|&\le C|||f_0|||_\theta \int _{t-1}^t (t-s)^{-n}\int _0^{t-s}y^{n-1-\theta}dyds\\
&=C|||f_0|||_\theta \int _{t-1}^t (t-s)^{-\theta}ds\le C|||f_0|||_\theta.
\end{align*}
It follows, for $t>1$
\begin{align*}
\left|\int _0^t (t-s)^{-n}\int _0^{t-s}F(f(s, y))y^{n-1}dyds\right|\le C|||f_0|||_\theta\, t,
\end{align*}
and,
\begin{align*}
|b(t)-\ell (f_0; t)|\le C|||f_0|||_\theta \left(t^{-\theta}+t \right),\,\,\forall t>0.
\end{align*}
The same arguments show that  $|\ell (f_0; t)|\le C|||f_0|||_\theta (1+t^{-\theta})$ for all $t>0$ and (\ref{S6LimzeroE25}) follows.
\end{proof}

\section{The functions $u(t)$ and  $p_c(t)$.}
\label{Sup_c}
\setcounter{equation}{0}
\setcounter{theo}{0}
We now return to the notation of the time variable as in sub Section \ref{MR}.  Then, given the function $f(\tau , x)$ obtained in Theorem \ref{S6WeakSolL2},   $t=t(\tau )$ and $p_c(t)$ must be determined in order to define
\begin{equation*}
u(t, x)=f(\tau , x),\,\,\forall t>0, \,\forall x>0.
\end{equation*}
The functions $t, \tau $ and $p_c(t)$ are related by the change of time variable (\ref{S2NewTime}), i.e. 
\begin{align}
\label{S7newtime}
\tau =\int _0^{t}p_c(s)ds.
\end{align}
\begin{prop}
\label{S5Em}
For all $\tau >0$, 
\begin{align}
\int _0^\tau \int_0^\infty \left|\mathcal L(f(\sigma  ))(x)\right| n_0(x^2)(1+n_0(x^2))x^4dxd\sigma  <\infty, \label{S5Em1}
\end{align}
and, if 
\begin{align}
&m(\tau )=\int_0^\infty \mathcal L(f(\tau ))(x) n_0(x^2)(1+n_0(x^2))x^4dx,\,\,\forall \tau >0 \label{S5Em2}\\
&|m(\tau )|<\infty,\,\,\forall \tau >0. \label{S5Em3}
\end{align}
\end{prop}
\begin{proof}
The proposition is a direct and straightforward  consequence of the integrability property (\ref{S4E140-4X}) of $\mathcal L(f)$.
\end{proof}
Let us denote,
\begin{align}
\mathcal M(r )=\int _0^r m(\rho )d\rho,\,\,\forall r>0. \label{S5Em5}\\
q_c(\tau )=q_c(0)e^{-\mathcal M(\tau )},\,\,\forall \tau >0. \label{S5Em6}
\end{align}
\begin{prop}
\label{S5Emn}
For all $t>0$ there exists a unique $\tau >0$ such that
\begin{equation}
t=\int _0^\tau \frac {d\sigma } {q_c(\sigma )},\,\,\forall \tau >0.\label{S5CPt1Inv}
\end{equation}
\end{prop}
\begin{proof}
By Proposition \ref{S5Em}, $|\mathcal M(\tau )|<\infty$ for all $\tau >0$ and then $q_c(\tau )\in (0, \infty)$ for all $\tau >0$ and the integral in the right hand side of (\ref{S5CPt1Inv}) is well defined and convergent. Since $q_c(t)>0$ this integral is a monotone increasing function of $\tau $. It only remains to check that its range is $[0, \infty)$.

By Corollary \ref{SeEC2f3}, for $\varepsilon >0$ and $\tau_\varepsilon  $ large enough,
$$
\int _0^\infty f(\tau , x)n_0(1+n_0)x^4dx>C_*\int _0^\infty n_0(1+n_0)x^4dx-\varepsilon,\,\,\forall \tau \ge \tau _\varepsilon .
$$
Since, on the other hand,
\begin{align}
 \mathcal M(\tau )&=\int _0 ^\tau \int _0^\infty \int _0^\infty W(x, y)(f(\sigma  , y)-f(\sigma  , x)) y^4x^2dydxd\sigma \nonumber \\
&=\int _0^\tau \frac {d} {d\sigma  }\int _0^\infty n_0(1+n_0) f(\sigma  , x)x^4dxd\sigma \nonumber \\
&=\int _0^\infty n_0(1+n_0)f(\tau, x)x^4dx - \int _0^\infty n_0(1+n_0)f_0(x)x^4dx. \label{S5EmnEMM}
\end{align}
It follows, for $\tau >\tau _\varepsilon $,
\begin{align*}
\mathcal M(\tau )\ge C_*\int _0^\infty n_0(1+n_0)x^4dx-\varepsilon - \int _0^\infty n_0(1+n_0)f_0(x)x^4dx.
\end{align*}
Therefore, the function $e^{\mathcal M(\sigma )}$ is not integrable at infinity and
$$
\lim _{ \tau \to \infty }\int _0^\tau \frac {d\sigma } {q_c(\sigma )}=\infty
$$
and, for all  $t>0$, there exists a unique $\tau>0$ satisfying (\ref{S5CPt1Inv}).
\end{proof}

\begin{proof}
[\upshape\bfseries{Proof of Theorem  \ref{MainThm}}]  
For all $t>0$, let $\tau >0$ be given by Proposition \ref{S5Emn} and define, 
\begin{align}
&u(t, x)=f(\tau , x),\,\,\,\forall x>0,\label{S5Euu}\\
&p_c(t)=q_c(\tau ).\label{S5Eppcc}
\end{align}
where $f$ is obtained by Theorem \ref{S6WeakSolL2} with initial data $f_0=u_0$. From the definition of $p_c$,
\begin{equation*}
\frac {d t} {d \tau }=\frac {1} {q_c(\tau )}\Longrightarrow \frac {d \tau } {d t}=q_c(\tau )=p_c(t)
\end{equation*}
and  (\ref{S7newtime}) is satisfied. On the other hand,
\begin{align}
\frac {dp_c(t)} {dt}=\frac {dq_c(\tau )} {d\tau }\frac {d\tau } {dt}
=-q_c(\tau )m(\tau )\frac {d\tau } {dt}=-p_c(t)m(\tau )\frac {d\tau } {dt}. \label{S5Em8}
\end{align}
By (\ref{S5Em2}) and  (\ref{S5Euu}),
\begin{align*}
m(\tau )&=\int_0^\infty \mathcal L(f(\tau  )) n_0(x^2)(1+n_0(x^2))x^4dx\\
&=\int_0^\infty \mathcal L(u(t)) n_0(x^2)(1+n_0(x^2))x^4dx=\mu (t).
\end{align*}
It follows from (\ref{S5Em6})
\begin{align*}
\frac {dp_c(t)} {dt}&=-p_c(t)m(\tau )\frac {d\tau } {dt}=-p_c(t)\frac {d\mu (t)} {dt}\\
&= -p_c(t)\frac {d} {dt} \int_0^\infty \mathcal L(u(s))n_0(x^2)(1+n_0(x^2))x^4dxds
\end{align*}
and then, $p_c\in C^1(0, \infty)$ and satisfies (\ref{S1ERRR10xp}). On the other hand, by  (\ref{S7newtime}) and Theorem  \ref{S6WeakSolL2B}
\begin{align*}
&\frac {\partial u} {\partial t}(t, x)=\frac {d\tau } {dt}\frac {\partial f} {\partial \tau }(\tau , x)=-p_c(t)\frac {\partial f} {\partial \tau }(\tau , x).
\end{align*}
Then, Theorem \ref{MainThm} follows from Theorem \ref{S6WeakSolL2} and  Theorem  \ref{S6WeakSolL2B}, where  the functions $H$ in (\ref{S4E500}) and $\lambda  _{ \beta , \delta  }$ in (\ref{S1EPLXl}) are given by,
\begin{align}
&H(t, x)=\tilde H(\tau , x), \,\,\,a. e. \,\,(0, \infty)\times (0, \infty) \label{S5EH}\\
&\lambda  _{ \beta , \delta }(t, x)=\tilde \lambda  _{ \beta , \delta }(\tau , x) \,\,\,a. e. \,\,(0, \infty)\times (0, \infty), \label{S5Ell}
\end{align}
with $\tilde H$ given in (\ref{S3EHH1}), and  $\tilde \lambda  _{ \beta , \delta  }$ defined in (\ref{S3fx1}).
\end{proof}
\begin{proof}
[\upshape\bfseries{Proof of Corollary   \ref{SeEC}, Corollary \ref{SeEC2} and Corollary \ref{SeEpc1}}]  
It follows from (\ref{S4E125u}), and Corollary \ref{SeECf},
\begin{align*}
\frac {dE(t)} {dt }&=p_c(t) \int _0^\infty n_0(1+n_0)x^4\mathcal L(u(t ))(x)dx\\
&=p_c(t)  \int_0^\infty n_0(x)(1+n_0(x)\mathcal L f(\tau , x)x^4dx \\
&=p_c(t)\frac {d} {d\tau }\int _0^\infty n_0(1+n_0)x^4f(\tau , x)x^2dx=0
\end{align*}
and this yields property  (\ref{SeEC1}) in Corollary \ref{SeEC}. Property (\ref{SeEC2}) of Corollary \ref{SeEC} follows from (\ref{S3E23590}) and (\ref{S1ERRR10xp}), since  by (\ref{S4E140}), integration  of (\ref{S3E23590})  on $(0, \infty)$ yields,
\begin{align*}
\frac {d} {dt}\int _0^\infty n_0(1+n_0)u(t, x)x^4dx=-\frac {dp_c(t)} {dt}.
\end{align*}
From Proposition \ref{SeEC2f2} property  (\ref{SeEC2E1}) in Corollary \ref{SeEC2} follows, and  property  (\ref{SeEC2E2}) is deduced from Corollary \ref{SeEC2f3}.

Since $f\in C([0, \infty); L^1(0, \infty))$, by the identity (\ref{S5EmnEMM}) and (\ref{S5Em6}), $q_c\in C([0, \infty))$. It follows that $p_c\in C([0, \infty))$ and it is bounded on any compact subset of $[0, \infty)$. Passing to the limit in (\ref{S5EmnEMM}) as $t\to \infty$ and using Corollary \ref{SeEC2f3} property (\ref{SeEpc11}) is obtained. Then, it also follows that $p_c$ is bounded on $(0, \infty)$.
\end{proof}

\begin{prop}
\label{S6ELimzerou}
Suppose that  $u_0$ and  $u$ are as in Theorem  \ref{MainThm}. Then for all $t>0$, and $\delta >0$ as small as desired,
\begin{align}
\label{S6ELimzerou1}
u(t, x)&=a(t )+\left(||f_0|| _{ L^\infty (1, \infty) }+||f_0||_1\right)\mathcal O\left(\tau ^\delta x^{1-\delta } \right)+\nonumber\\
&+\sup _{ 0<x<1 }x^\theta |f_0(x)|\mathcal O\left( \tau ^{\delta -\theta} x^{1-\delta } \right)
,\,\,x\to 0,\\
\hbox{where,}\hskip 0.3cm a(t)&=b\left(\int _0^{t}p_c(s)ds\right)\label{S6ELimzerou2}
\end{align}
satisfies, for some constant $C>0$,
\begin{equation}
|a(t)|\le C|||u_0|||_\theta \left( \left(\int _0^{t}p_c(s)ds \right)^{-\theta}+\int _0^{t}p_c(s)ds\right),\,\,\forall t>0.\label{S6ELimzerou3}
\end{equation}

\end{prop}
\begin{proof}
By construction, $u(t, x)=f(\tau , x)$ were $\tau $ is given in terms of $t $ by (\ref{S7newtime}) and therefore,  (\ref{S6ELimzerou1}),(\ref{S6ELimzerou2}) follow from (\ref{S6LimzeroE0}) and  (\ref{S6ELimzerou3}) follows from (\ref{S6LimzeroE25}).
\end{proof}

\section{Appendix}
\setcounter{equation}{0}
\setcounter{theo}{0}

\subsection{Some further properties of $S(t)$.}

We prove in this Appendix two   properties of the solution $S(t)f_0$ of  (\ref{S3E1.L1}) with initial data $f_0$, that are not given in \cite{m}. The first is just an elementary continuity result. We seek for the continuity of $S(t)f_0$ with respect to $x$ in order to deduce the same property for the solution $u$ of (\ref{S3E23590}) and to be able later  to speak of $\ell (u(t))$.

We first briefly recall ...
\begin{equation}
W(s)=-2\gamma_e -2\psi \left(\frac {s} {2} \right)-\pi \cot \left(\frac {\pi s} {4} \right),\,\,\,s\in \CC,\,\,\mathscr Re(s)\in (-2, 4) \label{S4E7}
\end{equation}
where $\gamma _e$ is the Euler constant and $\psi (z)=\Gamma '(z)/\Gamma (z)$ is the Digamma function.

For any $\beta \in (0, 2)$ fixed, the function
\begin{align}
\label{S3PBE0}
B(s)=\exp\left(\int  _{ {\mathscr Re} (\rho) =\beta  }  \log (-W(\rho ))\left( \frac {1} {1-e^{2i\pi (s-\rho) } }-\frac {1} {1+e^{-2i\pi \rho }}\right)d\rho\right)
\end{align}
is analytic in the domain $\{s\in \CC;\,\mathscr Re(s)\in (\beta , \beta +1)\}$ and satisfies
\begin{align}
\label{S3PBE1}
B(s)=-W(s-1)B(s-1),\,\,\,\forall s\in \CC; {\mathscr Re}(s)\in (\beta , \beta +1).
\end{align}

\begin{prop}
\label{S3P7}
If $f_0\in L^1(0, \infty)\cap L^\infty _{ loc }(0, \infty)$, then $S(t)f_0\in L^\infty _{ loc  }\cap  C((0, \infty))$ for every $t>0$ and $S(\cdot)f_0(x)\in C(0, \infty)$ for all $x>0$.
\end{prop}
\begin{proof}
By definition,
\begin{align*}
S(t)f_0(x)-S(t)f_0(x')=\int _0^\infty \left(
\Lambda\left(\frac {t} {y}, \frac {x} {y} \right)-\Lambda\left(\frac {t} {y}, \frac {x'} {y} \right)
\right)f_0\left( y\right)\frac {dy} {y}
\end{align*}

\begin{align*}
|S(t)f_0(x)-S(t)f_0(x')|&\le \int  _{ y<\delta _1 }  \left(
\Lambda\left(\frac {t} {y}, \frac {x} {y} \right)-\Lambda\left(\frac {t} {y}, \frac {x'} {y} \right)
\right)f_0\left( y\right)\frac {dy} {y}+\\
&+\int  _{ y>R }  \left(
\Lambda\left(\frac {t} {y}, \frac {x} {y} \right)-\Lambda\left(\frac {t} {y}, \frac {x'} {y} \right)
\right)f_0\left( y\right)\frac {dy} {y}+\\
&+ \int  _{ \left| 1-\frac {x} {y}\right| >\rho,\, y \in (\delta _1, R) } \left(
\Lambda\left(\frac {t} {y}, \frac {x} {y} \right)-\Lambda\left(\frac {t} {y}, \frac {x'} {y} \right)
\right)f_0\left( y\right)\frac {dy} {y}+\\
&+\int  _{ \left| 1-\frac {x} {y}\right| <\rho,\, y\in (\delta_1, R) } \left(
\Lambda\left(\frac {t} {y}, \frac {x} {y} \right)-\Lambda\left(\frac {t} {y}, \frac {x'} {y} \right)
\right)f_0\left( y\right)\frac {dy} {y}\\
&=\int  _{ y<\delta _1 }|..| dy+\int  _{ y>R } |...| dy+J_1+J_2
\end{align*}
We choose $\delta _1$ small enough and $R$ large enough, both depending on $t$, to have, :
\begin{align*}
\int  _{ y<\delta _1 }|..| dy+\int  _{ y>R } |...| dy\le \varepsilon 
\end{align*}
using that $f_0\in L^1$ and the asymptotics of $\Lambda$ for large and small arguments. Consider for example, the integral for $y<\delta $, with $\delta < t$. Then $t/y>1$ and 
\begin{equation*}
\left|\Lambda\left(\frac {t} {y}, \frac {x} {y} \right)\right|+\left|\Lambda\left(\frac {t} {y}, \frac {x'} {y} \right)\right|\le Cy^3 (\max (x, t)^{-3}+\max (x', t)^{-3})
\end{equation*}
and, if we use that $x'>x/2$, 
\begin{align*}
&\left|\Lambda\left(\frac {t} {y}, \frac {x} {y} \right)-\Lambda\left(\frac {t} {y}, \frac {x'} {y} \right)
\right| \frac {|f_0\left( y\right)|} {y}\le Cy^2 (\max (x, t)^{-3}+\max (x/2, t)^{-3})\\
&\lim _{ \delta \to 0 } \int  _{ y<\delta _1 }  \left(
\Lambda\left(\frac {t} {y}, \frac {x} {y} \right)-\Lambda\left(\frac {t} {y}, \frac {x'} {y} \right)
\right)f_0\left( y\right)\frac {dy} {y}=0
\end{align*}
A similar argument yields the limit of the integral for $y>R$ as $R\to \infty$.

On the other hand, if we suppose $y\in (\delta _1, R),\,\,\,|x-x'|<\delta $
then,
\begin{align}
&\left| 1-\frac {x} {y}\right| <\rho\Longrightarrow \left| 1-\frac {x'} {y}\right|\le \left| 1-\frac {x} {y}\right|+
\left| \frac {x} {y}-\frac {x'} {y}\right|\le \rho +\frac {\delta } {\delta _1}\label{S3xr1}\\
&\left| 1-\frac {x} {y}\right| >\rho\Longrightarrow \left| 1-\frac {x'} {y}\right| \ge \left| 1-\frac {x} {y}\right|-
\left| \frac {x} {y}-\frac {x'} {y}\right|\ge \rho -\frac {\delta } {\delta _1}\label{S3xr2}
\end{align}
Then,
\begin{align*}
J _{ 2, 1 }=&\int  _{ \left| 1-\frac {x} {y}\right| <\rho,\, y\in (\delta_1, R) }
\Lambda\left(\frac {t} {y}, \frac {x} {y} \right)|f_0\left( y\right)|\frac {dy} {y}\\
&\le C(f_0)
\int  _{ \left| 1-\frac {x} {y}\right| <\rho,\, y\in (\delta_1, R) }
\Lambda\left(\frac {t} {y}, \frac {x} {y} \right)\frac {dy} {y}
\end{align*} 
Since  $t/y\ge t/R$ for all $y\in (\delta_1 , R)$. If $t/R>1/2$, then  the function,
\begin{equation*}
y\mapsto \Lambda\left(\frac {t} {y}, \frac {x} {y} \right)
\end{equation*}
is continuous on 
\begin{equation*}
D=\left\{y\in (\delta _1, R) \right\}
\end{equation*}
and by the Lebesgue's convergence Theorem,
\begin{equation*}
\lim _{ \rho \to 0 }J _{ 2, 1 }=0.
\end{equation*}
On the other hand, for $|x-x'|<\delta $, by (\ref{S3xr1}):
\begin{align*}
J _{ 2, 2 }=&\int  _{ \left| 1-\frac {x} {y}\right| <\rho,\, y\in (\delta_1, R) }
\Lambda\left(\frac {t} {y}, \frac {x'} {y} \right)|f_0\left( y\right)|\frac {dy} {y}\\
&\le C(f_0)
\int  _{ \left| 1-\frac {x'} {y}\right| <\rho+\frac {\delta } {\delta _1},\, y\in (\delta_1, R) }
\Lambda\left(\frac {t} {y}, \frac {x'} {y} \right)\frac {dy} {y}
\end{align*} 
and 
\begin{equation*}
\lim _{ \rho+\frac {\delta } {\delta _1}  \to 0 }\int  _{ \left| 1-\frac {x} {y}\right| <\rho,\, y\in (\delta_1, R) }
\Lambda\left(\frac {t} {y}, \frac {x'} {y} \right)|f_0\left( y\right)|\frac {dy} {y}=0.
\end{equation*}

If $t/R<1/2$ we must divide the domain $D$ in two sub domains 
\begin{align*}
D_+=\left\{y\in (\delta _1, R); \frac {t} {y}>\frac {1} {2}\right\},\,\,\,\,D_-=\left\{y\in (\delta _1, R); \frac {t} {R}<\frac {t} {y}<\frac {1} {2}\right\}.
\end{align*}
With the previous argument,
\begin{align*}
&\lim _{ \rho  \to 0 } \int  _{ \left| 1-\frac {x} {y}\right| <\rho,\, y\in D^+ }
\Lambda\left(\frac {t} {y}, \frac {x} {y} \right)|f_0\left( y\right)|\frac {dy} {y}=0\\
&\lim _{ \rho +\frac {\delta } {\delta _1} \to 0 } \int  _{ \left| 1-\frac {x} {y}\right| <\rho,\, y\in D^+ }
\Lambda\left(\frac {t} {y}, \frac {x'} {y} \right)|f_0\left( y\right)|\frac {dy} {y}=0.
\end{align*}
In the domain $D_-$, with $r=t/R>0$ and $\alpha=r/2$
\begin{align*}
&\left|\frac {\log\left(\frac {x} {y} \right)} {|1-x/y|^\alpha }\Lambda\left(\frac {t} {y}, \frac {x} {y} \right)\right|\le C\\
&\left| \Lambda\left(\frac {t} {y}, \frac {x} {y} \right)\right|\le C\frac {|1-x/y|^\alpha } {\log\left(\frac {x} {y} \right)} \le C'|1-x/y|^{\alpha-1}
\end{align*}
and then,
\begin{align*}
J _{ 2, 1 }^-\le C(f_0)\int  _{ \left| 1-\frac {x} {y}\right| <\rho,\, y\in D^- } |1-x/y|^{\alpha-1} dy\to 0,\,\,\hbox{as}\,\,\rho \to 0
\end{align*}
and arguing as before, 
\begin{align*}
&\lim _{ \rho +\frac {\delta } {\delta _1} \to 0 } \int  _{ \left| 1-\frac {x} {y}\right| <\rho,\, y\in D^- }
\Lambda\left(\frac {t} {y}, \frac {x'} {y} \right)|f_0\left( y\right)|\frac {dy} {y}=0.
\end{align*}
By (\ref{S3xr2}), if $|x-x'|<\delta $ and $|1-x/y|>\rho $ then $|1-x'/y|>\rho-\frac {\delta } {\delta_1 } $ and,  by...
\begin{align*}
&\Lambda \in C((0, \infty)\times \tilde D)\\
&\tilde D=\left\{z>0; |z-1|>\rho -\frac {\delta } {\delta _1}\right\}
\end{align*}
It follows that, for all $y\in (\delta _1, R)$, such that $|1-x/y|>\rho $ the function
\begin{align*}
x'\mapsto \Lambda\left(\frac {t} {y}, \frac {x'} {y} \right)
\end{align*}
is continuous at $x$. Then, for all $t>0$ fixed and $y\in (\delta _1, R)$ such that  $|1-x/y|>\rho $:
\begin{align*}
&\lim _{ x'\to x }\Lambda\left(\frac {t} {y}, \frac {x'} {y} \right)=\Lambda\left(\frac {t} {y}, \frac {x} {y} \right)\\
&\left|\Lambda\left(\frac {t} {y}, \frac {x'} {y} \right)-\Lambda\left(\frac {t} {y}, \frac {x} {y} \right) \right|\le 2
\sup\left\{\Lambda(\tau , z);  \tau \in \left(\frac {t} {R}, \frac {t} {\delta _1} \right), \,z\in \tilde D \right\}.
\end{align*}
We deduce from the Lebesgue's convergence Theorem,
$$
\lim _{ x\to x' }J_1=0
$$
and this gives the continuity of with respect to $x$. On the other hand, fix $x>0$ and $t>0$ and suppose that $t_n\to t$.
\begin{align*}
\int _0^\infty \left| \Lambda \left(\frac {t} {y}, \frac {x} {y} \right)- \Lambda \left(\frac {t_n} {y}, \frac {x} {y} \right)\right|f_0(y)\frac {dy} {y}\le
\int  _{ \left| 1-\frac {x} {y}\right| <\rho }\left| \Lambda \left(\frac {t} {y}, \frac {x} {y} \right)- \Lambda \left(\frac {t_n} {y}, \frac {x} {y} \right)\right|f_0(y)\frac {dy} {y}+\\
+\int _{ \left| 1-\frac {x} {y}\right| >\rho } \left| \Lambda \left(\frac {t} {y}, \frac {x} {y} \right)- \Lambda \left(\frac {t_n} {y}, \frac {x} {y} \right)\right|f_0(y)\frac {dy} {y}=I_1+I_2.
\end{align*}
In the first integral, for $\rho \in (0, 1)$, we have 
$$
\frac {x} {1+\rho }<y<\frac {x} {1-\rho }.
$$
In that range of values of $y$, we have the estimates,
\begin{align*}
\left| \Lambda \left(\frac {t} {y}, \frac {x} {y} \right)\right|\le  C\left| 1-\frac {x} {y}\right| ^{-1+2\frac {t} {y}}\le C\left| 1-\frac {x} {y}\right| ^{-1+2\frac {t(1-\rho )} {x}}
\le C\left| 1-\frac {x} {y}\right| ^{-1+\frac {t} {x}}\\
\left| \Lambda \left(\frac {t_n} {y}, \frac {x} {y} \right)\right|\le  C\left| 1-\frac {x} {y}\right| ^{-1+2\frac {t_n} {y}}\le C\left| 1-\frac {x} {y}\right| ^{-1+\frac {t(1-\rho )} {x}}
\le C\left| 1-\frac {x} {y}\right| ^{-1+\frac {t} {2x}}
\end{align*}
since we may assume that $t_n\ge t$ and $1-\rho >1/2$. We deduce,
\begin{align*}
I_1&\le C\int  _{ \left| 1-\frac {x} {y}\right| <\rho }\left| 1-\frac {x} {y}\right| ^{-1+\frac {t} {2x}}f_0(y)\frac {dy} {y}\\
&\le C||g|| _{ L^\infty\left(x/2, 2x\right) }\int  _{ |1-z|<\rho  }(1-z)^{-1+\frac {t} {2x}}\frac {dz} {z}\to 0,\,\,\hbox{as}\,\,\rho \to 0.
\end{align*}
This fixes $\rho $. We then have,
\begin{align*}
&\lim _{ n\to \infty } \Lambda \left(\frac {t_n} {y}, \frac {x} {y} \right)=\Lambda \left(\frac {t} {y}, \frac {x} {y} \right)=0,\,\,\hbox{for}\,\,\left| 1-\frac {x} {y}\right| >\rho\\
&\Lambda \left(\frac {t_n} {y}, \frac {x} {y} \right)\frac {|f_0(y)|} {y}\le C(t, x)|f_0(y)|
\end{align*}
and by the Lebesque's convergence Theorem, $J_2\to 0$ as $n\to \infty$.
\end{proof}

The next result is  useful to consider  initial data $f_0$ that are unbounded near the origin.
\begin{prop}
\label{S5P1}
Suppose that $g\in L^1(0, \infty)$ is such that, for some $\theta>0$,
\begin{equation*}
|||g|||_\theta=\sup _{ 0<x<1 }x^\theta |g(x)|+\sup _{ x>1 }|g(x)|<\infty
\end{equation*}
Then, $S(t)g\in L^\infty(0, \infty)$ and more precisely,
\begin{align}
|S(t)(g)(x)|&\le  C_G||g|| _{ L^\infty(1, \infty) }+Ct^{-\theta}\sup _{ 0<y<1 }y^\theta |g(y)|\,\,\,\forall x\in (0, 2) \label{S5P1E1}\\
|S(t)(g)(x)|&\le  C_G||g|| _{ L^\infty(1, \infty) }+C||g||_1\,\,\,\forall x>2\label{S5P1E2}
\end{align}
\end{prop}
\begin{proof}
By hypothesis, for all $\delta >0$,  $g\in L^\infty _{ loc }(\delta , \infty)$  and for all $x>0$,
\begin{align*}
\left|\int _2^\infty \Lambda\left(\frac {t} {y}, \frac {x} {y} \right)g(y)\frac {dy} {y}\right|\le C_G||g|| _{ L^\infty(2, \infty) }.
\end{align*}
On the other hand, for $x>2$
\begin{align*}
\left|\int _0^1 \Lambda\left(\frac {t} {y}, \frac {x} {y} \right)g(y)\frac {dy} {y}\right|\le
\sup _{ y\in (0, 1) }\left(\Lambda\left(\frac {t} {y}, \frac {x} {y} \right) \frac {1} {y}\right)||g||_1
\end{align*}
where  $x/y>2$ for $y\in (0, 1)$.  
Several cases are now possible. If $t>x$, then $t>y$  and it follows from Proposition 3.1 and Proposition 3.2 of \cite{m},
\begin{equation*}
\Lambda\left(\frac {t} {y}, \frac {x} {y} \right)\frac {1} {y}\le Ct^{-3}y^2\le C.
\end{equation*}
The same argument yields, for $t\in (2, x)$, 
\begin{equation*}
\Lambda\left(\frac {t} {y}, \frac {x} {y} \right)\frac {1} {y}\le Cx^{-3}y^2\le C
\end{equation*}
and, for $0<y<t<2$,
\begin{equation*}
\Lambda\left(\frac {t} {y}, \frac {x} {y} \right)\frac {1} {y}\le \frac {C} {y} \left( \frac {t} {y}\right)^{-3}\left(\frac {x} {t} \right)^{-1}\le Ct^{-2}x^{-1}y^2\le C
\end{equation*}
By Proposition 3.5 of \cite{m},  when  $t<1$, and $y\in (t, 1)$,
\begin{align*}
\Lambda\left(\frac {t} {y}, \frac {x} {y} \right)\frac {1} {y} &\le C _{ \varepsilon  }\left( \frac {x} {y}\right)^{-3+\varepsilon }\left( \frac {t} {y}\right)^{9-\varepsilon }
+C \left( \frac {x} {y}\right)^{-5}\left( \frac {t} {y}\right)^{7}\\
&= C _{ \varepsilon  } t^{9-\varepsilon }x^{-3+\varepsilon }y^{-6}+Ct^7x^{-5}y^{-2}\le C _{ \varepsilon  }t^{3-\varepsilon }x^{-3+\varepsilon }+Ct^5x^{-5}\le C
\end{align*}
and this ends the proof of (\ref{S5P1E2}).\\
When  $x\in (0, 2)$ we first write,
\begin{align*}
\left|\int _0^1\Lambda\left(\frac {t} {y}, \frac {x} {y} \right)g(y)\frac {dy} {y}\right|\le 
\sup _{ 0<y<1 }y^\theta |g(y)|\int _0^1\Lambda\left(\frac {t} {y}, \frac {x} {y} \right)\frac {dy} {y^{1+\theta}}.
\end{align*}
When $t>y$ Proposition 3.1 and Proposition 3.2 may be applied. It follows that, for $x<t$,
\begin{align*}
\left|\Lambda\left(\frac {t} {y}, \frac {x} {y} \right)\frac {1} {y^{1+\theta}}\right|\le Ct^{-3}y^{2-\theta}\le Ct^{-1-\theta}\\
\end{align*}
and for $x>t$,
\begin{align*}
\left|\Lambda\left(\frac {t} {y}, \frac {x} {y} \right)\frac {1} {y^{1+\theta}}\right|\le Cx^{-3}y^{2-\theta}\le Ct^{-1-\theta}.
\end{align*}
in both cases,
\begin{align*}
\int _0^t\left|\Lambda\left(\frac {t} {y}, \frac {x} {y} \right)\frac {1} {y^{1+\theta}}\right|dy\le Ct^{-\theta}.
\end{align*}
Similar arguments show the same estimate for $y\in (t, 1)$. Therefore, 
$$
\int _0^1\Lambda\left(\frac {t} {y}, \frac {x} {y} \right)\frac {dy} {y^{1+\theta}}\le C t^{-\theta}\,\,\,\forall x\in (0, 2)
$$
and (\ref{S5P1E1}) follows.
\end{proof}

\begin{lem}
\label{S5LCX1}
There exists a constant $\sigma _0^*\in (-2, -1)$ such that, for   $\theta \in [0, 1)$, $\beta \in (0, 1)$

\noindent
(i)  there exists  $C>0$ such that for all $f_0\in L^1(0, \infty)$ satisfying (\ref{S4Etheta}) and for all $t\in (0, 1)$
\begin{align}
\left|\frac {\partial } {\partial x}S(t) f_0(x)\right|\le
C\left(t^{-\theta-1}+x^{-1-\sigma _0^*}t^{\sigma _0^*-\theta }\right)\sup _{ 0\le y\le t }(y^\theta |f_0(y)|)+\nonumber\\
+Cx^{-1-\sigma _0^*}||f_0|| _1+Ct^2||f_0||_1,\,\,\,\forall x\in (0, t) \label{S5LCX1E1}
\end{align}
and
\begin{align}
\left|\frac {\partial } {\partial x}S(t) f_0(x)\right|\le
Cx^{-1-\beta }t^{\beta-1} ||f_0||_1,\,\,\,\forall x >t.\label{S5LCX1E10}
\end{align}

\noindent
(ii) There also exists a constant $C>0$ such that for all $f_0\in L^1(0, \infty)$ satisfying (\ref{S4Etheta}) for all $t>1$,
\begin{align}
\label{S5LCX1E20}
\left|\frac {\partial } {\partial x}S(t) f_0(x)\right|\le Ct^{-4}\sup _{ 0\le y\le t }(y^\theta |f_0(y)|)+Ct^2||f_0||_1,\,\,\forall x\in (0, t),
\end{align}
and
\begin{align}
\label{S5LCX1E2}
\left|\frac {\partial } {\partial x}S(t) f_0(x)\right|\le 
Ct^{\beta -1}x^{-1-\beta }||f_0||_1,\,\,\forall x >t.
\end{align}
\end{lem}
\begin{proof}
Suppose first that $x\in (0, t)$ and write
\begin{align*}
&\frac {\partial f} {\partial x}(t, x)=I_1+I_2\\
&I_1=\int _0^t \frac {\partial \Lambda } {\partial z}  \left(\frac {t} {y}, \frac {x} {y} \right)  f_0(y) \frac {dy} {y^2},\,\,\,
I_2=\int _t^\infty  \frac {\partial \Lambda } {\partial z}   \left(\frac {t} {y}, \frac {x} {y} \right)  f_0(y) \frac {dy} {y^2}.
\end{align*}
In the term $I_1$, $y\in (0, t)$ and then,  Proposition 3.4 in \cite{m} may be applied to obtain,
\begin{align}
|I_1|&\le C\int _0^t  \left(\frac {t} {y} \right)^{-4} \left(A+\mathcal O\left(\left|\frac {x} {t}\right|^{\delta } \right)\right)f_0(y)\frac {dy} {y^2}\le Ct^{-4}\int _0^t f_0(y) y^2dy\nonumber\\
&\le 
\begin{cases}
\displaystyle{C\sup _{ 0\le y\le t }(y^\theta |f_0(y)|)t^{-4}\int _0^t y^{2-\theta}dy=Ct^{-1-\theta}\sup _{ 0\le y\le t }(y^\theta |f_0(y)|),\,\,t<1},\\
\displaystyle{Ct^{-4}\int _0^1 f_0(y) y^2dy+Ct^{-4}\int _1^t f_0(y) y^2dy\,\,\,t>1}
\end{cases}\nonumber\\
&\le
\begin{cases}
\displaystyle{C\sup _{ 0\le y\le t }(y^\theta |f_0(y)|)t^{-4}\int _0^t y^{2-\theta}dy=Ct^{-1-\theta}\sup _{ 0\le y\le t }(y^\theta |f_0(y)|),\,\,t<1},\\
\displaystyle{Ct^{-4}\sup _{ 0\le y\le t }(y^\theta |f_0(y)|)+Ct^{-2}\begin{pmatrix}||f_0||_1\\or\\ t||f_0|| _{ L^\infty(1, \infty) }    \end{pmatrix}\,\,\,t>1.} \label{S6EI11}
\end{cases}
\end{align}
In the term  $I_2$, $t/y<1$ and then,  by (3.32) in Proposition 3.5 of \cite{m}, there exists   constant $C$ such that,
\begin{align*}
\label{S4P9E1DXZ}
\left|\frac {\partial \Lambda } {\partial z}  \left(\frac {t} {y}, \frac {x} {y} \right)\right|
\le C\left( tx^{-1-\sigma _0^*}y^{\sigma _0^*}+ tx^{-1-\sigma _1^*}y^{\sigma _1^*}+t^{2} y^{-2}\right),\forall x\in (0, t).
\end{align*}
where $\sigma _j^*$ are given real numbers such that $\sigma _j^*\in (-2(2j+1), -2(2j+1)+1)$. Since $y>t>x$ in the integration's domain of $I_2$ and $\sigma _1^*<\sigma _0^*<0$ it follows that
\begin{align*}
 tx^{-1-\sigma _1^*}y^{\sigma _1^*}\le tx^{-1-\sigma _0^*}y^{\sigma _0^*}
 \end{align*}

Then,
\begin{equation*}
|I_2|\le C tx^{-1-\sigma _0^*}  \int _t^\infty y^{\sigma _0^*-2}|f_0(y)|dy+Ct^2\int _t^\infty |f_0(y)|dy
\end{equation*}
where the first term in the right hand side may estimated as follows. If $t\in (0, 1)$, 
\begin{align*}
&\int _t^\infty y^{\sigma _0^*-2}|f_0(y)|dy\le  \int _t^1y^{\sigma _0^*-2}|f_0(y)|dy+\int _1^\infty y^{\sigma _0^*-2}|f_0(y)|dy\\
&\le C\sup _{ 0\le y\le 1 }(y^\theta |f_0(y)|) t^{-1+\sigma _0^*-\theta}+
||f_0|| _{ 1 }
\end{align*}
and then
\begin{align*}
tx^{-1-\sigma _0^*}  \int _t^\infty y^{\sigma _0^*-2}|f_0(y)|dy\le  
C   t^{\sigma _0^*-\theta}x^{-1-\sigma _0^*}\sup _{ 0\le y\le 1 }(y^\theta |f_0(y)|)+ tx^{-1-\sigma _0^*}||f_0|| _1.
\end{align*}
Since $x\in (0, t)$ and $\sigma _0^*<0$, it follows that,
\begin{align*}
& t^{\sigma _0^*-\theta}x^{-1-\sigma _0^*}=t^{-\theta}x^{-1}\left( \frac {x} {t}\right)^{-\sigma _0^*}<t^{-\theta}x^{-1},\\
& tx^{-1-\sigma _0^*}=t^{-\theta}x^{-1}\left( x^{-\sigma _0^*}t^{1+\theta}\right)\le t^{-\theta}x^{-1} t^{-\sigma _0^*+1+\theta}\le t^{-\theta}x^{-1}
\end{align*}
and then,
\begin{align}
|I_2|
\le C   t^{-\theta}x^{-1}\left(\sup _{ 0\le y\le 1 }(y^\theta |f_0(y)|)+ ||f_0|| _1\right)+Ct^2||f_0||_1.
 \label{S6EI21}
\end{align}
If on the contrary $t>1$,
\begin{align}
&|I_2|\le C tx^{-1-\sigma _0^*} \int _t^\infty y^{\sigma _0^*-2}|f_0(y)|dy+Ct^2\int _t^\infty |f_0(y)|dy\nonumber\\
&\le \displaystyle{tx^{-1-\sigma _0^*}\begin{pmatrix} Ct^{\sigma _0^*-1}||f_0|| _{ L^\infty (1, \infty) } \\ or \\    t^{\sigma _0^*-2}||f_0|| _1 \end{pmatrix}}
+Ct^2||f_0||_1. \label{S6EI22}
\end{align}
This shows the result for all $t>0$ and $x<t$. If $x>t>y$, by  Proposition 3.4, for $\varepsilon >0$ as small as desired there exists $C_\varepsilon >0$ such that,
\begin{align*}
\left|\frac {\partial \Lambda } {\partial z}  \left(\frac {t} {y}, \frac {x} {y} \right)\right|
\le C_\varepsilon \left(x^{-4}y^4+y^4x^{-4-\varepsilon }t^\varepsilon \right)\le C_\varepsilon x^{-4}y^4,\,\,\,\forall x>t.
\end{align*}
Then,
\begin{align}
|I_1|&\le Cx^{-4}\int _0^t y^2 f_0(y)dy\nonumber\\
&\le 
\begin{cases}
\displaystyle{C\sup _{ 0\le y\le t }(y^\theta |f_0(y)|)x^{-4}\int _0^t y^{2-\theta}dy=Cx^{-4}t^{3-\theta}\sup _{ 0\le y\le t }(y^\theta |f_0(y)|),\,\,t<1},\nonumber \\
\displaystyle{Cx^{-4}\int _0^1 f_0(y) y^2dy+Cx^{-4}\int _1^t f_0(y) y^2dy\,\,\,t>1}
\end{cases} \nonumber\\
&\le
\begin{cases}
\displaystyle{Cx^{-4}t^{3-\theta}\sup _{ 0\le y\le t }(y^\theta |f_0(y)|),\,\,t<1},\\
\displaystyle{Cx^{-4}\sup _{ 0\le y\le t }(y^\theta |f_0(y)|)+Cx^{-4}t^{2}\begin{pmatrix}||f_0||_1\\or\\ t||f_0|| _{ L^\infty(1, \infty) }    \end{pmatrix}\,\,\,t>1.}
\end{cases} \label{S6EI12}
\end{align}
For $y>t$ and $x>t$, we must argue as in the previous cases, unfortunately there is an error in the estimate (3.32) of Proposition 3.5 in \cite{m}. The error is corrected in Lemma \ref{S6lemFF} below and gives
\begin{align*}
\left|\frac {\partial \Lambda } {\partial z}  \left(\frac {t} {y}, \frac {x} {y} \right)\right|
\le C\left( \left(\frac {x} {y} \right)^{-\beta -1}+  \left(\frac {x} {y} \right)^{-1-c'}\left(\frac {t} {y} \right)^{-\beta '+c'}\right)
\end{align*}
Therefore
\begin{align*}
|I_2|\le C_\beta x^{-\beta -1}\int _t^\infty y^{\beta-1}|f_0(y)|dy+Cx^{-1+c'}t^{-\beta '+c'}\int _t^\infty y^{-1+\beta '}|f_0(y)|dy
\end{align*}
and then, for $\beta \in (0, 1)$,
\begin{align}
|I_2|\le C\left(t^{\beta -1}x^{-1-\beta }+x^{-1-c'}t^{c'-1}\right)||f_0||_1,\,\,\forall x> t>0\label{S6EI23}
\end{align}
Since  $x^{-4}t^{3-\theta}<x^{-1-\sigma _0^*}t^{\sigma _0^*-\theta }$  when $x>t$, this ends the proof for all $t>0$ and $x>t$.
\end{proof}
The following straightforward Corollary simplifies somewhat Lemma's  \ref{S5LCX1} statement, 
at the price of a little loss in the estimates.
\begin{cor}
\label{S4C26m}
There exists a constant $C>0$ such that, for all $f_0\in L^1$, $y^\theta f_0\in L^\infty(0, 1)$,
\begin{align}
\label{S4C26mE1}
\left|\frac {\partial } {\partial x}S(t) f_0(x)\right|\le 
\begin{cases}
\displaystyle{C x^{-1+\delta }t^{-\theta-\delta }|||f_0||| _{\theta, 1 },\,\,\,0<x<t<1},\\
\displaystyle{Ct^2 |||f_0||| _{\theta, 1 },\,\,\,t>1,\,\,0<x<t},\\
\displaystyle{Ct^{\beta -1}x^{-1-\beta }|||f_0||| _{ \theta, 1 },\,\,\,x>t.}
\end{cases}
\end{align}
for all $\delta \in (0, 1)$,
where $|||f_0||| _{\theta, 1 }=\sup _{ 0\le y\le t }(y^\theta |f_0(y)|)+||f_0||_1$.
\end{cor}
\begin{rem}
\label{S4C26mR}
The right hand side of  (\ref{S4C26mE1}) may be denoted as $h(t)g(x)|||f_0||| _{ \theta, 1 }$ and  the expressions  of the functions $g$ and $h$ in the different domains of $t$ and $x$ are given  in (\ref{S4C26mE1}).
\end{rem}
\begin{lem}
\label{S6lemFF}
For all $\beta \in (0, 2), \beta '\in (-1, 0)$ and $c'$ such that $-1<\beta '<0<c'<1+\beta '<1$ there exists a  constant $C0$ such that 
\begin{align*}
\left|\frac {\partial \Lambda } {\partial x}  \left(t, x \right)\right|
\le C_\beta  \left(x^{-\beta -1}+x^{-1-c'}t^{-\beta '+c'}\right)\,\,\,\forall t\in (0, 1),\,\,x>t.
\end{align*}
\end{lem}
\begin{proof}
The proof is based on the following representation formula of the function $\partial _x\Lambda (t, x)$:

\begin{align*}
&\frac {\partial \Lambda } {\partial x}(t, x)=\left( x\frac {\partial } {\partial x}\right)^3(J(t, x))\\
&J(t, x)=-\frac {1} {4\pi^2} \int  _{ c-i\infty }^{c+i\infty} 
\int  _{ \mathscr Re(\sigma   )=\beta  }\frac {t^{-\sigma+s }B(s)\Gamma (\sigma  -s) } { B(\sigma   )}
s^{-2}x^{-s-1}d\sigma   ds.
\end{align*}
for all $\beta \in (0, 2)$ and $c\in (0, \beta )$. As indicated in \cite{m}, when $x/t>x>1$  it is suitable to deform first the  $s$ contour integrals  in $J$ towards larger values of $\mathscr Re (s)$. Since by construction $c<\beta $ this  first requires to cross  the pole of  $\Gamma (\sigma  -s)$ at $s =\sigma$, from where, for $c'\in (\beta , 2)$,
\begin{align*}
J(t, x)&=\frac {1} {2 i\pi }
\int  _{ \mathscr Re(\sigma   )=\beta  }\sigma  ^{-2}x^{-\sigma-1 }d\sigma +J_1(t, x)\\
J_1(t, x)&=-\frac {x^{-1}} {4\pi^2}
\int  _{ \mathscr Re(\sigma   )=\beta  }\int  _{ c'-i\infty }^{c'+i\infty} \frac {t^{-\sigma}B(s)\Gamma (\sigma  -s) } { B(\sigma   )s^2}
\left(\frac {x} {t} \right)^{-s}ds   d\sigma
\end{align*}
Since  $B(s)$ is analytic on $\mathscr Re(s)\in (0, 1)$ and $(B(\sigma ))^{-1}$ is analytic  on $\mathscr Re(\sigma )\in (-1, 2)$ it is possible to move the $\sigma $-integration contour towards lower values of $\mathscr Re(\sigma   )$ and 
\begin{align}
J_1(t, x)&=-\frac {x^{-1}} {4\pi^2}
\int  _{ \mathscr Re(\sigma   )=\beta'  }\int  _{ c'-i\infty }^{c'+i\infty} \frac {t^{-\sigma}B(s)\Gamma (\sigma  -s) } { B(\sigma   )s^2}
\left(\frac {x} {t} \right)^{-s}ds   d\sigma \label{S6lemFFE1}
\end{align}
where $\beta '$ and $c'$ satisfy $-1<\beta '<0<c'<1+\beta '<1$.
For $\beta \in (0, 2)$
\begin{align*}
\left|\int  _{ \mathscr Re(\sigma   )=\beta  }\sigma  ^{-2}x^{-\sigma-1 }d\sigma \right|\le \sqrt 2
x^{-1-\beta }\int  _{ \RR }\frac {dv} {(\beta^2+v^2)}.
\end{align*}
Since $x>t$, $c'>0$ and  $0<t<1$, $\beta '<0$,  $J_1(t, x)=\mathcal O\left( x^{-1-c'}t^{-\beta '+c'}\right)$ and the result~follows.
\end{proof}
We  close this Appendix with the following Remark and  some simple estimates.
\begin{rem}
\label{S5R1}
If $f_0\in L^\infty\cap L^1(0, \infty)$ it follows from the estimate  (1.37) in \cite{m}
\begin{align}
&\frac {\partial } {\partial t}\int _0^\infty f_0(z) \Lambda \left(\frac {t} {z}, \frac {x} {z} \right)\frac {dz} {z}\le
\frac {C t^{3-\theta}||x^\theta f_0|| _{ L^\infty(0, t) }} {\max (t^4, x^4)}+C||f_0|| _{ \infty }\1 _{ 2x/3<t<2x }+\nonumber\\
&\hskip 9cm +C\zeta _{ \theta  } (t,x) ||f_0|| _{ 1, \theta }
\label{S5R1}
\end{align}
where the function $\zeta  _{ \theta }(t, x)$ is defined in (\ref{S1Zetai}). 
\end{rem}
\begin{prop} 
\label{SAP5}
For all $t>0$ fixed,
\begin{align}
&\int _0^t \zeta _{ \theta } (s, x)ds\le \frac {C\min \left(t^2, t^{3-\theta}\right)} {x^3} \1 _{ t<2x/3 }+C\min (x^{-1}, x^{-\theta})\1 _{2x/3< t}=:\omega  _{ 1,\theta }(t, x)\label{SAEw1}\\
&\int _0^t \zeta _{0 } (s, x)(t-s)^{-\theta}ds\le \frac {C \min (t^{2-\theta}, t^{3-\theta})} {x^3} \1 _{ t<2x/3 }+\nonumber\\
&\hskip 6.5cm +C\min (x^{-1-\theta}, x^{-\theta})\1 _{2x/3< t}=:\omega  _{ 2 } (t, x) \label{SAEw2} \\
&\int _0^t\left(1+\xi (t-s, x)\right)ds=\left(t+\frac {\min(t^4, x^4)} {4x^4} \right)+\log\left( \frac {t} {x}\right)_+=:\Xi_1(t, x) \label{SAEXi1}\\
&\int _0^t s^{-\theta} \xi (t-s, x)ds\le C x^{-4}t^{4-\theta}\1 _{ x\ge t }+Ct^{-\theta}\left(1+\log\left( \frac {t} {x}\right)\right)\1 _{ x\le t }=:\Xi_2(t, x) \label{SAEXi2}
\end{align}
\end{prop}

\begin{proof}
\begin{align*}
\int _0^t \zeta _{ \theta } (s, x)ds&=\frac {1} {x^3}\int _0^t \min (s, s^{2-\theta})ds\le \frac {C} {x^3}\min \left(t^2, t^{3-\theta}\right),\,\,\forall t<2x/3\\
\int _0^t \zeta _{ \theta } (s, x)ds&=\frac {1} {x^3}\int _0^{2x/3}\min (s, s^{2-\theta})ds+\frac {1} {\max(x^2, x^{1+\theta}))} \int  _{ 2x/3 }^{t}ds\\
&\le \frac {C} {x^3}\min \left(x^2, x^{3-\theta}\right)+\frac {C(t-2x/3)} {\max(x^2, x^{1+\theta}))},\,\,\, \hbox{if}\,\,2x/3<t<2x.\\
\int _0^t \zeta _{ \theta } (s, x)ds&=\frac {1} {x^3}\int _0^{2x/3}\min (s, s^{2-\theta})ds+\frac {1} {\max(x^2, x^{1+\theta})} \int  _{ 2x/3 }^{2x}ds+\\
&+x \int _{ 2x }^t \min (s^{-2-\theta}, s^{-3})ds\le \frac {C} {x^3}\min \left(x^2, x^{3-\theta}\right)+\frac {C} {\max(x, x^{\theta}))}+
\\&+Cx\min \left((2x)^{-1-\theta}-t^{-1-\theta}, (2x)^{-2}-t^{-2} \right)\,\,\,\hbox{if}\,\,\,2x<t.
\end{align*}
On the other hand, we easily obtain if $t<2x/3$,
\begin{align*}
\int _0^t \zeta _{0 } (s, x)&(t-s)^{-\theta}ds=\frac {1} {x^3}\int _0^t \min (s, s^{2})(t-s)^{-\theta}ds\le \frac {C} {x^3} \min (t^{2-\theta}, t^{3-\theta})\\
\hskip -2.5cm  \hbox{since}:\hskip 2.2cm &\int _0^t s(t-s)^{-\theta}ds=t^{2-\theta}\int _0^1 r(1-r)^{-\theta}dr\\
&\int _0^t s^{2}(t-s)^{-\theta}ds=t^{3-\theta}\int _0^1 r^{2-\theta}(1-r)^{-\theta}dr,
\end{align*}
and, if $2x/3<t<2x$,
\begin{align*}
\int _0^t \zeta _{ 0 } (s, x)(t-s)^{-\theta}ds & \le \frac {C} {x^3} \min (x^{2-\theta}, x^{3-\theta})
+\frac {1} {\max(x^2, x))} \int  _{ 2x/3 }^{t} (t-s)^{-\theta}ds\\
&\le  C \min (x^{-1-\theta}, x^{-\theta})+\frac {Ct^{1-\theta}} {\max(x^2, x)}.
\end{align*}
For $2x<t$,
\begin{align*}
&\int _0^t \zeta _{ 0 } (s, x)(t-s)^{-\theta}ds  \le   C \min (x^{-1-\theta}, x^{-\theta})+\frac {Cx^{1-\theta}} {\max(x^2, x)}+\\
&+x \int _{ 2x }^t \min (s^{-2}, s^{-3})(t-s)^{-\theta}ds\le C \min (x^{-1-\theta}, x^{-\theta})+\\
&+C\min (x^{-1-\theta}, x^{-\theta})+\min \left(Ct^{-\theta}, \frac {Ct^{-\theta}} {x}\right)\,\,\,\hbox{since}\\
&\int  _{ 2x }^ts^{-3}(t-s)^{-\theta}ds=t^{-2-\theta}\int  _{ 2x/t }^1 \rho ^{-3}(1-\rho )^{-\theta}d\rho\le C t^{-2-\theta}\frac {t^2} {x^2}=\frac {Ct^{-\theta}} {x^2}\\
&\int  _{ 2x }^ts^{-2}(t-s)^{-\theta}ds=t^{-1-\theta}\int  _{ 2x/t }^1 \rho ^{-2}(1-\rho )^{-\theta}d\rho\le C t^{-1-\theta}\frac {t} {x}=\frac {Ct^{-\theta}} {x}
\end{align*}
On the other hand,
\begin{align*}
&\int _0^t\left(1+\xi (t-s, x)\right)ds=\left(t+\frac {\min(t^4, x^4)} {4x^4} \right)+\log\left( \frac {t} {x}\right)_+\,\,\,\,\forall x>0,\\
&\int _0^t s^{-\theta}  \xi (t-s, x)ds=\frac {6x^{-4}t^{4-\theta}} {24-50 \theta+35\theta^2-10\theta ^3+\theta^4},\,\,\forall x\ge t>0\\
&\int _0^t s^{-\theta} \xi (t-s, x)ds=t^{-\theta}\beta \left( 1-\frac {x} {t}, 1-\theta, 0\right)+r(t, x),\,\,\,\forall t>x>0\
\end{align*}
\begin{align*}
&r(t, x)=\frac {x^{-4}(t(t-x))^{-\theta}} {(4-\theta)(3-\theta)(2-\theta)(1-\theta)} \times \\
&\hskip 1cm \times \Big(-6t^{4+\theta}+6t^4 (t-x)^\theta+6\theta t^{3+\theta}x+3(1-\theta)\theta t^{2+\theta}x^2+\\
&\hskip 1.5cm \quad+(2-\theta)(1-\theta)\theta t^{1+\theta}x^3+(3-\theta)(2-\theta)(1-\theta)t^\theta x^4 \Big),\,\,\forall t>x>0\\
&=(t-x)^{-\theta} \left(C _{ 1, \theta}\left( \frac {t} {x}\right)^{4}+C _{ 3, \theta }\left( \frac {t} {x}\right)^3
+C _{ 4, \theta }\left( \frac {t} {x}\right)^2+C _{ 4, \theta }\left( \frac {t} {x}\right)+C _{ 5, \theta } \right)+\\
&+C_2x^{-4}t^{4-\theta}
\end{align*}
\begin{align*}
=&t^{-\theta}\left(1-\frac {x} {t}\right)^{-\theta} \left(C _{ 1, \theta}\left( \frac {t} {x}\right)^{4}+C _{ 3, \theta }\left( \frac {t} {x}\right)^3
+C _{ 4, \theta }\left( \frac {t} {x}\right)^2+C _{ 4, \theta }\left( \frac {t} {x}\right)+C _{ 5, \theta } \right)+\\
&+C_2x^{-4}t^{4-\theta}=\frac {1} {4}\left(24-50\theta+35\theta^2-10\theta^3+\theta^4 \right)t^{-\theta}+\\
&+\theta t^{-\theta} \left( 24-50\theta+35\theta^2-10\theta^3+\theta^4\right) \frac {x} {5t}+t^{-\theta}\mathcal O\left(\frac {x} {t} \right)^2,\,\,\frac {t} {x} \to \infty\\
=& t^{-\theta} \left( 24-50\theta+35\theta^2-10\theta^3+\theta^4\right)\left(\frac {1} {4}-\frac {\theta x} {5t} \right)+t^{-\theta}\mathcal O\left(\frac {x} {t} \right)^2,\,\,\frac {t} {x} \to \infty.
\end{align*}
\vskip -0.5cm 
\end{proof}
\subsection{The system (\ref{S3E23590B}), (\ref{PBB}). }
We briefly present in this Section the results for the non linear system (\ref{S3E23590B}), (\ref{PBB}),
\begin{align*}
&\frac {\partial v} {\partial t }(t, x)=\tilde p_c(t)\int _0^\infty (v(t , y)-v(t , x)) M(x, y) dy\\
&\frac {\partial \tilde p_c} {\partial t}(t)=-\tilde p_c(t)\int  _{0}^\infty I_3\left(n_0+n_0(1+n_0)x^2 v(t, x)\right)x^2dx.
\end{align*}
The following change of time variable, similar to (\ref{S2NewTime}),
\begin{align}
\label{S2NewTimeB}
\tau =\int _0^t\tilde p_c(s)ds,\,\,f(\tau , x)=v(t, x)
\end{align}
leads again to the equation (\ref{S1ERRR10fx}). All the results of  Section 4 are then  available.

The argument goes now as in Section \ref{Sup_c}. First, define the auxiliary function,
\begin{align*}
\tilde m(\tau )= \int _0^\infty &I_3\Big(n_0+n_0(1+n_0)x^2 f(t, x)\Big)x^2dx.
\end{align*}

\subsubsection{The function  $I_3(n_0+n_0(1+n_0)|p|^2 f(t, |p|))$.} 
\label{S1SI3}
The first result is a simpler expression of the term  $I_3(n_0+n_0(1+n_0)|p|^2 f(t, x))$ when $f$ is the solution obtained in Section 4. 
\begin{prop}
\label{S1Ea}
There exists two numerical constants $C_1>0$ and $C_2>0$ such that, if $u$ is the solution of (\ref{S3E23590})--(\ref{S1ERRR10xp})
obtained in Theorem \ref{MainThm} for  $u_0$ satisfying (\ref{S4Etheta}) and $\rho >0$,

\begin{align*}
\lim _{ \delta \to 0 }\int _{|p|>\delta } &I_3\Big(n_0+n_0(1+n_0)x^2 f(t, x)\Big)x^2dx=-C_1(1+a(t))^2+\\
&+C_2\int _0^\infty x^3  \Big(n_0+n_0(1+n_0)x^2 f(t, x)\Big)dx,\,\,\forall t>0.
\end{align*}
\end{prop}
\begin{proof}
[\upshape\bfseries{Proof of Proposition  \ref{S1Ea}}]  
Following \cite{S, Sv} the argument is more clear and the calculations simpler in the energy variable $\omega$. Define then the function $F=F(\omega )$ as follows
\begin{align*}
&n_0+n_0(1+n_0)x^2 f(t, x)=F(t, \omega ),\,\,\omega =x^2.
\end{align*}
After  suitable time rescaling to absorb a positive constant, 
\begin{align}
&I_3(n_0+n_0(1+n_0)|p|^2 u(t, |p|))=\nonumber\\
&=\frac {2} {\sqrt \omega } \int _0^\omega  \Big( F(\omega -\omega ')F(\omega ')-F(\omega )F(\omega -\omega ')
-F(\omega )F(\omega ')-F(\omega )\Big)d\omega '+ \nonumber\\
&+\frac {4} {\sqrt \omega } \int _\omega ^\infty \Big(F(\omega' )+F(\omega )F(\omega' )+F(\omega '-\omega )F(\omega ')-F(\omega )F(\omega '-\omega ) \Big)d\omega '.
\label{S124.E2}
\end{align}
We arrive now to the point. The integral in the right hand side of (\ref{PB}) is obtained by integration over $(0, \infty)$ of (\ref{S124.E2}) multiplied by $\sqrt \omega $. 
The singular behavior of the function $F(t, \omega )$ as $\omega \to 0$ makes  delicate the estimate of that integral. This was done in detail in \cite{S} under some H\"older conditions on $F$ that are not known to hold true for the function $F$ . Following the notations of \cite{S}, let us define
\begin{align*}
&A _{ \delta  }(F, G)=A_\delta ^1(F, G)+A_\delta ^2(F, g),\\
&A_\delta ^1(F, G)=2 \int _\delta ^\infty \int _0^\omega  \Big( F(\omega -\omega ')G(\omega ')-F(\omega )G(\omega -\omega ')
-F(\omega )G(\omega ')-F(\omega )\Big)d\omega 'd\omega,\\
&A_\delta ^2(F, G)=4 \int _\delta ^\infty \int _\omega ^\infty \Big(F(\omega' )+F(\omega )G(\omega' )+F(\omega '-\omega )G(\omega ')-F(\omega )G(\omega '-\omega ) \Big)d\omega 'd\omega.
\end{align*}
The right hand side of  (\ref{S124.E2}) is then strictly speaking,
\begin{align}
\lim _{ \delta \to 0 } \left(A_\delta (F(t), F(t))+4 \int _\delta ^\infty F(\omega' )(\omega '-\delta )d\omega ' -
2\int _\delta ^\infty \omega  F(\omega )d\omega \right)
\end{align}
If we define, as in \cite{S}, for some $d>0$ small fixed,
\begin{align*}
&g _{ < }(t, \omega )=F(t, \omega )\1 _{ \omega <d };\,\,\,\, g _{ > }(t, \omega )=F(t, \omega )\1 _{ \omega >d }\\
&g _{ < }(t, \omega )=h_0(t, \omega )\1 _{ \omega <d }+h(t, \omega ),\,\,\,h_0(t, \omega )=\frac {1+a(t)} {\omega }.
\end{align*}
Then,
\begin{align*}
A_\delta (F, F)=A_\delta (g_<, g_<)+A_\delta (g_<, g_>)+A_\delta (g_>, g_<)+A_\delta (g_>, g_>)
\end{align*}
As in \cite{S}, by Lebesgue's convergence,
\begin{align*}
\lim _{ \delta \to 0 }A_\delta (g_<, g_>)=\lim _{ \delta \to 0 }A_\delta (g_>, g_<)=\lim _{ \delta \to 0 }A_\delta (g_<, g_>)=0,
\end{align*}
and we are then left with, $A_\delta (g_<, g_<)$. Define now the function $h(t, \omega )$ such that, 
\begin{align*}
h(t, \omega )=g_<(t, \omega )-h_0(t, \omega ),\,\,\,\,\,\, h_0(t, \omega )=\frac {1+a(t)} {\omega }\1 _{ \omega <d }.
\end{align*}
Then, 
 \begin{align*}
h(t, \omega )&=\left(F(t, \omega )-\frac {1+a(t)} {\omega }\right)\1 _{ \omega <d }=\left(n_0+n_0(1+n_0)\omega u(t)-\frac {1+a(t)} {\omega }\right)\1 _{ \omega <d }
\end{align*}
and
\begin{align*}
A_\delta (g_<, g_<)=A_\delta (h, h)+A_\delta (h_0, h)+A_\delta (h, h_0)+A_\delta (h_0, h_0).
\end{align*}
The last term is explicit and gives $A_\delta (h_0, h_0)=-\pi ^2 a(t)^2/3$. 
On the other hand, by (\ref{S6LimzeroE0}),
\begin{align*}
u(t, x)&=a(t)+\mathcal O(\tau (t)^{-1-\theta} \omega^{\frac {1} {2}-\delta }\,),\,\omega \to 0,\,\,t\to 0
\end{align*}
we deduce,
\begin{align*}
h(t, \omega )&=n_0+n_0(1+n_0)\omega  (a(t)+\mathcal O(\tau (t)^{-1-\theta} \sqrt \omega \,))-\frac {1+a(t)} {\omega },\,\omega \to 0,\,\,t\to 0\\
&=\mathcal O\left(\tau (t)^{-1-\theta} \omega^{-\frac {1} {2}+\delta } \right),\,\,\,\omega \to 0,\,t\to 0.
\end{align*}
The function $h(t,\cdot)$ is then integrable for all $t>0$ fixed  and,
\begin{equation*}
\lim _{ \delta \to 0 }A _{ \delta} (h, h)=A _{ 0} (h, h)=0. 
\end{equation*}
We slightly rearrange now the term $A_\delta (h_0, h)+A_\delta (h, h_0)$  as 
\begin{align*}
A_\delta (h_0, h)+A_\delta (h, h_0)=A^1_\delta (h_0, h)+A^2_\delta (h_0, h)+A^1_\delta (h, h_0)+A^2_\delta (h, h_0)\\
=\left(A^1_\delta (h_0, h)+A^1_\delta (h, h_0)\right)+\left(A^2_\delta (h_0, h)+A^2_\delta (h, h_0)\right).
\end{align*}
with
\begin{align}
\label{S5EAD1}
&A^1_\delta (h_0, h)+A^1_\delta (h, h_0)=2\int _\delta ^\infty \int _0^\omega  \Big((h(\omega -\omega ')-h(\omega ))h_0(\omega ')+\nonumber\\
&+(h(\omega ')-h(\omega ))h_0(\omega -\omega' )-(h(\omega -w')+h(\omega' ))h_0(\omega) \Big)d\omega 'd\omega
\end{align}
and
\begin{align}
\label{S5EAD2}
&A^2_\delta (h_0, h)+A^2_\delta (h, h_0)=4\int _\delta ^\infty \int _\omega^\infty  \Big((h(\omega ')-h(\omega'-\omega  ))h_0(\omega)+\nonumber\\
&+(h(\omega)-h(\omega'-\omega  ))h_0(\omega' )+(h(\omega')-h(\omega ))h_0(\omega) \Big)d\omega 'd\omega.
\end{align}
The argument still  follows as in \cite{S}, even if our function $h$ satisfies slightly different conditions than  (A.13) and (A.14). Indeed we
claim that here also, the two functions under the integral signs in (\ref{S5EAD1}) and (\ref{S5EAD2}) are integrable on $(0, \infty)$. The only delicate region is where bot $\omega$ and $|\omega '-\omega |$ are arbitrarily small.

 Consider for example the term $(h(t, \omega)-h(t, \omega'))h_0(t, \omega -\omega' )$ for $\omega '\in (0, \omega )$ in (\ref{S5EAD1}). Rewrite first, with $u(t)=u(t, \sqrt \omega )$ and $u'(t)=u(t, \sqrt {\omega '})$
\begin{align}
&h(\omega)-h(\omega') =n_0+n_0(1+n_0)\omega u(t)-n'_0-n'_0(1+n'_0)\omega' u'(t)-\label{S1.24E3PX} \\
& -\frac {1+a(t)} {\omega }+\frac {1+a(t)} {\omega' }=\varphi (\omega )-\varphi (w')+\psi (\omega )-\psi (\omega ')\nonumber\\
&\varphi (\omega )=\left(n_0-\frac {1} {\omega }\right)\label{S1.24E3fi}\\
&\psi (\omega )=\left(n_0(1+n_0)\omega u(t)-\frac {a(t)} {\omega }\right)\label{S124E3}
\end{align}
It is now immediate that the function $\varphi $ is globally Lipschitz on $[0, \infty)$.
On the other hand, if the difference $\psi (\omega )-\psi (\omega ')$ is written as 
\begin{align}
\psi (\omega )-&\psi (\omega ')=n_0(1+n_0)\omega (u(t)-u'(t))+\nonumber\\
&+\left(n_0(1+n_0)\omega-n'_0(1+n'_0)\omega'\right)u'(t)-\frac {a(t)} {\omega }+\frac {a(t)} {\omega '} \label{S124E3psi}
\end{align}
the two terms in the right hand side of (\ref{S124E3psi}) are estimated as follows. In the first one, for $\omega $ and $\omega $ small 
\begin{align*}
\left|n_0(1+n_0)\omega (u(t)-u'(t))\right|\le \frac {C|u(t, \sqrt \omega )-u(t, \sqrt {\omega '}|} {\omega }
\end{align*}
By Lemma \ref{S5LCX1}, for each $t>0$ fixed, if $x$ and $x'$ are small enough,
\begin{align*}
|f(t, x )-f(t,x ')|\le Ct^{-2}|x-x'|
\end{align*}
then, if $\omega $ is small enough and $\omega '\in (0, \omega )$,
\begin{align*}
|u(t, \sqrt{\omega } )-u(t,\sqrt{\omega '})|\le Ct^{-2}(\sqrt{\omega }-\sqrt{\omega '})
\end{align*}
and
\begin{align}
\left|n_0(1+n_0)\omega (u(t)-u'(t))\right|\le \frac {C(\sqrt{\omega }-\sqrt{\omega '})} {t^2\omega }
\le   \frac {C\sqrt{\omega-\omega ' }} {t^2\omega } \label{S124E3psi5}
\end{align}
The second term in the right hand side of (\ref{S124E3psi}), is written,
\begin{align}
&\left(n_0(1+n_0)\omega-n'_0(1+n'_0)\omega'\right)u'(t)-\frac {a(t)} {\omega }+\frac {a(t)} {\omega '}=\nonumber\\
&=\Big(n_0(1+n_0)\omega-n'_0(1+n'_0)\omega'\Big)\left(a(t)+\mathcal O\left(\tau (t)^{-1-\theta}\omega'^{\frac {1} {2}-\delta }\right) \right)-\frac {a(t)} {\omega }+\frac {a(t)} {\omega '}\nonumber\\
&=\Bigg(n_0(1+n_0)\omega-\frac {1} {\omega }-n'_0(1+n'_0)\omega'+\frac {1} {\omega' }\Bigg)a(t)+\nonumber\\
&+\Big(n_0(1+n_0)\omega-n'_0(1+n'_0)\omega'\Big)\mathcal O\left(\tau (t)^{-1-\theta}\omega'^{\frac {1} {2}-\delta }\right).
\label{S124E3psi2}
\end{align}
The function $\omega \mapsto n_0(1+n_0)\omega-\omega ^{-1}$ is Lipschitz and then, the factor of $a(t)$ in the right hand side of (\ref{S124E3psi2}) yields,
\begin{align}
\Bigg|n_0(1+n_0)\omega-\frac {1} {\omega }-n'_0(1+n'_0)\omega'+\frac {1} {\omega' }\Bigg|\le C(\omega -\omega ').
\label{S124E3psi3}
\end{align}
For the last term in the right hand side of (\ref{S124E3psi2})we notice that,
\begin{align*}
\frac {d } {d \omega }(n_0(1+n_0)\omega)&=-\frac {1} {4}\Big(-1+\omega \hbox{coth} (\omega /2) \Big)
\left( \hbox{csch}(\omega /2)\right)^2=-\frac {1} {\omega ^2}+\mathcal O(1),\,\,\omega \to 0.
\end{align*}
from where, if we call $g(\omega )\equiv n_0(1+n_0)\omega $, for $\omega $ and $\omega '$ small,
\begin{align}
\Big|n_0(1+n_0)\omega-n'_0(1+n'_0)\omega'\Big|\le (\omega -\omega ')\int _0^1\Big| \frac {d g} {d \omega }(r \omega +(1-r)\omega ')\Big|dr \nonumber \\
\le C(\omega -\omega ')\int _0^1(r \omega +(1-r)\omega ')^{-2}dr\le  \frac {C|\omega -\omega '|} {\omega \omega '}.
\label{S124E3psi4}
\end{align}
It follows from (\ref{S1.24E3PX} ), (\ref{S124E3psi})-- (\ref{S124E3psi4}) that for $d$ small, $\omega \in (0, d)$ and $\omega '\in (0, \omega )$
\begin{align*}
|(h(t, \omega)-h(t, \omega'))h_0(t, \omega -\omega' )|&\le \frac {C} {(\omega -\omega ')}\Bigg((\omega -\omega ')+\frac {\sqrt{\omega-\omega ' }} {t^2\omega } + \frac {(\omega -\omega ')} {\omega\, \omega '^{\frac {1} {2}+\delta }}\Bigg)\\
&\le C\Bigg( 1+\frac {1} {t^2\omega (\omega -\omega ')^{1/2}}+\frac {1} {\omega\, \omega'^{\frac {1} {2}+\delta }} 
\Bigg).
\end{align*}
Moreover, if  $\omega '\in (0, \omega )$, were such that $\omega >d+\omega '$ or $\omega '>d$ it would follow that 
$(h(t, \omega)-h(t,\omega'))h_0(t, \omega -\omega' )=0$. Therefore,
\begin{align*}
&\int _0^\infty\int _0^\omega |(h(t, \omega)-h(t,\omega'))h_0(t, \omega -\omega' )|d\omega 'd\omega =\\
&=\int _0^{2d}\int _0^\omega |(h(t, \omega)-h(t,\omega'))h_0(t, \omega -\omega' )|d\omega 'd\omega \\
&\le C\int _0^{2d}\int _0^\omega \Bigg( 1+\frac {1} {t^2\omega (\omega -\omega ')^{1/2}}+\frac {1} {\omega\, \omega'^{\frac {1} {2}+\delta }} \Bigg)d\omega 'd\omega\to 0,\,\,\hbox{as}\,\,d\to 0.
\end{align*}
Arguing in the same way for all the other terms in  (\ref{S5EAD1}) and (\ref{S5EAD2}) it follows that,
\begin{align*}
&\lim _{ \delta \to 0 }A_\delta (h_0, h)+A_\delta (h, h_0)=0\\
\hskip -2.3cm \hbox{and then}\hskip 3cm
&\lim _{ \delta \to 0 }A_\delta (F, F)=-\frac {\pi ^2} {3}a(t)^2+
\int _0^\infty \frac {f(t, x)\,x^5 } {(\sinh(x^2/2))^2}dx.
\end{align*}
Proposition follows when the initial  time rescaling is inverted.
\end{proof}
By Proposition \ref{S1Ea}, $\tilde m$ is well defined and finite for all $\tau >0$. Let us then define
\begin{align*}
&\widetilde {\mathcal M}(\tau )=\int _0^\tau \tilde m(\sigma )d\sigma;\,\,\,\,\tilde q_c(\tau )=q_c(0)e^{-\widetilde {\mathcal M}(\tau )},\,\,\forall \tau >0.
\end{align*}
\begin{prop}
\label{S5EmnB}
There exists $T_*\in (0, \infty]$, such that for all $t\in (0, T_*)$ there exists a unique $\tau >0$ such that
\begin{equation}
t=\int _0^\tau \frac {d\sigma } {q_c(\sigma )},\,\,\forall \tau >0.\label{S5CPt1InvB}
\end{equation}
\end{prop}
\begin{proof}
By Proposition \ref{S1Ea}, $|\mathcal M(\tau )|<\infty$ for all $\tau >0$ and then $q_c(\tau )\in (0, \infty)$ for all $\tau >0$ and the integral in the right hand side of (\ref{S5CPt1InvB}) is well defined and convergent. Since $q_c(t)>0$ this integral is a monotone increasing function of $\tau $. The value of $T^*$ is then given by
\begin{align*}
T^*=\int _0^\infty \frac {d\sigma } {q_c(\sigma )}
\end{align*}
\vskip -1cm 
\end{proof}
\begin{rem}
The function $\widetilde{\mathcal M}$ can not be estimated as $\mathcal M$ in Proposition \ref{S5Emn}, using the conservation of the total number of particles.By  Proposition \ref{S1Ea} and 
 Proposition \ref{S6Limzero}, 
\begin{align}
\label{S7Etiw}
\tilde m(\sigma )=-C_1(1+b(\sigma ))^2+C_2\int _0^\infty (n_0+n_0(1+n_0)x^2f(\sigma , x))x^3dx.
\end{align}
By Corollary \ref{SeEC2f2} and Corollary \ref{SeEC2f3},
\begin{align*}
\lim _{ t\to \infty }\int _0^\infty (n_0+n_0(1+n_0)x^2f(\sigma , x))x^3dx=
C_*\int _0^\infty (n_0+n_0(1+n_0)x^2x^3dx
\end{align*}
from where, for some constant $C>0$,
\begin{align*}
C_2\int _0^\tau \int _0^\infty (n_0+n_0(1+n_0)f(\sigma , x))x^5dx\le C\tau ,\,\,\forall \tau >0.
\end{align*}
But, the first term in the right hand side of (\ref{S7Etiw}) may only be estimated  using (\ref{S6LimzeroE25}),
\begin{align*}
\int _0^\tau (1+b(\sigma ))^2d\sigma &\le \int _0^\tau \left(1+C|||f_0|||_\theta\left(\sigma ^{-\theta}+\sigma  \right)^2 \right)d\sigma\le C\Big(\tau +|||f_0|||_\theta\left(\tau^{1-2\theta} +\tau ^3\right) \Big).
\end{align*}
It then follows
\begin{align}
\tilde m(\tau )\ge  C\Bigg(-\Big(\tau +|||f_0|||_\theta\left(\tau^{1-2\theta} +\tau ^3\right) \Big)+\tau 
\Bigg),
\end{align}
but this estimate is not sufficient to deduce that $T^*=\infty$.
\end{rem}

 \end{document}